\documentclass[11pt]{amsart}
\usepackage{amsmath, amsthm, amssymb}
\usepackage{amsmath,amscd}
\usepackage{mathabx}
\usepackage{hyperref}
\usepackage{mathrsfs}
\usepackage{xcolor, changepage} 
\usepackage{graphicx}
\usepackage{titletoc}
\usepackage{enumerate}
\usepackage{accents}

\usepackage{tikz}
\usetikzlibrary{matrix,arrows}
\usetikzlibrary{shapes}
\usetikzlibrary{calc}
\usetikzlibrary{arrows}
\usetikzlibrary{decorations.pathreplacing,decorations.markings}
\usepackage[all]{xy}
\usepackage{tikz-cd}

\usepackage{caption}
\usepackage{subcaption}
\usepackage{geometry}
\theoremstyle{plain}
\newtheorem{theorem}{Theorem}
\newtheorem{proposition}[theorem]{Proposition} 
\newtheorem{lemma}[theorem]{Lemma}
\newtheorem{remark}[theorem]{Remark}
\newtheorem{corollary}[theorem]{Corollary}
\newtheorem{definition}[theorem]{Definition}
\newtheorem{example}[theorem]{Example}

\numberwithin{theorem}{section}

\DeclareMathOperator{\Lip}{Lip}
\DeclareMathOperator{\dist}{dist}
\DeclareMathOperator{\diam}{diam}
\DeclareMathOperator{\supp}{supp}
\DeclareMathOperator{\degg}{deg}
\DeclareMathOperator{\width}{width}

\DeclareMathOperator{\divv}{div}

\DeclareMathOperator{\Hess}{Hess}

\numberwithin{equation}{section}

\title[PMT on singular spaces]{Positive mass theorems on singular spaces and some applications}
\author{Shihang He}
\address{Key Laboratory of Pure and Applied Mathematics,
	School of Mathematical Sciences, Peking University, Beijing, 100871, P. R. China
}
\email{hsh0119@pku.edu.cn}
\author{Yuguang Shi}
\address{Key Laboratory of Pure and Applied Mathematics,
	School of Mathematical Sciences, Peking University, Beijing, 100871, P. R. China
}
\email{ygshi@math.pku.edu.cn}

\author{Haobin Yu}
\address{
	School of Mathematics, Hangzhou Normal University, Hangzhou, 311121, P. R. China
}
\email{yhbmath@hznu.edu.cn}
\thanks{S. He, Y. Shi  are funded by the National Key R\&D Program of China Grant 2020YFA0712800 and NSFC12431003. H. Yu is funded by Zhejiang Provincial NSFC No. LY24A010008}

\subjclass[2010]{Primary 53C21, secondary 53C24 }

\begin{document}
\begin{abstract}
		Building upon dimension reduction techniques in the study of positive scalar curvature (PSC) geometry, we prove an effective version of the positive mass theorem (PMT) for asymptotically flat (AF) manifolds of dimension $n\leq 8$ with arbitrary ends (Theorem \ref{thm: 8dim Schoen conj}).  Furthermore, we prove two "free of singularity type rigidity theorems" for minimal hypersurfaces with isolated singularities (Theorem \ref{prop: rigidity for minimal surface} and Theorem \ref{thm: georch free of singularity}). Our approach bypasses the need for N. Smale's regularity theorem for minimal hypersurfaces in generic $8$-dimensional compact manifolds, providing a direct derivation of the PMT for such AF manifolds (Theorem \ref{thm: pmt8dim}). 
        
        Motivated by these developments, we further establish  PMT for singular spaces (Theorems \ref{thm:pmt with singularity4}). These results assume only that the scalar curvature is non-negative in a strong spectral sense, a condition naturally aligned with the stability of minimal hypersurfaces in PSC ambient manifolds.
	\end{abstract}
	\maketitle
	\tableofcontents	
	
	\section{Introduction}

	\subsection{Positive Mass Theorem and Geroch's Conjecture in Dimensions Not Exceeding  $8$}
	
    $\quad$
    
	In \cite{HSY24}, we showed that the existence and behavior of a certain class of area-minimizing hypersurfaces within an exhaustive sequence of coordinate cylinders \footnote{A coordinate cylinder of radius $R$ in an AF end is given by $C_R=\{ (x_1,x_2,\dots,x_n,z)\in \mathbf{R}^{n+1}, x_1^2+x_2^2+\dots+x_n^2\le R^2\}.$} in an asymptotically flat (AF) manifold of dimension less than or equal to $7$ and with non-negative scalar curvature is heavily dependent on the  mass of the AF manifold (cf. Theorem 1.3 in \cite{HSY24}).	Such a phenomenon can be regarded as an effective version of positive mass theorem (PMT) of the AF manifold, by which one get  PMT directly.   In this paper, we continue this investigation in higher dimensional  AF manifolds $(M^{n+1},g)$ with arbitrary ends. Let $E$ be an AF end of $M^{n+1}$, $U_i$, $i=1,2$ be  two neighborhoods of $E$  with $ E\subset U_1\subset U_2 $ and  for $1\leq i\leq 2$, each $\overline{U_i\setminus E}$ being compact and $\partial U_i$ being smooth. Then we have: 	
	
	\begin{theorem}\label{thm: 8dim Schoen conj}
		Let $(M^{n+1},g)$ be an AF manifold with arbitrary ends and  an AF end  $E$ of asymptotic order $\tau>max \{\frac{n}{2}, n-2\}$ $(n+1\le 8)$. Suppose the scalar  curvature $R_g\geq 0 $ on $M^{n+1}$,  and $R_g>0$ in  $U_2\setminus U_1$  and  $m(M, g, E) \neq 0$, then one of the following happens:
		\begin{itemize}
			\item There exists $R_0>0$ such that for all $R>R_0$ there exists no element $\Sigma_R$, which minimizes the volume in the  $\mathcal{F}_R$ (see \eqref{eq: 79}  for detailed definitions);
			\item For any sequence $\{R_i\}$ tending to infinity such that there exists a sequence of hypersurfaces $\Sigma_{R_i}$  minimizing the volume in $\mathcal{F}_{R_i}$, we have that $\Sigma_{R_i}$ drifts  to the infinity,i.e. for any compact set $\Omega \subset M$, 	$\Sigma_{R_i}\cap \Omega= \emptyset$ for sufficiently large $i$.
		\end{itemize}
	\end{theorem}
	\begin{remark}
			 \item If $(M^{n+1},g)$  admits  only AF ends, then the conclusion of Theorem \ref{thm: 8dim Schoen conj} remains true under a weaker assumption that $R_g\ge 0$ everywhere. In this situation,  Theorem \ref{thm: 8dim Schoen conj} can be deduced  from Theorem 1.6  of \cite{CCE16}  for AF ends of asymptotic order $\tau>\frac{1}{2}$ when $n=2$ and Theorem 2 of \cite{Carlotto16} for  asymptotically Schwarzschild manifolds  when $3 \leq n \leq 6$. 
	\end{remark}

	Theorem \ref{thm: 8dim Schoen conj} reflects certain global effects of the mass for AF manifolds with nonnegative scalar curvature and   arbitrary ends. We would like to emphasize that, even when all ends are AF, an effective form of the positive mass theorem for AF manifolds (as stated in Theorem \ref{thm: 8dim Schoen conj}) appears to be new in dimensions greater than 7. The sequence of minimal hypersurfaces provides a quantitative characterization of the positivity of the mass, thus a key point lies in the understanding of the interaction between ambient scalar curvature geometry and minimal hypersurfaces with singularities, where the latter naturally occurs in higher dimensions.

	 The rigidity of stable minimal hypersurfaces in ambient spaces with nonnegative scalar curvature have a rich history, for instance see (\cite{FcS1980}\cite{CG00}\cite{BBN10}\cite{Carlotto16}\cite{EK23} etc.) Theorem \ref{prop: rigidity for minimal surface} below continues the investigation of this problem in dimension 8, where a priori we allow the minimal hypersurface to possess singularities. However, we find that singularities cannot occur in the scalar-extremal setting of Theorem \ref{prop: rigidity for minimal surface}. We interpret Theorem \ref{prop: rigidity for minimal surface} as a "free of singularity type rigidity theorem"	 which plays an important role in the proof of above Theorem \ref{thm: 8dim Schoen conj}. Namely, 
	\begin{theorem}\label{prop: rigidity for minimal surface}
		Let $(M^{n+1},g)$ be an AF manifold with arbitrary ends, an AF end  $E$ of asymptotic order  $\tau>n-2$  and $R_g\ge 0$. Let $\Sigma^{n}\subset M^{n+1}$ be an area-minimizing boundary in $M^{n+1}$ with isolated singular set $\mathcal{S}$. Assume $\Sigma$ is  strongly stable (see Definition \ref{defn: strongly stable hypersurface} below), and one of the following holds:
		
		(a) $n+1\le 8$
		
		(b) $\Sigma\backslash\mathcal{S}$ is spin,  and the tangent cone at each point in $\mathcal{S}$ has isolated singularity.
		
		$\quad$
		
		(1) If $\Sigma\backslash E$ is compact, then we have $\mathcal{S} = \emptyset$, and $\Sigma$ is isometric to $\mathbf{R}^{n}$, with $R_g = |A|^2 = Ric(\nu,\nu) = 0$ along $\Sigma$.
		
		(2) If there exists $U_1,U_2\subset M$ with $E\subset U_1\subset U_2$, such that $U_i\backslash E$ is compact ($i = 1,2$) and  $R_g> 0$ on $U_2\backslash U_1$, then $\Sigma\subset U_1$, and the same conclusion in (1) holds. 
	\end{theorem}

    \begin{figure}
            \centering
            \includegraphics[width=15cm]{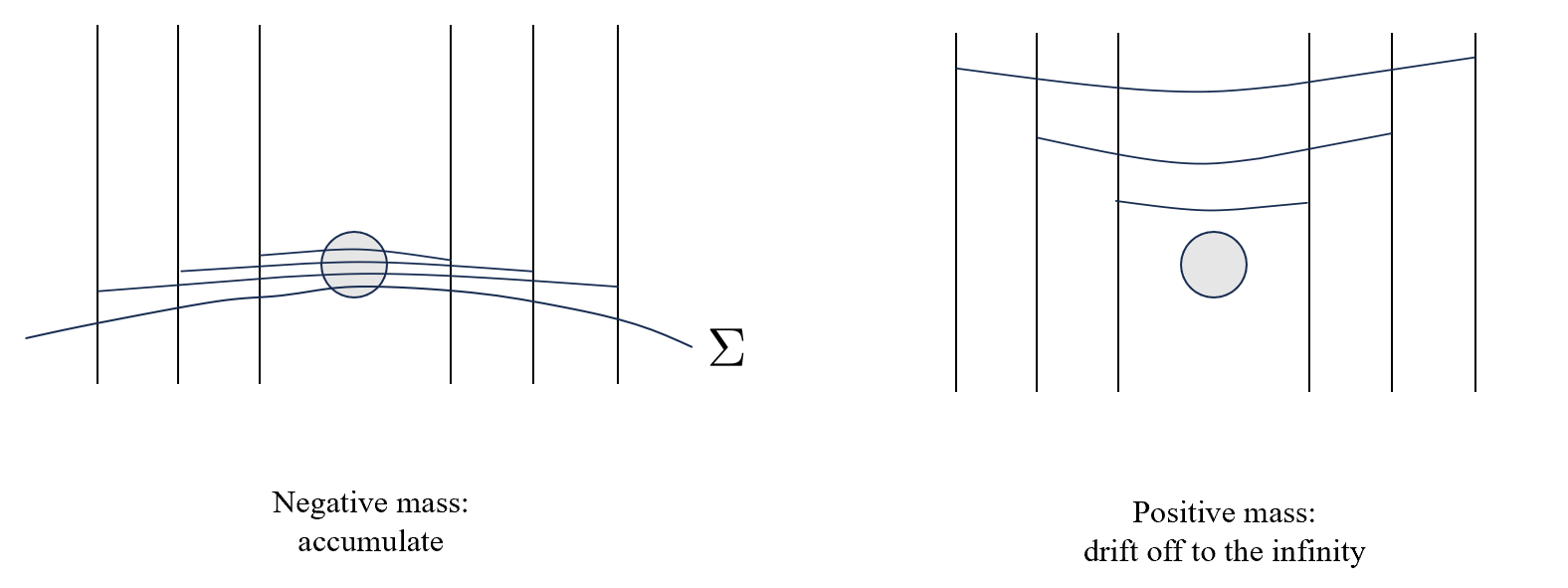}
            \caption{An illustration of Theorem \ref{thm: 8dim Schoen conj}: Vertical pairs of lines represent coordinate cylinders, while horizontal curves represent free boundary minimizing hypersurfaces.}
            \label{c2}
        \end{figure}
	
	 In his Four Lectures \cite[Section 3.7.1]{Gro23}, Gromov remarked that singularities must enhance the power of minimal hypersurfaces and stable $\mu$-bubbles, since the large curvatures add to the positivity of the second variation.  Gromov's observation suggests a possibility to use the minimal hypersurface method to study scalar curvature problems in higher dimensions without any smoothing procedure. Theorem \ref{prop: rigidity for minimal surface} and  Theorem \ref{thm: georch free of singularity} below  consolidate this Gromov's observation.
	
	 As a direct corollary of Theorem \ref{thm: 8dim Schoen conj}, we obtain the following positive mass theorem for AF manifolds $(M^{n+1},g)$ with arbitrary ends, where $n\leq 7$.

	\begin{theorem}\label{thm: pmt8dim}
		Let $(M^{n+1}, g)$ be a smooth AF manifold  with  arbitrary ends and non-negative scalar curvature. If $2\leq n\leq 7$, then the mass of $(M^n, g)$ is non-negative, the equality holds if and only if 	$(M^{n+1}, g)$ is isometric to $\mathbf{R}^{n+1}$.
	\end{theorem}

 In \cite{LUY21}, Lesourd-Unger-Yau established a shielded version of the positive mass theorem. The second statement of our Theorem \ref{prop: rigidity for minimal surface} is largely influenced by their work, and we refer to it as the \textit{shielding principle} for strongly stable area-minimizing hypersurfaces in AF manifolds with arbitrary ends. As illustrated by Theorem \ref{thm: 8dim Schoen conj}, the positivity of the mass is evaluated by sequences of (possibly singular) free boundary minimizing hypersurfaces, we note that minimal hypersurface method is enough to obtain the positive mass theorem for AF manifolds with arbitrary ends. Thus one of the novelty of Theorem \ref{thm: pmt8dim} lies in that it can be obtained independent of the $\mu$-bubble approach. On the other hand, manifolds with arbitrary ends arise naturally as the blow up models for singular spaces. A spectral version of the positive mass theorem for AF manifolds with arbitrary ends, Theorem \ref{thm: PMT arbitrary end spectral}, will play a central role in our establishment of the main results of this paper. This is our another motivation for considering manifolds with arbitrary ends.
	
	 Compared with
    Theorem \ref{prop: rigidity for minimal surface}, we obtain another "free of singularity type rigidity theorem", concerning area-minimizing hypersurfaces in closed manifolds.  
		\begin{theorem}\label{thm: georch free of singularity}
			Let $(M^{n+1},g)$ be a closed Riemannian manifold and let $\Sigma^n\subset M^{n+1}$ be an area-minimizing hypersurface with isolated singular set $\mathcal{S}$. Assume that there exists a non-zero degree map from $\Sigma$ to an enlargeable manifold $X$ in the sense of Definition B.3.
			
			If $R_g\ge 0$ along $\Sigma$ and one of the following holds:
			
			\begin{enumerate}
				\item $n+1 = 8$;
				\item $\Sigma\backslash\mathcal{S}$ is spin, and the tangent cone at each point in $\mathcal{S}$ has isolated singularity.
			\end{enumerate}
			Then $\mathcal{S} = \emptyset$, $\Sigma$ is intrinsically flat, and $R_g = Ric(\nu,\nu) = |A|^2_\Sigma = 0$ along $\Sigma$.
	\end{theorem}

	As a direct corollary of Theorem \ref{thm: georch free of singularity}, we have the following topological obstruction result for PSC metric without using N.Smale's regularity theorem for minimal hypersurfaces in  compact $8$-dimensional  manifolds with generic metric  \cite{Smale93}.

	\begin{corollary}\label{cor:8dim georch}
		Let $M^n$ $(n\le 8)$ be a compact manifold and assume there exists a non-zero degree map $f: M^n\longrightarrow \mathbf{T}^n$, then $M^n$ admits no PSC metric.
	\end{corollary}
	
	\subsection{Positive Mass Theorems on General Singular Spaces}
	
        $\quad$
    
	In recent years, there has been growing interest in the nonexistence of metrics with positive scalar curvature (PSC) on singular spaces and  the corresponding positive mass theorems. In \cite{LM2019},\cite{Kaz24},\cite{WX2024}, the authors established Georch type conjectures for uniform Euclidean metrics which are smooth outside a codimension $3$ submanifold. In \cite{DaSW2024} and \cite{DaSW2024}, the authors considered related problems for manifolds with isolated conical singularities. \cite{DWWW2024} established the global rigidity in general settings using methods from the RCD theory. \cite{JSZ22} and \cite{CLZ22} established positive mass theorems for $C^0\cap W^{1,p}$ metrics. More recently, \cite{CFZ24} constructed closed manifolds that admit no PSC metric, but admit PSC metrics with isolated singularities. 
    
    On one hand, as previously discussed, a critical step in the proofs of Theorem \ref{thm: 8dim Schoen conj}, Theorem \ref{thm: pmt8dim}, and Theorem \ref{thm: georch free of singularity} is the establishment of positive mass theorems for asymptotically flat (AF) manifolds admitting singular sets. On the other hand, singular minimizing hypersurfaces emerge naturally in the dimensional reduction approach to studying the geometry of PSC manifolds. Motivated by these considerations, we prove positive mass theorems for a class of singular spaces, which are of independent mathematical interest. More specifically,  let $(M^n,d, \mu)$ be a metric  measure space with a Lebesgue measure $\mu$ and satisfy the following assumptions which is named \textit{almost-manifold conditions.}
	
	\begin{enumerate}
		\item  There is a locally compact subset $\mathcal{S}\subset M$ so that $M^n\setminus \mathcal{S}$ is a $n$-dimensional open smooth manifold, and $(M^n\setminus \mathcal{S}, d, \mu)$ is induced by a Riemannian metric $g$. Here and in the sequel, $\mathcal{S}$ is called the singular set of $M^n$;
		
		\item There holds $\mu(\mathcal{S}) = 0$. Moreover, for each Borel subset $A$ of $M\backslash\mathcal{S}$, it holds $\mu(A) = \mathcal{H}_g^n(A)$ where $\mathcal{H}_g^n(A)$ denotes the Hausdorff measure of $A$ with respect to the Riemannian metric $g$ ;
		
		\item  There is a Lipschitz function $r(\cdot, \cdot)$  on $M^n \times M^n$, for any $x,p\in M^n$  it holds $0\leq r(x,p)\leq d(x,p)$, the sub-level set $\{x\in M: r(x,p)\leq t\}$ is compact for any $t\geq 0$ and fixed $p\in M^n $. $r(x,p)= 0$ if and only if $x=p$. In addition, we assume that $r(x,y)+r(y,z)\geq r(x,z)$  for any $x,y,z\in M$;
		
		\item $(M^n,d, \mu)$ supports the Poincare inequality, i.e. there exist constants $r_0>0$, $ \gamma>0$, $\beta \leq1$ such that for all ball $B(x,r) \subset M $	with $r\leq r_0$ and all $f\in Lip_{loc}(M,d)$ there holds
		$$
		\min _{k \in \mathbf{R}}\left\{\int_{B(x,\beta r)}|f-k|^{n /(n-1)} d\mu\right\}^{(n-1) / n} \leqq 2 \gamma \int_{B(x,r)}|\nabla f| d\mu.;$$
		where $B(p,s):=\{x\in M: r(x,p)<s\}$.\footnote{ We remind the readers of the notational conventions used throughout the paper:

    Let $(M,d,\mu)$ be a metric measure space, with a Lipschitz function $r(\cdot,\cdot)$ defined on $M\times M$, we have:
    \begin{itemize}
        \item $B(p,s) = \{ x, r(x,p)<s\}$, which always denotes an \textit{extrinsic} metric ball;
        \item $\mathcal{B}_s(p) =  \{ x, d(x,p)<s\}$, which always denotes an \textit{intrinsic} metric ball.
    \end{itemize}
    
    Let $(M^n,g)$ be an manifold with an AF end E, then:
    \begin{itemize}
        \item $B_s^n(p)$ (or simply $B_s(p)$) denotes the coordinate ball $\{x: x_1^2+x_2^2+\dots+x_n^2<s^2\}$ in $E\cong \mathbf{R}^n\backslash B^n_1(O)$.
    \end{itemize}};
		\item $(M^n,d, \mu)$  supports the Sobolev inequality, i.e.  for any region $U\subset M$ with compact closure and any $u\in Lip(M)$ with its support set contained in $U$  there holds 
		$$
		\left( \int_{U}|u|^{\frac{n-1}{n-2}}\right)^{\frac{n-2}{n-1}}d\mu\leq \Lambda(U) \int_{U}|\nabla u|d\mu. 
		$$
	\end{enumerate}
	
	For any Riemannian manifold $(M^n,g)$ (may not be complete or compact), we have a metric measure space $(M^n, d, \mu)$ induced by this Riemanian metric $g$.  We say that $(M^n,g)$ satisfies the almost-manifold condition if the completion of $(M^n, d, \mu)$ satisfies this condition.    
	
	\begin{example}
		Let $(M^n,g)$ be a smooth Riemannian manifold and  $\bar g$ be another metric on $M^n$ with $\Lambda^{-1}g\leq  \bar{g}\leq \Lambda g$ for some constant $\Lambda>0$. Suppose $\bar g$ is smooth in $M^n\setminus \mathcal{S}$ with $\mathcal{S}$ being a closed subset of $M^n$. Then  $(M^n,\bar g)$ is a metric measure space  satisfying the almost-manifold conditions.
	\end{example}

	\begin{example}
		Suppose $U $ is an open set in $\mathbf{R}^{n+1}$, $M^n\subset U$ is an oriented boundary of least area in $U$. Let $r(~ ,~)$, $d$, $\mu$ be the distance function in $\mathbf{R}^{n+1}$, induced metric and Lebesgue measure from $\mathbf{R}^{n+1}$ respectively, then by Proposition \ref {prop: Pioncare-Sobolev inequality}
		and Proposition \ref{prop: Sobolev inequality on S}, we know that  $(M^n, d, \mu)$	 is a metric measure space satisfying the  almost-manifold conditions.
	\end{example}
	
	Recently, the notion of scalar curvature lower bound in the {\it spectral sense}  has been gaining increasing importance (see \cite{Gro20},\cite{CL23},\cite{CL2024},\cite{CL24b},\cite{LM2023},\cite{Gr2024} and references therein). Indeed, it is a natural generalization of the conception of non-negative scalar curvature, and is automatically satisfied by a stable minimal hypersurface in a Riemannian manifold with non-negative scalar curvature. For the case of asymptotically flat manifolds, it was remarked in \cite{ZZ00} that non-negative Dirichlet eigenvalue of the conformal Laplacian is not enough to deduce the non-negative mass. On the other hand, \cite{ZZ00} proved a positive mass theorem by assuming the Neumann eigenvalue for the conformal Laplacian to be non-negative on exhaustions of the AF end. In this paper, we aim to study the positive mass theorems for AF manifolds with non-negative scalar curvature in the {\it strong spectral sense}, which naturally arises in area-minimizing hypersurfaces stable under asymptotically constant variations in AF manifolds. The positive mass theorem of this type dates back to Schoen's dimension reduction proof of the positive mass theorem for AF manifolds of dimension no greater than $7$ in \cite{Schoen1989}, see also \cite{Carlotto16},\cite{EK23},\cite{HSY24} for some further applications. To begin with, we introduce the following definitions:

	\begin{definition}\label{defn: strong test function}
		Let $(M^n,g)$ be  a manifold (not necessarily complete) containing an AF end $E$. We say that a locally Lipschitz function $\phi$ is an \textit{asymptotically constant test function} on $M$, if it is supported on a neighborhood $\mathcal{U}$ of $E$ such that $\mathcal{U}\Delta E$ is compact, satisfying $\lim\limits_{|x|\to\infty, x\in E}\phi = 1$ and $\phi-1\in W^{1,2}_{-q}(\mathcal{U})$, for some $q> \frac{n-2}{2}$.
	\end{definition}
	
	\begin{definition}\label{defn: strong spectral}
		Let $(M^n,g)$ be  a manifold (not necessarily complete) containing an AF end $E$. Let $h\in L^1(M)$. We say that $(M^n,g)$ has $\beta$-scalar curvature no less than $h$ in the strong spectral sense, if for any locally Lipschitz function $\phi$ which is either compactly supported or asymptotically constant in the sense of Definition \ref{defn: strong test function}, it holds
		\begin{align}\label{eq: 20}
			\int_{M} |\nabla \phi|^2+\beta R_g\phi^2 d\mu_g \ge \int_M \beta h\phi^2d\mu_g\ \ \ \text{for some}\ \ \ \beta>0.
		\end{align}
		In particular, we say $(M^n,g)$ has $\beta$-scalar curvature non-negative in the strong spectral sense if $h=0$.
	\end{definition}
	\begin{remark}\label{re: NNSC}
		By the definition, we see that  $(M^n,g)$ has $\beta'$-scalar curvature no less than $h$ in the strong spectral sense for any $\beta'\in(0,\beta)$ provided $(M^n,g)$ has $\beta$-scalar curvature no less than $h$ in the same  sense. Furthermore, according to our definition, the classical curvature condition $R_g\ge h$ corresponds to the case of having $\infty$-scalar curvature no less than $h$ in the strong spectral sense.
	\end{remark}
	
	The positive mass theorems on smooth AF manifolds  with non-negative scalar curvature  and with arbitrary ends  were obtained by \cite{CL2024} and \cite{LUY21} etc.
	In this paper, we prove the theorems on certain singular spaces with nonnegative
	scalar curvature in the strong spectral sense. Namely,  	
	
	\begin{theorem}\label{thm:pmt with singularity4}
		Let $(M^n,d, \mu)$, $3\leq n\leq 7$, be a noncompact  and complete metric  measure space with a Lebesgue measure $\mu$,  $\mathcal{S}$ be a closed set of $M^n$ that satisfies
		\begin{enumerate}
			\item $\mathcal{S}=\bigcup^{\infty}_{i=1} \mathcal{S}_i$, where each $\mathcal{S}_i$ is a compact set of $M^n$. Moreover, for each compact set $K$, there exists at most a finite number of $i$ such that $\mathcal{S}_i\cap K\ne\emptyset$; 
			\item  For each $i$ there is some $\Lambda_i>0$ and  $0\leq k_i\leq \frac{n}{2}-1$ such that $\Lambda_i^{-1}\leq \mathcal{H}^{k_i}(\mathcal{S}_i)\leq \Lambda_i$.  Here and in the sequel, $ \mathcal{H}^{k}(S)$ denotes the $k$-dimensional Hausdorff measure of the set $\mathcal{S}$;
			\item   For each $i$ there is a $\delta_i>0$, such that for any $x\in  \mathcal{S}_i$ and all $0<r\leq  \delta_i$ we have $\Lambda_i^{-1}r^{k_i}\leq \mathcal{H}^{k_i}(B(x, r)\cap \mathcal{S}_i)\leq \Lambda_i r^{k_i}$;
			\item For each $i$ there exists an open set $U_i\subset M$ with compact closure that contains $\mathcal{S}_i$ and has smooth boundary $\partial U_i\subset M\backslash\mathcal{S}$;
			\item  For any $p\in\mathcal{S}$, there is a $\delta_0>0$ such that  for any $r<\delta_0$, $\partial B(p,r)$ is connected;
			\item For each compact set $K\subset M$, there exists $\Lambda = \Lambda(K)>0$, such that for all $p\in K$, it holds $\Lambda^{-1}r^n\leq\mathcal{H}^n(B(p, r))\leq \Lambda r^n$ for all $0<r<1$;
		\end{enumerate}
		
		Assume $(M^n\setminus \mathcal{S},d, \mu)$ is induced by a Riemannian metric $g$ and it satisfies the almost-manifold conditions. Suppose $(M^n\setminus \mathcal{S},g) $ has  $\beta$-scalar curvature no less than $h$ in the strong spectral sense for some $\beta\geq\frac{1}{2}$ and arbitrary ends with $E$ being its AF end, then
        
		$(\mathrm{I})$ the  mass of $(M^n,g,E)$ is positive if $h$ is a non-negative  continuous function on $(M^n,d, \mu)$ with $h(p)>0$ for some $p\in M$;
        
		$(\mathrm{II})$ the  mass of  $(M^n,g,E)$ is non-negative provided $h\geq 0$.
        Moreover, if $\beta>\frac{1}{2}$ and the equality holds, then $(M^n\setminus \mathcal{S},g)$ is Ricci flat.
	\end{theorem}
	\begin{figure}
		\centering
		\includegraphics[width = 12cm]{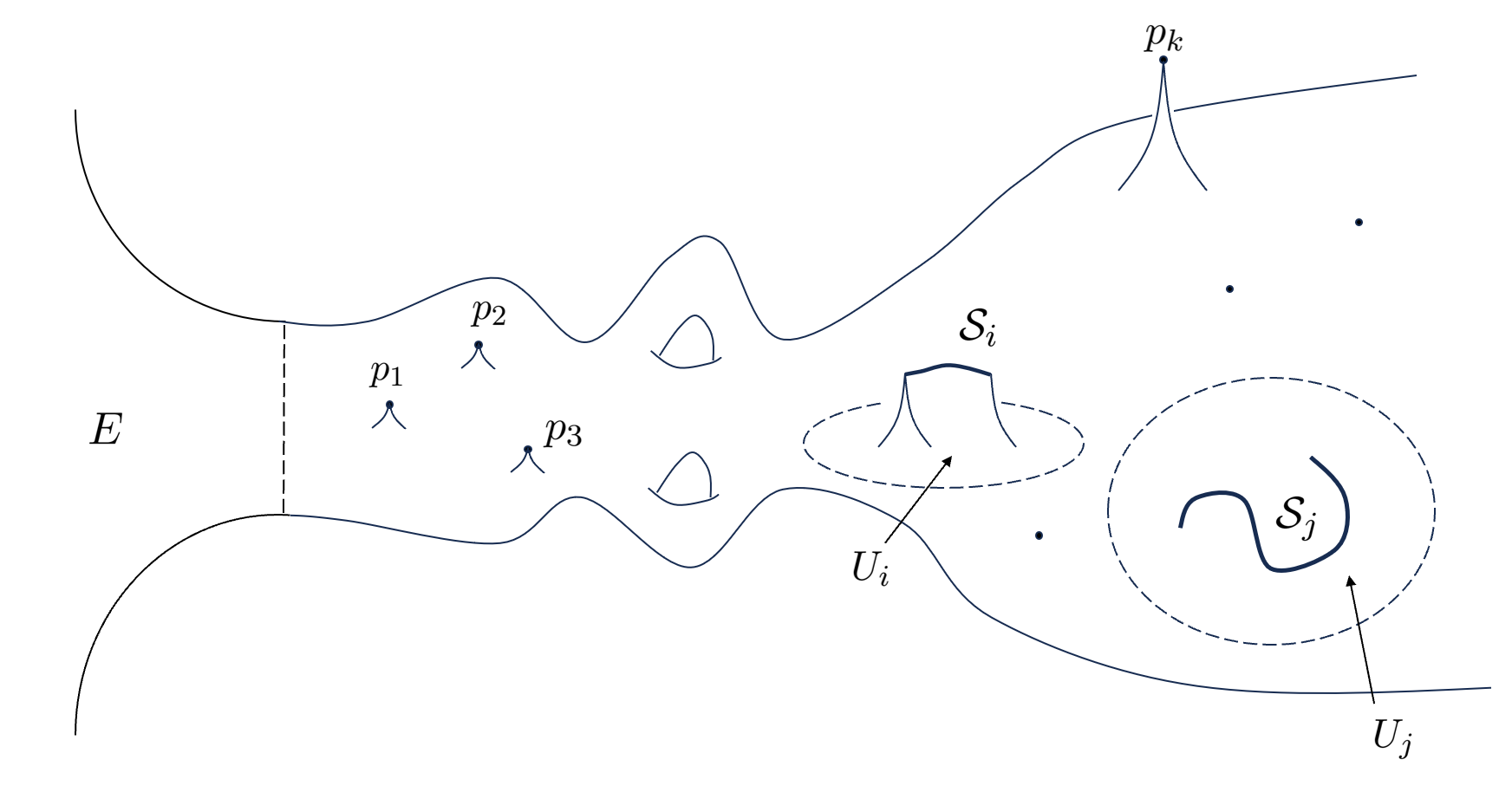}
		\caption{A metric measure space satisfying the almost manifold condition with an AF end and some arbitrary ends}
		\label{f1}
	\end{figure}
	
	 All assumptions in Theorem \ref{thm:pmt with singularity4} are satisfied when $M^7$ is a strongly stable area-minimizing boundary in an $8$-dimensional AF manifolds with nonnegative scalar curvature. The assumption that $h(p)>0$ at a point $p$ arises naturally  if $M^7$ admits a non-empty singular set, see Lemma \ref{lem: |A|>0}.

	\begin{remark}
		In \cite{ST2018}, the second author and Tam constructed an AF manifold $(M,g)$ with an isolated point singularity $p$, where the scalar curvature satisfies $R_g \geq 0$ on $M \setminus \{p\}$ and the  mass is negative (See \cite[Proposition 2.3, Proposition 2.4]{ST2018} for detail). Near the singularity $p$, the metric behaves asymptotically as $d\rho^2 + \rho^{\frac{4}{3}} h_0$, which fails to satisfy the two-sided local volume bound condition  required in the theorems above.
	\end{remark}
    
	Schoen-Yau's work \cite{ScY2019} illustrates that the  positive mass theorem has deep relation to the topology of manifolds with positive scalar curvature. Inspired by this, we prove the following generalized Geroch's conjecture on certain singular spaces, specifically
	
	\begin{theorem}\label{thm:non-existence psc on singular space1}
		Let $(M^n,d, \mu)$  be a compact metric measure space with $3\leq n\leq 7$,  $\mathcal{S}:=\{p_1,\cdots,p_k\}$  and $X^n$ be an enlargeable manifold (see Definition \ref{Defn: enlargeable} below). Suppose there is non-zero degree map $\psi$: $M^n \mapsto X^n$ (See Definition \ref{defn: degree 2} for detail). Then there is no smooth Riemannian  metric $g$ on $M^n\setminus \mathcal{S}$ so that
		\begin{enumerate}
			\item  $(M^n\setminus\mathcal{S},g)$   satisfies the above almost-manifold conditions ; 
			\item  There is a $\delta>0$ such that $\partial B(p,r)$ is connected for any $r<\delta$ and any $p\in\mathcal{S}$;
			\item  There exists a  continuous function $h$ on $(M^n,d, \mu)$ with $h\geq 0$ and $h(p)>0$  for some $p\in M$. For any $\phi\in Lip(M)$  there holds
			
			$$
			\int_{M}(|\nabla \phi|^2+\frac12 R_g\phi^2)d\mu\geq \int_{M}\frac{1}{2}h\phi^2 d\mu\geq 0,
			$$ 
			
			\item   $\Lambda^{-1}r^n\leq\mathcal{H}^n(B(p, r))\leq \Lambda r^n$ for all $r>0$, $p\in M.$
			
		\end{enumerate}
		Moreover, if $(M^n,d, \mu)$ satisfies item(1), (2), (4) above and  there exists a $\beta>\frac{1}{2}$  so that for any $\phi\in Lip(M)$  there holds
		$$
		\int_{M}(|\nabla \phi|^2+\beta R_g\phi^2)d\mu\geq 0,	
		$$ 		
		then $(M^n\setminus\mathcal{S},g)$ is Ricci flat.
	\end{theorem}

	To prove Theorem \ref{thm:pmt with singularity4}, we first establish the following positive mass theorems on smooth AF manifolds with arbitrary ends under the strong spectral non-negative scalar curvature condition which has its own interests. 
	
	\begin{theorem}\label{thm: PMT arbitrary end spectral}
		Let $(M^n,g)$ be an AF manifold with arbitrary ends and $E$  be its AF end 
		of order 
		$\frac{n-2}{2}<\tau\le n-2$ and $3\leq n\leq 7$. Suppose $(M^n, g)$ has $\beta$-scalar curvature no less than $h$ in the strong spectral sense for some $\beta>0$.
		\begin{enumerate}
			\item If $\beta\ge\frac{1}{2}$, $h\ge 0$ everywhere and $h(p)>0$ at some point $p$, then the  mass of $(M,g,E)$ is strictly positive.
			\item If $\beta\ge\frac{1}{2}$, $h\ge 0$ everywhere, then the  mass of $(M,g,E)$ is non-negative. Moreover, if $\beta>\frac{1}{2}$ and the  mass of $(M,g,E)$ is zero, then $(M,g)$ is isometric to $(\mathbf{R}^n,g_{Euc})$.
			\item If $\beta\ge\frac{n-2}{4(n-1)}$, $h\ge 0$ everywhere, and $M$ has only AF ends, then the  mass of $(M,g,E)$ is non-negative. Moreover, if the  mass of $(M,g,E)$ is zero, then $(M,g)$ is isometric to $(\mathbf{R}^n,g_{Euc})$.
		\end{enumerate}
	\end{theorem}
	
	We remark that for case (3) where $M$ has  only AF ends, the corresponding result was established in \cite{Schoen1989} and \cite{Carlotto16}, which we include here for comparison. The key advancement of Theorem \ref{thm: PMT arbitrary end spectral} lies in its allowance for manifolds with arbitrary ends, thereby extending its applicability to a broader range of geometric settings.
	
	\subsubsection{Foliations of Area-minimizing hypersurfaces in General AF manifolds}

    $\quad$
	
	In  \cite{HSY24}, foliations of  area-minimizing hypersurfaces in AF manifolds of dimensions no greater than $7$ were obtained. In this paper we generalize this to general dimensions. Namely, we are able to prove the following:

	\begin{theorem}\label{thm: foliation}
		Let $(M^{n+1},g)$ be an AF manifold of dimension $n\geq 3$ with arbitrary ends and $E$ be an AF end of  order $\tau>\frac{n}{2}$. Suppose $(M^{n+1},g)$ is geometric bounded, i.e. its sectional curvature is bounded, injective radius has a positive uniform lower bound. Then 
		
		(1) For any $t$, there exists an area-minimizing hypersurface $\Sigma_t$  which is asymptotic to the coordinate hyperplane \{z=t\}. Outside some compact set, $\Sigma_t$ can be represented by  some graph $\{(y,u_t(y)):y\in\mathbf{R}^{n} \setminus \mathbf{D}\}$  where  $\mathbf{D}$ is a compact set of $\mathbf{R}^{n}$ ;  
		
		(2) There is a constant $C$ depends only on $k$ and $(M,g)$, such that for any  $1\leq k\leq 3$, $\epsilon>0$, the graph function satisfies 
		$|u_t-t|+|y|^k|D_ku_t(y)|\leq C|y|^{1-\tau-k+\epsilon}$ as $|y|\to +\infty$.
		
		(3) There exists some $T>0$ such that any area-minimizing hypersurface $\Sigma_t$ with $|t|>T$ is smooth and any area-minimizing hypersurface satisfying the condition above is unique. Furthermore, the region above $\Sigma_T$(or below $\Sigma_{-T}$) can be $C^1$ foliated by $\{\Sigma_t\}$.
		
		Moreover, if $M$ is an  AF manifold with some arbitrary ends and without geometric bounded condition, then the conclusions $(1)-(3)$ hold for $t$ with $|t|>T_0$ for some sufficiently large $T_0$.
	\end{theorem}

       The following remark is closely related to Theorem \ref{thm: foliation}, and plays a quite important role in the proof of Theorem \ref{prop: rigidity for minimal surface}, as it justifies the application of Theorem \ref{thm:pmt with singularity4} to area-minimizing hypersurfaces arising naturally in AF manifolds.

	\begin{remark}
		Let $(M^{n+1},g)$ be as in Theorem \ref{thm: foliation} and $\Sigma'_t$ be any area-minimizing hypersurface asymptotic  the coordinate hyperplane \{z=t\} for some $|t|\leq T$. Then $\Sigma'_t$ may contain singular set $\mathcal{S}$. However, we are able to show that $\mathcal{S}$ is contained in a  fixed compact set of $M^{n+1}$ (see Corollary \ref{cor: smooth outside compact set} below) and $\Sigma'_t$ has the same decay property as (2).  More generally, if $M$ is an  AF manifold with some arbitrary ends  but without geometric bounded condition, let $\Sigma'_t$ be an area-minimizing hypersurface in $M$ that is asymptotic to the hyperplane $\{z = t\}$, we can show that $\mathcal{S}\subset \mathbf{K}\cup (M\backslash E)$, where $\mathbf{K}\subset E$ is a compact set. 
		
		In particular,  if $n=7$, by  geometric measure theory, we know that for any compact set $\mathbf{L}\subset M^{n+1}$, $\mathcal{S}\cap \mathbf{L}$ consists of finite points. Thus,  $\Sigma_t$ in Theorem \ref{thm: foliation} satisfies all assumptions of Theorem \ref{thm:pmt with singularity4} when $n=7$.
	\end{remark}
	\begin{remark}
	  It would be interesting to compare the progress in CMC foliations with the area-minimizing foliation constructed here. In both settings the smooth foliation exists in all dimensions in the asymptotic region. For an AF manifold of dimension $n\geq3$ with
positive mass, Eichmair-Koerber \cite{EK24} established the existence of an
asymptotic foliation of $(M, g)$ by stable constant mean curvature spheres based on Lyapunov-Schmidt reduction. See also Eichmair-Metzger \cite{EM13} for the case of isoperimetric foliations in asymptotically Schwarzschild manifolds with positive mass.
	\end{remark}

	\subsection{Outline of Proofs}

$\quad$
    
	 We will blow up the singular set $\mathcal{S}$ of $(M^n, d, \mu)$ by suitable functions. The technique can be traced back to Schoen's work on the Yamabe problem and \cite{LM2019}. Since we have only assumed a spectral nonnegative condition for the scalar curvature in Theorem \ref{thm:pmt with singularity4}, a natural candidate of the blow up function is the Green's function for the conformal Laplacian, as it yields pointwise nonnegative scalar curvature. However, in Theorem \ref{thm:pmt with singularity4}, neither local structure of the singular set $\mathcal{S}$  nor conditions that $M^n$  is a manifold are assumed. Indeed, the scalar curvature blows up rapidly near the singular set, making it hard to control the Green's function for the conformal Laplacian.

     Instead of using the Green's function for the conformal Laplacian, we employ the potential theory to construct some positive harmonic functions which tends to constants at the infinity of AF manifolds and blow up at certain rates near $\mathcal{S}$. Then we perform conformal deformations using these singular harmonic functions, obtaining a  complete AF manifold  with arbitrary ends. This allows us to bypass the challenges involving hard analysis of the scalar curvature term near the singularity. The cost of this approach is that the blown up scalar curvature is not necessarily pointwise non-negative. However, we note that the assumption that $\beta$-scalar curvature no less than $h$ in the strong spectral sense is preserved well under conformal deformations via singular harmonic functions (see Proposition \ref{conformal deformation1} ).  The positive mass theorem asserted in Theorem \ref{thm:pmt with singularity4} then reduces to problems of the type of Theorem \ref{thm: PMT arbitrary end spectral}.
	
	For the proof of Theorem \ref{thm: PMT arbitrary end spectral}, the case that $(M^n,g)$ has exactly one AF end follows from the classical conformal deformation argument as demonstrated in \cite{Schoen1989} and \cite{Carlotto16}. The conformal deformation makes the scalar curvature everywhere non-negative, allowing the application of the  classical positive mass theorem. However, if $M$ has arbitrary ends, then the completeness issue of the conformal deformed metric arises as a central problem. To overcome this difficulty, instead of performing the conformal deformation, we construct a $\mathbf{S}^1$-symmetric asymptotically locally flat (ALF) manifold $(\hat{M}^{n+1},\hat{g})$ by taking warped product with $\mathbf{S}^1$. The condition of $\beta$-scalar curvature no less than $h$ in the strong spectral sense for some  $\beta\ge \frac{1}{2}$ allows us to construct such a suitable warped-product function, and the resulting manifold is ALF with non-negative scalar curvature, allowing the application of the arguments in \cite{CLSZ2021}. 
	
	Here, we remark that the $\mathbf{S}^1$-symmetry of $\hat{M}^{n+1}$ is particularly important since the $\mathbf{S}^1$-symmetric minimal hypersurfaces or $\mu$-bubble can be chosen to be smooth provided the codimension of the $\mathbf{S}^1$-orbit is no greater than $7$ (See \cite[Proposition 3.4]{WY2023}), which is guaranteed by our assumption. Furthermore, we make a crucial observation that under the assumption of non-negative  $\beta$-scalar curvature in the strong spectral sense, the  mass is non-increasing in the process of constructing the warped product, illuminating similar phenomenon in case of performing conformal deformation in \cite{Schoen1989}, see Proposition \ref{lem: mass decay ALF} for detail. These new elements enable us to complete the proof of Theorem \ref{thm: PMT arbitrary end spectral}. 
    
     We further observe that the model space in Theorem \ref{thm: PMT arbitrary end spectral} not only serves as a blow up model for singular minimal hypersurfaces, but also naturally emerges when one attempts to prove the positive mass theorem with arbitrary ends through a direct application of Schoen-Yau's original method \cite{SY79}. Specifically, this involves constructing an area-minimizing hypersurface as the limit of a sequence of free-boundary minimizing hypersurfaces within coordinate cylinders. Thus, the method used in Theorem \ref{thm: PMT arbitrary end spectral} yields an alternative proof of the positive mass theorem with arbitrary ends independent of the $\mu$-bubble approach, see Remark \ref{remark: arbitrary ends} for details.
	
	Returning to the proof of Theorem \ref{thm:pmt with singularity4}, a notable distinction arises as the  mass may be increased slightly during the singular harmonic function blow-up process. Through meticulous parameter selection, we will demonstrate that this increment will be outweighed by the reduction in mass attributed to the $\mathbf{S}^1$-warped product as previously mentioned. See Lemma \ref{lem: negative mass 2} for detailed calculation. With this "blow up - warped product" process, we are able to complete the proof of the strict inequality part in Theorem \ref{thm:pmt with singularity4}.
	By employing the similar strategy, combined  with  Theorem \ref{thm: PMT arbitrary end spectral}, we are able to obtain the non-negativity of the  mass in  Theorem \ref{thm:pmt with singularity4}. Thanks to  Lemma \ref{lmm:spectrum psc2},  in conjunction with  Proposition \ref{prop: eq1-arbitrary end} and  the inequality part in Theorem \ref{thm:pmt with singularity4}, we establish the rigidity part of Theorem \ref{thm:pmt with singularity4}. By a careful topological analysis and blowing up near singularities, Theorem \ref{thm:non-existence psc on singular space1} can be deduced from Proposition \ref{prop: noncompact dominate enlargeable 2}. 
        
$\quad$
        
	Next, we focus on area-minimizing boundaries in AF manifolds. Following  \cite{HSY24}, we select a sequence of exhausting coordinated cylinder $C_{r_i}\subset C_{r_{i+1}}$  in the AF end of  $(M^{n+1},g)$, and solve the Plateau's problem in $C_{r_i}$ with  an $(n-1)$-dimensional coordinated sphere within $C_{r_i}$ as boundary datum. Here $r_i$ denotes the coordinated  radius of $C_{r_i}$ and $r_i\rightarrow \infty$ as $i\rightarrow \infty$. Note that the solution to our Plateau's problem is  an area-minimizing hypersurface  that forms   boundary of a Caccioppoli set. Then by comparing with the coordinate ball enclosed by the $(n-1)$-dimensional coordinated sphere in $C_{r_i}$, we observe that  for $i$ large enough, the volume density of the solution of our Plateau's problem can be arbitrarily close to $1$ outside a fixed compact set of $M^{n+1}$ (see Lemma \ref{lem: density estimate1} below). Consequently, invoking of Allard's regularity theorem in geometric measure theory, we deduce that the singular set of $\Sigma_t$ in Theorem \ref{thm: foliation} is contained within a fixed compact set of $M^{n+1}$ (see Corollary \ref{cor: bound ms} below). The foliation structure of the AF end follows from the similar argument in \cite{HSY24}.

\begin{figure}
		\centering
		\includegraphics[width = 12cm]{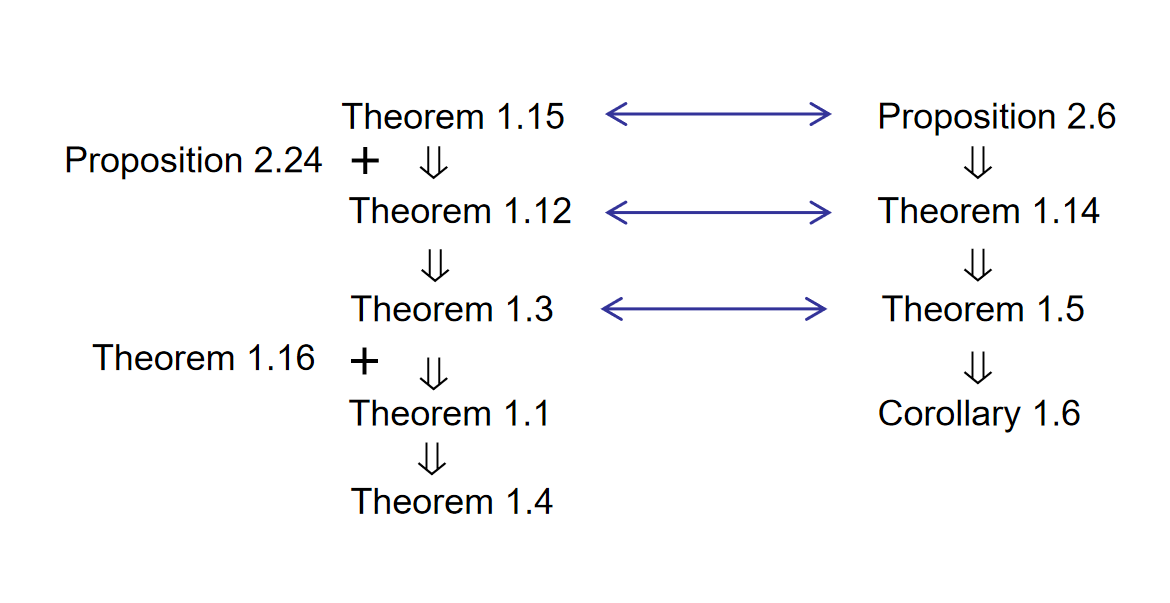}
		\caption{In the left column, we consider PMTs and the rigidity of stable minimal surfaces on AF manifolds; while in right column, we concern the corresponding 
        nonexistence of metrics with PSC in certain closed manifolds. The horizontal arrows represent correspondence between the PMT and the Georch type conjecture.}
		\label{outline}
	\end{figure}
    
 By a careful analysis of the possible singular set of the  area minimizing boundary, we can show that the strongly stable hypersurface $\Sigma$ in Theorem \ref{prop: rigidity for minimal surface}, as an AF manifolds with arbitrary ends, satisfies the assumption $(1)-(6)$ in Theorem\ref{thm:pmt with singularity4}. Moreover, it has $\frac{1}{2}$-scalar curvature no less than some nonnegative function $h = R_g+|A|^2$ with $h(p)>0$ for some $p\in \Sigma$ if $\mathcal{S}\ne 0$ in $\Sigma$. This enables us to use Theorem \ref{thm:pmt with singularity4} to get the desired result. The same strategy can also be applied to obtain the proof of Theorem \ref{thm: georch free of singularity}.

    For the  proof of Theorem \ref{thm: 8dim Schoen conj}, we adopt the contradiction argument.
    Suppose there exist a sequence of area minimizing boundary $\Sigma_i$ in free boundary sense
    such that $\Sigma_i\cap K\neq\emptyset$ for some fixed compact $K$. Then $\Sigma_i$ converge to a strongly stable area minimizing hypersurface. We can construct  a strongly stable area minimizing boundary $\Sigma_p$ passing through $p$ for any $p\in E$ with $|z(p)|\gg1$. In particular, $\Sigma_p\cap E$ is asymptotic to some coordinate hyperplane.  By combining Theorem \ref{prop: rigidity for minimal surface} with Theorem \ref{thm: foliation}, we show the mass of the AF end $E$ of $(M,g)$ equals $0$, which gives a contradiction.
    
     As an immediate  corollary, we provide a proof of positive mass theorem for AF manifolds with arbitrary ends and dimension no greater than $8$ without using N. Smale's regularity theorem for minimal hypersurfaces in a compact $8$-dimensional  manifold with generic metrics.
     
	We illustrate our outline of the proofs in Figure \ref{outline} above.

	\subsection{Organization of The Paper}
	In Section 2, we introduce some key definitions and preliminary results necessary for proving the main theorems. In Section 3, we  establish several positive mass theorems for smooth asymptotically flat (AF) manifolds with arbitrary ends and non-negative scalar curvature in the strong spectral sense. In Section 4,  we extend these results to singular AF manifolds.  Section 5 investigates the foliation structure of area-minimizing hypersurfaces in AF manifolds of dimension at least $4$ with arbitrary ends. In Section 6, we analyze stable minimal hypersurfaces in 8 manifolds and prove Theorem \ref{prop: rigidity for minimal surface} and Theorem \ref{thm: georch free of singularity}. Section 7 is devoted to the proof of Theorem \ref{thm: 8dim Schoen conj},  as an application, we derive the positive mass theorem for AF manifolds with arbitrary ends in dimensions no greater than $8$.

\section{Preliminaries}
\subsection{Asymptotically locally flat manifolds}
Asymptotically locally flat manifolds (ALF) manifolds are complete manifolds which can be regarded as natural generalizations of $\mathbf{R}^n\times \mathbf{T}^k$. For related studies about the positive mass theorems on ALF manifolds, one could see \cite{Dai04},\cite{Min09},\cite{CLSZ2021} and references therein. In this subsection, we will review some basic concepts. Let us begin with the following:

\begin{definition}\label{def: AF manifold}
		An end $E$ of an $n$-dimensional Riemannian manifold $(M,g)$ is said to be asymptotically flat (AF) of order $\tau$ for some $\tau>\frac{n-2}{2}$, if $E$ is diffeomorphic to $\mathbf{R}^n\backslash B^n_{1}(O)$, and the metric $g$ in $E$ satisfies
		\begin{align*}
			|g_{ij}-\delta_{ij}|+|x||\partial g_{ij}|+|x|^2|\partial^2 g_{ij}| = O(|x|^{-\tau}),
		\end{align*}
		and $R_g\in L^1(E)$. The  mass of $(M,g,E)$ is defined by
            \begin{align*}
                m_{ADM}(M,g,E) = \frac{1}{2(n-1)\omega_{n-1}}\lim_{\rho\to\infty}\int_{S^{n-1}(\rho)}(\partial_ig_{ij}-\partial_j g_{ii})\nu^jd\sigma_x
            \end{align*}
  \end{definition}

    In the case above we always call $(M,g)$ an asymptotically flat manifold: $(M,g)$ has an distinguished AF end $E$ and some arbitrary ends. Throughout the paper, we always assume
    the AF end of a  Riemannian manifold $(M^n,g)$  is of order $\tau$ for some $\frac{n-2}{2}<\tau\leq n-2$.
    
  \begin{definition}\label{def: ALF manifold}
      An end $E$ of an $n+k$-dimensional Riemannian manifold $(M,g)$ is said to be asymptotically locally flat (ALF) of order $\tau$ for some $\tau>\frac{n-2}{2}$, if $E$ is diffeomorphic to $(\mathbf{R}^n\backslash B^n_{1}(O))\times \mathbf{T}^k$, and the metric in $E$ satisfies
		\begin{align*}
			|g_{ij}-\bar{g}_{ij}|+|x||\partial g_{ij}|+|x|^2|\partial^2 g_{ij}| = O(|x|^{-\tau}),
		\end{align*}
		with $R_g\in L^1(E)$. Here $\bar{g}_{ij} = g_{Euc}\oplus g_{\mathbf{T}^k}$, where $g_{\mathbf{T}^k}$ denotes a flat metric on $\mathbf{T}^k$. The  mass of $(M,g,E)$ is defined by
        \begin{align*}
            m_{ADM}(M,g,E) = \frac{1}{2\omega_{n-1}Vol(\mathbf{T}^k,g_{\mathbf{T}^k})}\lim_{\rho\to\infty}\int_{S^{n-1}(\rho)\times \mathbf{T}^k}(\partial_ig_{ij}-\partial_j g_{aa})\nu^jd\sigma_xds
        \end{align*}
        where $i,j$ run over the index of the Euclidean space $\mathbf{R}^n$ and $a$ runs over all index.
  \end{definition}

    For the convenience of subsequent discussions, we introduce the following weighted spaces:
  \begin{definition}\label{defn: weighted space}
        Let $(M^n,g,E)$ be a complete manifold with an AF end. Let $r$ be a positive smooth function on $M$ which equals $|x|$ in $E$ and equals $1$ outside a neighborhood of $E$. 
        
      (1) Given any $s\in \mathbf{R}$ and $p>1$, we define $L^p_s(M)$ to be the subspace of $L^p_{loc}(M)$ with finite weighted norm
      \begin{align*}
          \|u\|_{L^p_s(M)} = (\int_M|u|^pr^{-sp-n}d\mu)^{\frac{1}{p}}.
      \end{align*}

      For each $k\in \mathbf{Z}_+$, we define $W^{k,p}_s(M)$ to be the subspace of $W^{k,p}_{loc}(M)$ with finite weighted norm
      \begin{align*}
          \|u\|_{W^{k,p}_s(M)} = \sum_{i=0}^k \|\nabla^i u\|_{L^p_{s-i}(M)}.
      \end{align*}

      (2) Given any $s\in \mathbf{R}$ and $k\in \mathbf{Z}_+$, we define $C^k_{s}(M)$ to be the subspace of $C^k(M)$ with finite weighted norm
      \begin{align*}
          \|u\|_{C^k_{s}(M)} = \sum_{i=0}^k \|r^{-s}\nabla^i u\|_{C^0(M)}.
      \end{align*}
  \end{definition}

\subsection{Some topological obstructions to PSC}

In this subsection, we explore several topological obstructions to the existence of PSC on complete manifolds. These results will be used in the subsequent sections.

\begin{definition}\label{Defn: enlargeable} (\cite{GL83}\cite{Gro18})

		A compact Riemannian manifold $X^n$ is said to be enlargeable if for each $\epsilon>0$, there is an oriented covering $\Tilde{X}\longrightarrow X$ and a continuous non-zero degree map $f:\Tilde{X}\longrightarrow S^n(1)$ that is constant outside a compaact set, such that $\Lip f<\epsilon$. Here $S^n(1)$ is the unit sphere in $\mathbf{R}^{n+1}$.
	\end{definition}
        
 Next, we prove the following proposition:

        \begin{figure}
    \centering
    \includegraphics[width = 12cm]{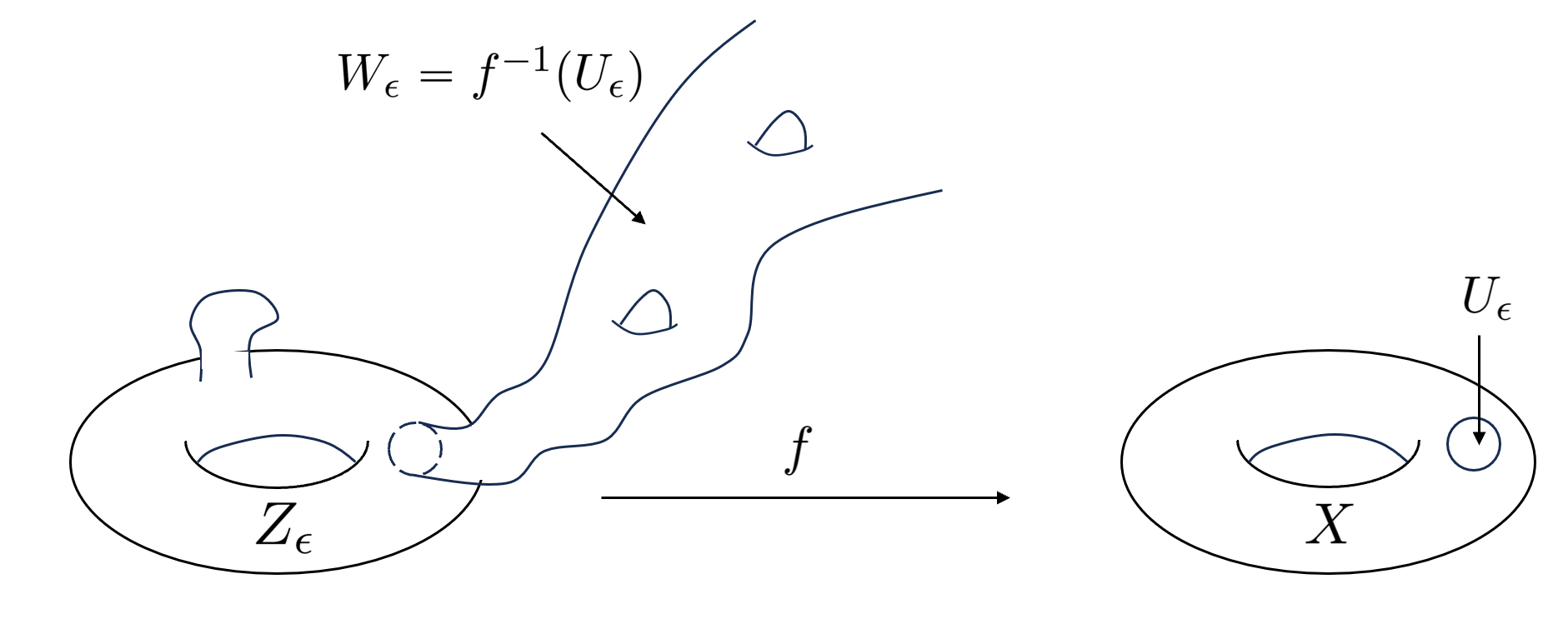}
    \caption{A noncompact manifold that admits a non-zero degree map to an enlargeable manifold}
    \label{f2}
\end{figure}
        
        \begin{proposition}\label{prop: noncompact dominate enlargeable}
            Let $X^n$ be a compact enlargeable manifold, and let $Y^n$ be a manifold which is not necessarily compact and admits a quasi-proper map $f: Y^n\longrightarrow X^n$ of non-zero degree $(n\le 7)$ (see Appendix B and Figure \ref{f2}). Then there exists no complete metric $g$ on $Y^n$ with $R_g> 0$.
        \end{proposition}
        \begin{proof}
            Let $S_{\infty} = \{p_1,p_2,\dots,p_m\}$. Fix a metric $h$ on $X^n$, and denote $U_{\epsilon}$ to be the $\epsilon$-neighborhood of $S_{\infty}$. We can select $\epsilon$ sufficiently small such that for any Riemannian covering $q_X:(\tilde{X},\tilde{h})\longrightarrow (X,h)$, each component of $q_X^{-1}(U_{2\epsilon})$ maps homeomorphically to $U_{2\epsilon}$ under $q_X$. Therefore, for any $x_1,x_2\in\tilde{X}$, it holds
            \begin{align}\label{eq: 46}
                d(x_1,x_2)\ge 4\epsilon
            \end{align}

             Note that $X$ is enlargeable, from \cite{Gro18}, for any $D>0$, we can specify a Riemannian covering $q_X:(\tilde{X},\tilde{h})\longrightarrow (X,h)$ with an embedded over-torical band $ {\mathcal{V}}\hookrightarrow\tilde{X}$ that satisfies $\width(\mathcal{V})>2D+1$. We use $\psi: \mathcal{V}\longrightarrow \mathbf{T}^{n-1}\times [-1,1]$ to be the non-zero degree map, and use $\pi: \mathbf{T}^{n-1}\times [-1,1]\longrightarrow \mathbf{T}^{n-1}$ to be the projection map. From \cite[Lemma 4.1]{Zhu21} we can find a proper smooth function $\rho:\mathcal{V}\longrightarrow [-D,D]$ that satisfies $\Lip\rho<1$, $\rho|_{\partial_{\pm}\mathcal{V}} = \pm D$. Without loss of generality, we can assume that $0$ is a regular value of $\rho$ and $0\not\in \mbox{Im}\rho\circ q^{-1}_X(S_{\infty})$. Denote $S = \rho^{-1}(0)$, then the map
            \begin{align*}
                \pi\circ(\psi|_S): S\longrightarrow \mathbf{T}^{n-1}
            \end{align*}
            has non-zero degree. In fact, by a standard differential topology argument, the degree of the map
            \begin{align*}
                \pi\circ(\psi|_{\partial_-\mathcal{V}}): \partial_-\mathcal{V}\longrightarrow \mathbf{T}^{n-1}
            \end{align*}
            equals that of $\degg\psi$. Because $S = \rho^{-1}(0)$ and $\partial_-\mathcal{V} = \rho^{-1}(-D)$, we know that $S$ and $\partial_-\mathcal{V}$ bound a region, so it follows that
            \begin{align}\label{eq: 59}
                \degg \pi\circ(\psi|_S) = \degg \pi\circ(\psi|_{\partial_-\mathcal{V}}) \ne 0
            \end{align}

            Next, by considering pullback objects repeatedly, we can draw the following diagram  which is commutative:

            \begin{equation*}
\xymatrix{\hat{S}\ar[r]\ar[d]^{\tilde{f}}&\hat{\mathcal{V}}\ar[r]\ar[d]^{\tilde{f}}&\tilde{Y}\ar[r]^{q_Y}\ar[d]^{\tilde{f}}&Y\ar[d]^f\\
S\ar[r]&\mathcal{V}\ar[r]\ar[d]^{\psi}&\tilde{X}\ar[r]^{q_X}&X\\
&\mathbf{T}^{n-1}\times[-1,1]\ar[r]^{\pi}&\mathbf{T}^{n-1}}
\end{equation*}
        In this diagram, each unlabeled horizontal arrow is given by an inclusion map. $q_Y:\tilde{Y}\longrightarrow Y$ is a Riemannian covering, and $\tilde{f}$ is also a quasi-proper map. $\hat{\mathcal{V}}$ is  an embedded band in $\tilde{Y}$, which is possible to be noncompact. Without loss of generality, we may assume $\tilde{f}$ is transversal to $\partial_{\pm}\mathcal{V}$ and $S$, so $\partial_{\pm}(\hat{\mathcal{V}})$ and $\hat{S}$ are regular hypersurfaces in $\tilde{Y}$. Furthermore, our assumption $0\not\in \mbox{Im}\rho\circ q^{-1}_X(S_{\infty})$ ensures that $\hat{S}$ is compact. Denote $W_{\epsilon} = f^{-1}(U_{\epsilon})$. Since $f$ is quasi-proper, we know that $Z_{\epsilon} = f^{-1}(X\backslash U_{\epsilon})$ is compact. Therefore, there exists a constant $C$, such that
        \begin{align}\label{eq: 44}
            \Lip(f|_{Z_{\epsilon}})<C.
        \end{align}
        Denote $\tilde{W}_{\epsilon} = q_Y^{-1}(W_{\epsilon})$ and $\tilde{Z}_{\epsilon} = q_Y^{-1}(Z_{\epsilon})$. It follows from \eqref{eq: 44} that
        \begin{align}\label{eq: 45}
            \Lip(\tilde{f}|_{\tilde{Z}_{\epsilon}})<C.
        \end{align}

        The main idea of the subsequent proof is to derive a contradiction by utilizing the favorable properties of the noncompact band $\hat{\mathcal{V}}$. The first step is to show that $\hat{S}$ is a {\it deep} hypersurface in $\hat{\mathcal{V}}$:
        \begin{align}\label{eq: 47}
            d(\partial_{\pm}(\hat{\mathcal{V}}),\hat{S})\ge\frac{D}{2C}.
        \end{align}
        To prove \eqref{eq: 47}, let $\gamma$ be a curve joining $\partial_-\hat{\mathcal{V}}$ and $\hat{S}$. We represent $\gamma$ in a piecewise manner: $\gamma = \xi_1\zeta_1\xi_2\zeta_2\dots\xi_{m-1}\zeta_{m-1}\xi_m$, where $\xi_i\subset \tilde{Z}_{\epsilon}$ and $\zeta_i\subset \tilde{W}_{\epsilon}$. Denote $\bar{\gamma} = \tilde{f}\circ \gamma$, $\bar{\xi}_i = \tilde{f}\circ \xi_i$ and $\bar{\zeta}_i = \tilde{f}\circ \zeta_i$. Recall that $d(\partial_{\pm}(\hat{\mathcal{V}}),\hat{S})\ge D$. Combined with the fact that $\diam(B_{\epsilon}(p_i))<2\epsilon$ and \eqref{eq: 46}, we obtain
        \begin{align}\label{eq: 57}
            D<\sum_{i=1}^m L(\bar{\xi}_i)+2(m-1)\epsilon\le 2\sum_{i=1}^m L(\bar{\xi}_i).
        \end{align}
In conjunction with \eqref{eq: 45}, we obtain
        \begin{align}\label{eq: 58}
            \sum_{i=1}^m L(\bar{\xi}_i)\le\frac{1}{C}\sum_{i=1}^m L({\xi}_i)\le \frac{1}{C}L(\gamma).
        \end{align}
       Combining  \eqref{eq: 57} with \eqref{eq: 58} yields \eqref{eq: 47}.

Define $\Omega_0 = (\rho\circ\tilde{f})^{-1}((-D,0))$ and
\begin{align*}
    \mathcal{G} = \{ &\Sigma\subset \hat{\mathcal{V}}, \Sigma = \partial\Omega\cap\mathring{\hat{\mathcal{V}}},\\
    &\Omega\subset \mathcal{V}\mbox{ is a region with smooth boundary and }\Omega\Delta\Omega_0 \mbox{ is compact }.\}
\end{align*}
For each $\Sigma\subset\mathcal{G}$, \eqref{eq: 59} implies
\begin{align*}
    \degg (\pi\circ\psi\circ\tilde{f}|_{\Sigma}) = \degg (\pi\circ\psi\circ\tilde{f}|_{\hat{S}}) = \degg f\degg (\pi\circ\psi|_S) \ne 0.
\end{align*}
In conjunction with \cite{SY79b}, $\Sigma$ admits no PSC metric. The $\mu$-bubble argument in \cite{CL2024} then goes through smoothly.
\end{proof}

With a similar argument, using $\mathbf{S}^1$-invariant $\mu$-bubble in \cite{WY2023} instead of the ordinary $\mu$-bubble, we have the following proposition.

        \begin{proposition}\label{prop: noncompact dominate enlargeable 2}
            Let $X^n$ be an enlargeable manifold, and let $Y^n$ be a manifold which is not necessarily compact and  admits a map $f: Y^n\longrightarrow X^n$ of non-zero degree $(n\le 7)$. Then there exists no $\mathbf{S}^1$-invariant complete metric $g$ on $Y^n\times \mathbf{S}^1$ with $R_g> 0$.
        \end{proposition}

        Since the proof of Proposition \ref{prop: noncompact dominate enlargeable 2} is a direct consequence of the argument of Proposition \ref{prop: noncompact dominate enlargeable} combined with the method developed in \cite{WY2023}, we omit the proof. As a corollary, we are able to prove the following obstruction result for PSC in the spectral sense.

        \begin{corollary}\label{cor: noncompact dominate enlargeable 3}
            Let $X^n$ be an enlargeable manifold, and let $Y^n$ be a manifold which is not necessarily compact and admits map $f: Y^n\longrightarrow X^n$ of non-zero degree $(n\le 7)$. Then any complete metric $g$ on $Y^n$ with $\lambda_1(-\Delta_g+\beta R_g)\ge 0$ for some $\beta\ge\frac{1}{2}$ is flat.
        \end{corollary}

        To prove Corollary \ref{cor: noncompact dominate enlargeable 3}, we need an $\mathbf{S}^1$-invariant version of Kazdan's deformation theorem in \cite{Kazdan82}.

        \begin{lemma}\label{G-invariant Kazdan}
    Let $(M^n,g)$ be a complete Riemannian manifold with $\mathbf{S}^1$ acting freely and isometrically on it. Assume $R_g\ge 0$, and $Ric_g$ is not identically zero. Then there exists an $\mathbf{S}^1$-invariant complete metric $\tilde{g}$ on $M$ with $R_{\tilde{g}}>0$ everywhere.
\end{lemma}
        \begin{proof}
            We adopt the deformation arguments in \cite{Kazdan82}. We assert that the Riemannian  metric $G$ constructed  in Theorem B in \cite{Kazdan82} can be constructed to be $\mathbf{S}^1$-symmetric. Indeed, let $\rho(x)$ be the distance function of $M/\mathbf{S}^1$. By an $\mathbf{S}^1$-invariant mollification, we may assume $\rho$ is a $\mathbf{S}^1$-invariant smooth function on $M$. For any $i\geq 1$ we define
            \begin{align*}
                M_i:=\{x\in M: \rho(x)<i\}.
            \end{align*}
Then, $M_i$ is a $\mathbf{S}^1$-symmetric domain in $M$, and we can  make $p$ in Lemma 2.1 in  \cite{Kazdan82}  to be $\mathbf{S}^1$-symmetric, i.e. $p$ depends only on $x\in M/\mathbf{S}^1$. Thus, the domains and equations involved in  Theorem B in \cite{Kazdan82} are all $\mathbf{S}^1$-symmetric. By the uniqueness of solutions of involved equations, we know that the Riemannian manifold $(\hat M^{n}, G)$ constructed in  Theorem B in \cite{Kazdan82} is complete and  $\mathbf{S}^1$-symmetric, and  its scalar curvature is positive  everywhere. 
        \end{proof}

        \begin{proof}[Proof of Corollary \ref{cor: noncompact dominate enlargeable 3}]
            By \cite[Theorem 1]{FcS1980}, there exists $v\in C^{\infty}(Y)$ that solves $-\Delta_g v+\frac{1}{2}R_g v = 0$. Consider $(\hat{Y},\hat{g}) = (Y\times \mathbf{S}^1, g+v^2ds^2)$, then $R_{\hat{g}} = 0$, and $\hat{g}$ is $\mathbf{S}^1$-symmetric. If the Ricci curvature of $(\hat{Y},\hat{g})$ is not identically zero, it then follows from Lemma \ref{G-invariant Kazdan} that $\hat{Y}$ admits a $\mathbf{S}^1$-invariant complete metric with positive scalar curvature everywhere, which contradicts with Proposition \ref{prop: noncompact dominate enlargeable 2}. Hence, we have $Ric_{\hat{g}}\equiv 0$.  By direct computation, $Ric_{\hat{g}}(\frac{\partial}{\partial s},\frac{\partial}{\partial s}) = -v\Delta_g v = 0$, so $v$ is a constant and $Ric_{g}\equiv 0$.

            Next, we show $Y$ is compact. Suppose this is not the case, we have $S_{\infty}\ne\emptyset$. Let $\tilde{X}$ be the universal covering of $X$. We adopt the notations in the proof of Proposition \ref{prop: noncompact dominate enlargeable}. Since enlargeable manifolds have infinite fundamental group, $q_X^{-1}(U_{\epsilon})$ has infinitely many components. Consequently, $\tilde{f}^{-1}(q_X^{-1}(U_{\epsilon})) = q_Y^{-1}(f^{-1}(U_{\epsilon}))$ has infinitely many non-compact components. From the properness of $\tilde{f}$ we deduce $\tilde{Y}$ has infinitely many ends. Since manifolds with at least two ends contains a geodesic line, it follows readily from \cite{CG71} that Ricci flat manifolds have at most two ends, a contradiction.

            Since the enlargeability is preserved under non-zero degree maps, $Y$ is also enlargeable. The flatness then follows from \cite[Proposition 4.5.8]{LM89}.
        \end{proof}

       As an application, we obtain the following $\mathbf{S}^1$-invariant version positive mass theorem for ALF manifolds with $\mathbf{S}^1$ fibers and dimension less or equal than $8$:
        \begin{proposition}\label{prop: PMT for S^1 symmetric ALF}
	Let $(\hat{M}^{n+1}, \hat{g}) = (M^{n}\times \mathbf{S}^1,\hat{g})$ be an ${n+1}$-dimensional complete Riemannian manifold with an $ALF$ end $\hat{E} = E\times \mathbf{S}^1$ and $2\leq n\leq 7$. Assume $\hat{g}$ is $\mathbf{S}^1$ invariant along the $\mathbf{S}^1$ fiber, and $R_{\hat{g}}\ge 0$. Then $m_{ADM}(\hat{M},\hat{g},\hat{E})\ge 0$. Furthermore, if $m_{ADM}(\hat{M},\hat{g},\hat{E})= 0$, then $\hat{M}$ is isometric to the Riemannian product $(\mathbf{R}^{n}\times \mathbf{S}^1, g_{euc}+ds^2)$.
\end{proposition}

\begin{proof}
    The proof is a modification of that of \cite[ Theorem 1.10]{CLSZ2021}. Assume that the ADM mass is negative. Due to the $\mathbf{S}^1$-symmetry of $(\hat M, \hat g )$ and \cite[Proposition 4.10]{CLSZ2021}
 we know that the conformal factor in the gluing is independent of $s$, where $s$ parametrizes $\mathbf{S}^1$. To be more precise, from Fredholm alternative the equation (4.16) in \cite{CLSZ2021} has a unique solution, so it must be $\mathbf{S}^1$-symmetric as long as $(\hat M, \hat g )$ is $\mathbf{S}^1$-symmetric. Thus, just as what was done in the proof of \cite[Theorem 1.10 ]{CLSZ2021}, we finally get an $\mathbf{S}^1$-symmetric complete Riemannian manifold $((\mathbf{T}^{n}\# N^{n})\times \mathbf{S}^1, \tilde{g})$ with $n\leq 7$, which has non-negative scalar curvature everywhere and strictly positive scalar curvature somewhere.  Here $N^{n}$ is an $n$-dimensional manifold. Now, this contradicts to Proposition \ref{prop: noncompact dominate enlargeable 2} and Lemma \ref{G-invariant Kazdan}. The rigidity part also follows from the same reason and the argument in \cite{CLSZ2021}.
\end{proof}

\subsection{Sobolev space on metric measure spaces}

Let $(M, d,\mu)$ be a metric measure space that satisfies the condition of Theorem \ref{thm:pmt with singularity4}. Let $X\subset M$ be an open set with compact closure. In this subsection, we will collect some known facts about the Sobolev space defined on $(X,d_X,\mu_X)$, where $d_X = d|_X$ and $\mu_X = \mu|_X$ denote the restriction of $d$ and $\mu$ on $X$. Since $X$ has compact closure we have $\mu(X)<+\infty$.

\begin{definition}
    For each $f\in L^2(X,\mu_X)$, the Cheeger energy of $f$ is defined by
    \begin{align}\label{eq: 67}
        Ch(f) = \inf\{\liminf_{i\to\infty}\frac{1}{2}\int_X(\Lip f_i)^2d\mu_X:\quad f_i\in\Lip(X),\  \|f_i-f\|_{L^2}\to 0.\}
    \end{align}

The Sobolev space $W^{1,2}(X,d_X,\mu_X)$ is then defined to be $\{f: Ch(f)<+\infty\}$. 
\end{definition}
Since $(X,d_X,\mu_X)$ is doubling, $W^{1,2}(X,d_X,\mu_X)$ is Banach and reflexiv when endowed with the norm
\begin{align}\label{eq: 73}
    \|f\|_{W^{1,2}} = (\|f\|_{L^2(X,\mu_X)}^2+2Ch(f))^{\frac{1}{2}}.
\end{align}

\begin{lemma}\label{lem: estimate for eta} (Zero capacity property)
    Let $\mathcal{S}$ be a compact singular set with $\mathcal{H}^{n-2}(\mathcal{S}) = 0$. Then for any neighborhood $U$ for $\mathcal{S}$ and $\epsilon>0$, there exists an open set $V$ with $\mathcal{S}\subset V\subset U$ and a cut off function $\eta\in C^{\infty}(M\backslash \mathcal{S})$ that satisfies $\eta = 0$ in $V$, $\eta = 1$ in $M\backslash U$ and
    \begin{align*}
        \int_{M}|\nabla\eta|^2d\mu<\epsilon^2.
    \end{align*}
\end{lemma}
\begin{proof}
    We use an argument of \cite[Appendix]{AX24}. For each $\epsilon>0$, we can find a finite collection of balls $B_{r_i}(x_i)$ that covers $\mathcal{S}$ and satisfies
    \begin{align*}
        \sum_ir_i^{n-2}<\epsilon^2.
    \end{align*}
    For each $i$, we find a cutoff function $\eta_i$ which is smooth in $M\backslash\mathcal{S}$ that satisfies
    \begin{align*}
        \eta_i = 0\text{ in } B_{2r_i}(x_i),\quad \eta_i = 1\text{ in } M\backslash B_{3r_i}(x_i),\quad |\nabla_g\eta_i|<\frac{2}{r_i}\text{ in }M\backslash\mathcal{S}.
    \end{align*}
    Let $\hat{\eta} = \min_i\{\eta_i\}\in\Lip(M\backslash\mathcal{S})$ and let $\eta$ be obtained by regularizing $\hat{\eta}$. We have $|\nabla_g\hat{\eta}|\le 2|\nabla_g\eta|$ and
    \begin{align*}
        \eta = 0\text{ in }\bigcup_iB_{r_i}(x_i),\quad\eta = 1\text{ in }M\backslash\bigcup_iB_{4r_i}(x_i).
    \end{align*}
    Therefore,
    \begin{align*}
        \int_M|\nabla_g\eta|^2d\mu\le 2\sum_i\int_M|\nabla_g\eta_i|^2d\mu\le 8C\sum_i r_i^n\cdot r_i^{-2}<8C\epsilon^2,
    \end{align*}
    and the conclusion follows.
\end{proof}

The next Proposition is a consequence of \cite[Proposition 4.10]{AGS14}, \cite[Proposition 3.3]{Hon18} and the zero capacity property Proposition \ref{lem: estimate for eta}. A slight difference is that in \cite[Proposition 3.3]{Hon18}, $X$ is assumed to be compact. As we have assumed $\mu(X)<+\infty$, the proof of \cite[Proposition 3.3]{Hon18} can be smoothly carried to our setting.
\begin{proposition} (Density of Lipschitz functions in $W^{1,2}$)
    $W^{1,2}(X)$ is a Hilbert space and $\Lip(X)$ is dense in $W^{1,2}(X)$.
\end{proposition}
\begin{proof}
    We sketch the proof following the strategy in \cite[Proposition 3.3]{Hon18}. First, on the smooth part, we have
    \begin{align}\label{eq: 66}
        Ch(\varphi+\psi)+Ch(\varphi-\psi) = 2Ch(\varphi)+2Ch(\psi)
    \end{align}
    for all $\varphi,\psi\in \Lip(X\backslash\mathcal{S})$. By combining Lemma \ref{lem: estimate for eta} with an approximation argument, we obtain that \eqref{eq: 66} holds for all $\varphi,\psi\in \Lip(X)$. By \eqref{eq: 67} we can further obtain the same thing holds for $\varphi,\psi\in W^{1,2}(X)$ (see \cite[Proposition 3.3]{Hon18} for detail). The conclusion then follows directly from \cite[Proposition 4.10]{AGS14}.
\end{proof}

\begin{lemma}\label{lem: compact embed} ($L^2$ strong compactness)
    The inclusion $L^2(X)\hookrightarrow W^{1,2}(X)$ is a compact operator.
\end{lemma}
\begin{proof}
    Since $X$ is doubling due to the local volume bound, and $X$ satisfies the Poincare inequality, the result of the lemma follows from \cite[Theorem 8.1]{HK00}.
\end{proof}

 For $f\in\Lip(X)$, we say $u\in W^{1,2}(X)$ is the solution of the Poisson equation $-\Delta u = f$ if for any $\varphi\in \Lip(X)$ with compact support, there holds
\begin{align}\label{eq: 71}
    \int_X \nabla u\cdot\nabla\varphi d\mu= \int_Xf\varphi d\mu.
\end{align}
Define $W^{1,2}_0(X)$ to be the closure of Lipschitz functions with compact support in $W^{1,2}(X)$. We say $u$ satisfies the Direchlet boundary condition $u = 0$ on $\partial X$ if it further satisfies $u\in W^{1,2}_0(X)$. We also say $-\Delta u\le f$ if we have $\le$ sign in \eqref{eq: 71} for all $\varphi\ge 0$ with compact support in $\Lip(X)$.

\begin{lemma}\label{lem: Poisson equation} (Solvability of the Poisson equation)
    For $f\in\Lip(X)$, there exists $u\in W^{1,2}_0(X)$ such that $-\Delta u = f$.
\end{lemma}

\begin{proof}
    Consider the functional
    \begin{align*}
        E(u) = \int_X (\frac{1}{2}|\nabla u|^2+fu)d\mu.
    \end{align*}
    By the Sobolev inequality, $E(u)$ is bounded from below, so we can consider a minimizing sequence $u_i$ for $E$. Using Lemma \ref{lem: compact embed}, we obtain a limit $u$. Together with the lower semicontinuity of $E$, $u$ is a minimizer for $E$. Then a variational argument shows that $u$ solves the equation.
\end{proof}

By employing the same reasoning as presented in \cite[Theorem 8.15, Theorem 8.16]{GT}, we deduce that the weak maximum principle holds for subharmonic functions $u$ on $X$.
\begin{lemma}\label{lem: weak maximum principle} (Weak maximum principle)
    Assume $u\in W^{1,2}(X)$ satisfies $\Delta u \ge 0$ in $X$. Define $\sup_{\partial X} u = \inf \{l\in\mathbf{R}: \max \{u-l,0\}\in W^{1,2}_0(X)\}$. Then we have
    \begin{align*}
        \sup_X u\le \sup_{\partial X} u.
    \end{align*}
\end{lemma}

\begin{lemma}\label{lem: removable singularity} (Removable singularity)
    Let $u\in W^{1,2}(X\backslash\mathcal{S})\cap L^{\infty}(X\backslash\mathcal{S})$ be a solution to $\Delta u = 0$ in $X\backslash\mathcal{S}$, then $u\in W^{1,2}(X)$ and is a solution to $\Delta u = 0$ in $X$. 
\end{lemma}
\begin{proof}
    Let $\eta$ be the cut-off function in Lemma \ref{lem: estimate for eta}. By integration by parts we have
    \begin{align*}
        \int_{X\backslash\mathcal{S}}\divv(\eta^2 u\nabla u)d\mu = \int_{X\backslash\mathcal{S}}|\nabla\eta u|^2-|\nabla\eta|^2u^2 d\mu.
    \end{align*}
    Since the left hand side is bounded, using Lemma \ref{lem: estimate for eta} and passing to a limit we obtain $u\in W^{1,2}(X)$. 
    
    Next, we check that $u$ is the weak solution of the Poisson equation $\Delta u = 0$ in $X$. Let $\varphi\in C^{\infty}_0(X)$ and $\eta$ be the cut-off function in Lemma \ref{lem: estimate for eta}, by the definition of the weak solution we have
    \begin{align*}
    \int_X \nabla u\cdot\nabla\eta\varphi d\mu= \int_Xf\eta\varphi d\mu.
\end{align*}
Passing to a limit we see that $u$ solves the equation $\Delta u = 0$ in $X$.
    \end{proof}

Since we have the Poincare inequality, by exactly the same argument in \cite{BG1972}, we have the following Harnack's inequality holds in $X$.
    \begin{lemma}\label{lem: Harnack}(Harnack's inequality)
        There exists $\alpha = \alpha(n), \beta = \beta(n), C = C(n)$, such that for any positive solution $u\in W^{1,2}(B(p,R))$ of equation $\Delta u = 0$ in $B(p,R)\subset X$, we have $u\in C^{\alpha}$ and
        \begin{align*}
           \sup_{B(p,\beta R)} u \le C \inf_{B(p,\beta R)} u .
        \end{align*}
        
        \end{lemma}

\subsection{Singular harmonic functions on complete metric measure spaces}

\subsubsection{Singular harmonic function on a region}

Let $(M^n, d,\mu)$ be a metric measure space that satisfies the condition of Theorem \ref{thm:pmt with singularity4}. We will utilize a singular harmonic function to blow up the singular set $\mathcal{S}$. Here and in the sequel, we always assume $\mathcal{H}^{n-2}(\mathcal{S}) = 0$, which automatically holds under the assumption of Theorem \ref{thm:pmt with singularity4}. Let's  start with  

\begin{proposition}\label{prop:existence of singular positive hf}
Suppose $(M^n, d, \mu)$ is a metric measure space satisfying almost-manifold condition.
Let $\Omega$ be a compact domain of $M^n$ with smooth boundary and containing a singular point  $p\in \mathcal{S}$. Then there is a harmonic function $G$ on $\Omega\setminus\{p\}$ such that
\begin{enumerate}
	\item $G|_{\partial\Omega}=0$;
	\item $G(x)>0$ for any $x\in \Omega\setminus\{p\}$;
	\item $\lim_{x\rightarrow p}G(x)=+\infty$.
\end{enumerate} 
Furthermore, if $v$ is another harmonic function on $\Omega\setminus\{p\}$ satisfying all above properties, then  $v=\alpha G$ for some $\alpha>0$.
\end{proposition}

\begin{proof}
	Let $\{r_i\}$ be a decreasing sequence which approaches zero, and $G_i$ be the solution of the following Dirichlet problem

\begin{equation}\label{eq:weak hf}
\left\{
\begin{aligned}
&\Delta G_i=0, \quad \text{in $\Omega\setminus B(p,r_i)$},\\
&G_i|_{\partial \Omega}=0,\\
&G_i|_{\partial B(p,r_i)}=1
\end{aligned}
\right.
\end{equation}
By the weak maximum principle, we see that 
$G_i$ in \eqref{eq:weak hf} is positive in $\Omega\setminus  B(p,r_i)$. Let $x_0\in \Omega$, by a scaling we may assume $G_i(x_0)=1$. Then due to Harnack inequality, we see that, by passing to a subsequence,  $\{G_i\}$ converges to a positive harmonic function $G$ satisfying all the first two items in Proposition \ref{prop:existence of singular positive hf} and $G(x_0)=1 $. Next, we will show the third item is also fulfilled. Indeed, for any small $r>0$, note that $\Lambda^{-1}r^n\leq\mathcal{H}^n(B(p, r))\leq \Lambda r^n$, by a covering argument,  we see that there is a fixed number $K$ which depends only on $\Lambda$ so that there is $\{q_1, \cdots, q_K\}\subset \partial B(p,r)$ with  
\begin{enumerate}
	\item $\partial B(p,r)\subset\bigcup^K_{i=1}B(q_i, \frac{r}{10})$;
	\item $B(q_i, \frac{r}{10})\cap B(q_{i+1}, \frac{r}{10})\neq\emptyset$ for $1\leq i\leq K$. Here we assume $B(q_{K+1}, \frac{r}{10})=B(q_{1}, \frac{r}{10})$.
\end{enumerate}
Again, due to Harnack inequality, there is a positive constant $C$ depends only on $\Lambda$ and $\gamma$,  and for each $i$ there holds

$$
\sup_{x\in B(q_i, \frac{r}{10})}G(x)\leq C \inf_{x\in B(q_i, \frac{r}{10})}G(x),
$$
thus,

\begin{equation}\label{eq:harnack1}
\sup_{x\in \partial B(p,r)}G(x)\leq C^K \inf_{x\in \partial B(p,r)}G(x).	
\end{equation}

Using the assumption (5) in Theorem \ref{thm:pmt with singularity4} and Lemma \ref{lem: weak maximum principle} we know that if $G(x)$ does not approach to infinity, then it is uniformly bounded. Combined with Lemma \ref{lem: removable singularity} and Lemma \ref{lem: Harnack}, $G$ can be extended to $p$ continuously. Once again, by the weak maximum principle mentioned above, we get $G=0$  on $\Omega$, which contradicts to $G(x_0)=1$. 

Finally, we will verify the uniqueness. For simplicity, we assume $v(x_0)=1$ and are going to show $G=v$ in $\Omega\setminus\{p\}$. Let 
$$
m(r):=\inf_{x\in \partial B(p,r)}\frac{v(x)}{G(x)}.
$$
Then $w(x):=v(x)-m(r)G(x)$ is harmonic in $\Omega\setminus B(p,r)$, and hence, $w\geq 0$ in $\Omega\setminus B(p,r)$, which implies $m(r)$ is non-decreasing, i.e., for any $r_2\geq r_1$ there holds $m(r_2)\geq m(r_1)$.

\textbf{Claim: Let $m$ be the limit of $m(r)$ as $r$ approaches zero.
Then $m>0$.}

Suppose not, then we consider 
\[
\bar{m}(r):=\inf_{x\in \partial B(p,r)}\frac{G(x)}{v(x)}.
\]
By Harnack inequality, $v(x)$ satisfies
\[
\sup_{x\in \partial B(p,r)}v(x)\leq C^K \inf_{x\in \partial B(p,r)}v(x).
\]
for any small $r>0$.
It follows that
\[
\bar{m}(r):=\inf_{x\in \partial B(p,r)}\frac{G(x)}{v(x)}\geq C^{-2K}
m^{-1}(r). 
\]
Thus, $\bar{m}(r)\to \infty$  as $ r\to 0$.

On the other hand, by the argument before, one can show $\bar{m}(r)$ is also non-decreasing. 
Then we get the contradiction.

  For any $x\in \Omega\setminus\{p\}$, define
  $$
  h(x)=v(x)-m G(x).
  $$
Then either $h(x)=0$ for all $x\in \Omega\setminus\{p\}$, which we finish the proof; or $h(x)>0$ for all $x\in \Omega\setminus\{p\}$. If the latter case happens, using $v|_{\partial\Omega} = G|_{\partial\Omega} = 0$, we obtain 
$$\lim_{x\rightarrow p}h(x)=+\infty.
$$
By a similar argument, there is a constant $m_1>0$ with $h_1(x):=h(x)-m_1 G(x)>0$, that means for all $x\in \Omega\setminus\{p\}$ 
$$
v(x)-(m+m_1)G(x)>0,
$$
which contradicts the definition of $m$. Therefore, we get the conclusion.
\end{proof}

Our next proposition concerns the capacity of sub-level set of  a singular harmonic function. Specifically, let $\mathcal{S}\subset M$ be a compact set, $G$ be a harmonic function on $M\setminus \mathcal{S}$ with 
\begin{enumerate}
	\item $G|_{\mathcal{S}}=+\infty$;
	\item $G|_{\partial\Omega}=0$.
\end{enumerate}

\begin{proposition}\label{prop:capacity1}
Let $M$ be as in Proposition \ref{prop:existence of singular positive hf},  $\mathcal{S}\subset M$ be a compact set, $G$ be a singular harmonic function described above, then
$$
\int_{\{0<G\leq t\}}|\nabla G|^2 d\mu =C t,
$$ 	
where
\begin{align}\label{eq: 75}
    C = C_G = \int_{\partial\Omega} |\nabla G|d\sigma
\end{align}
is a positive constant. 
\end{proposition}
\begin{proof}
        Since $\mathcal{H}^{n-2}(\mathcal{S}) = 0$, we can use the co-area formula to conclude that
        \begin{align}\label{eq: 54}
            \int_{\{G=t\}} |\nabla G|
        \end{align}
        is integrable for $a.e.$ $t$ and
        \begin{equation}\label{eq: 56}
		\begin{split}
                \int_{\{0<G\leq t\}}|\nabla G|^2 d\mu
			&=\int^t_0\int_{\{G=s\}}|\nabla G|^2	\cdot |\nabla G|^{-1}	d\sigma_s ds\\
			&=\int^t_0 ds\int_{\{G=s\}}|\nabla G|	d\sigma_s. \\
			\end{split}
	\end{equation}
        By Lemma \ref{lem: estimate for eta}, for any neighborhood $U$ for $\mathcal{S}$ and $\epsilon>0$, there exists an open set $V$ with $\mathcal{S}\subset V\subset U$ and a cut off function $\eta\in C^{\infty}(M\backslash \mathcal{S})$ that satisfies $\eta = 1$ in $V$, $\eta = 0$ in $M\backslash U$ and
    \begin{align}\label{eq: 55}
        \int_{M}|\nabla\eta|^2d\mu<\epsilon^2.
    \end{align}
    Assume $t$ is a regular value for $u:M\backslash\mathcal{S}\longrightarrow\mathbf{R}$ and such that \eqref{eq: 54} is integrable, then by integrating by parts, we have
    \begin{equation*}
        \begin{split}
            \int_{\{G=t\}}\eta|\nabla G|d\sigma_t - \int_{\{G=0\}}\eta|\nabla G|d\sigma_0 = \int_{\{0\le G\le t\}}div(\eta\nabla G)d\mu.
        \end{split}
    \end{equation*}
    By Cauchy-Schwartz inequality and \eqref{eq: 55}, we have
    \begin{equation*}
        \begin{split}
            |\int_{\{0\le G\le t\}}div(\eta\nabla G)d\mu| = |\int_{\{0\le G\le t\}}\nabla\eta\cdot\nabla Gd\mu|\le\epsilon\|u\|_{W^{1,2}}
        \end{split}.
    \end{equation*}
    Let $\epsilon\to 0$, and by applying the dominate convergence theorem, we conclude
    \begin{align*}
        \int_{\{G=t\}}|\nabla G|d\sigma_t = \int_{\{G=0\}}|\nabla G|d\sigma_0.
    \end{align*}
    Plugging the equality above into \eqref{eq: 56}, we get the desired result.
\end{proof}

As a corollary of Proposition \ref{prop:capacity1}, we obtain:

\begin{proposition}\label{prop:growth of hf1}
Let \( G \) be a positive singular harmonic function in \( \Omega \setminus \{p\} \), as described in Proposition \ref{prop:capacity1}. Then, there exists a positive constant \( C = C_G\cdot\Lambda^{-1} \), such that for any $r\in (0,r_0:=\inf_{x\in\partial\Omega} r(p,x))$, the following holds:
$$
G|_{\partial B(p,r)}\geq Cr^{2-n},
$$
where $C_G$ is given by \eqref{eq: 75}.
\end{proposition}

\begin{proof}
Let $t=\max_{x\in \partial B(p,r)}G(x)$, then due to the weak maximum principle, 

$$\{x\in \Omega: G(x)\geq t\}\subset  B(p,r).$$	
Note that  the variational capacity of a set $D\subset \Omega$ is given by
$$
Cap(D, \Omega)=\inf_{u|_D\geq 1, u|_{\partial\Omega}=0}\int_{\Omega}|\nabla u|^2d\mu.
$$
In conjunction with Proposition \ref{prop:capacity1}, we have
$$
C_G t^{-1}=Cap(\{x\in \Omega: G(x)\geq t\}, \Omega)\leq Cap(B(p,r), \Omega)\leq \Lambda r^{n-2}, 
$$
where $C_G$ is defined by \eqref{eq: 75} and in the last inequality we have used the local volume bound on $M$. Therefore,
\begin{align}\label{eq: 74}
    t\ge C_G \Lambda^{-1} r^{2-n} = Cr^{2-n}.
\end{align}
Combined with Harnack inequality we get the conclusion.	
\end{proof}

Based on Proposition \ref{prop:growth of hf1}, we are able to  investigate  the the higher dimensional singular set case. Namely, 

\begin{proposition}\label{prop:growth of hf2}
Let $\mathcal{S}=\bigcup^m_{i=1} \mathcal{S}_i$, each $\mathcal{S}_i$ is a compact set of $M^n$. There exist $\Lambda>0$, $\delta_0>0$ such that  for each $i$ there is a $0\leq k_i\leq n-2$ with $\Lambda^{-1}\leq \mathcal{H}^{k_i}(\mathcal{S}_i)\leq \Lambda^{-1}$, and   for any $x\in \mathcal{S}_i$, $r\leq \delta_0$ there holds $\Lambda^{-1}r^{k_i}\leq \mathcal{H}^{k_i}(\mathcal{S}_i\cap B(x,r))	\leq \Lambda r^{k_i}$. Let $\Omega$ be a compact domain of $M^n$ with smooth boundary that contains $\mathcal{S}$. Then there is a positive harmonic function $G$ in $\Omega\setminus \mathcal{S}$ and a positive constant $C$ such  that 
\begin{enumerate}
	\item $G|_{\partial \Omega}=0$;
	\item Let $\rho(x):=\inf_{y\in\mathcal{S}}r(x,y)\leq d(x, \mathcal{S})$. Then for each $x$ satisfying
    $\rho(x)<\rho_0:=\frac{1}{2}\inf_{y_1\in \mathcal{S},y_2\in\partial\Omega}r(y_1,y_2)$, there holds
\begin{align*}
    G(x)\geq C\rho(x)^{2+k-n}, 
\end{align*}
where $k=\max\{k_1,\cdots, k_m\}$.   
\end{enumerate}

\end{proposition}

To prove Proposition \ref{prop:growth of hf2}, we need first to prove the following lemma.
\begin{lemma}\label{lem: C_G lower bound}
    Let $\mathcal{S},\Omega$ be as in Proposition \ref{prop:growth of hf2} and let $q\in\Omega\backslash\mathcal{S}$ be a fixed point. For each $p\in\mathcal{S}$, let $G_p$ be the singular harmonic function constructed in Proposition \ref{prop:existence of singular positive hf} that blows up at $p$ and equals $1$ at $q$. Then there exists $c_0>0$, such that $C_{G_p}\ge c_0$ for all $p\in\mathcal{S}$.
\end{lemma}
\begin{proof}
    By Harnack's inequality, there exists a small neighborhood $V$ of $q$ such that 
    \begin{align}\label{eq: 76}
        \lambda_0>G_p(x)>\lambda_0^{-1}>0 \mbox{ in }V
    \end{align}
    for all $p\in\mathcal{S}$, where $\lambda_0$ is a uniform constant. Assume the conclusion of the lemma is false, then we can select a sequence of $x_i\in \mathcal{S}$ such that the corresponding singular harmonic function $G_{x_i}$, denote by $G_i$ for simplicity, satisfies $C_{G_i}\to 0$. Using Proposition \ref{prop:capacity1}, we have
    \begin{align*}
        \int_{0\le G_i\le 2\lambda_0}|\nabla G_i|^2d\mu = 2C_{G_i}\lambda_0\to 0 \mbox{  as } i\to\infty.
    \end{align*}
    Let $\tilde{G}_i = \min\{G_i,2\lambda_0\}$. By the Sobolev inequality, we have
    \begin{align*}
       ( \int_{\Omega}\tilde{G}_i^{\frac{2n}{n-2}}d\mu)^{\frac{n-2}{2n}}\le C\cdot C_{G_i}^{\frac{1}{2}}\to 0.
    \end{align*}
    This shows $\|G_i\|_{L^{\frac{2n}{n-2}}}\to 0$, which contradicts \eqref{eq: 76}.
\end{proof}

\begin{proof}[Proof of Proposition \ref{prop:growth of hf2}]
For simplicity, we only consider the case \( m = 1 \) and \( k = k_1 \). The general case can be proved by summing up the singular harmonic functions constructed for each \( \mathcal{S}_{k_i} \). For each $r\in (0,\rho_0)$, let \( K \) be the maximal number such that there exists a set \( \{x_1, \dots, x_K\} \subset \mathcal{S} \) satisfying \( B(x_i, r) \cap B(x_j, r) = \emptyset \) for \( i \neq j \). Let \( \mathcal{F} \) denote the family of such balls. Under the assumptions \( \Lambda^{-1} \leq \mathcal{H}^k(\mathcal{S}) \leq \Lambda \) and \( \Lambda^{-1} r^{k} \leq \mathcal{H}^k(\mathcal{S} \cap B(x, r)) \leq \Lambda r^{k} \), it follows that
\begin{align}\label{eq: 77}
    K\le \Lambda^2 r^{-k}.
\end{align}
For each \( x_i \), we may define a positive singular harmonic function \( G_i \) in \( \Omega \setminus \{x_i\} \) as in Proposition \ref{prop:existence of singular positive hf}. We may further assume that
$$
G_i (q)=1,\quad \text{for a fixed $q\in \Omega\setminus \mathcal{S}$},
$$
and define
$$
w_r=\frac{1}{K}(G_1+G_2+\cdots+ G_K).
$$	
Then $w_r$ is a positive harmonic function in $\Omega\setminus \mathcal{S}$ and $w_r(q)=1$. 

For any $x\in \Omega\setminus \mathcal{S}$ with $\rho(x)<\rho_0$, let $J$ be the maximal number of balls $B(y, r)$ in $\mathcal{F}$ so that 
$$
\bigcup_{B(y, r)\in \mathcal{F}}B(y, r)\subset B(x, 2\rho(x)).
$$		
Without loss of generality, we may denote these balls by 
$\{B(x_1, r),\cdots, B(x_J, r)\}$,  then by the maximality of $J$, we know that 

$$
\bigcup_{1\leq i\leq J}(B(x_i, 2r)\cap \mathcal{S})\supset B(x, 2\rho(x))\cap \mathcal{S},
$$	
which implies 
\begin{align}\label{eq: 78}
    J\geq \Lambda^{-2}(\frac{\rho(x)}{r})^k.
\end{align}
Using Proposition \ref{prop:growth of hf1}, Lemma \ref{lem: C_G lower bound},
the inequalities \eqref{eq: 77} and \eqref{eq: 78}, we deduce that
\begin{align*}
    w_r(x)&\ge\frac{1}{K}\sum_{i=1}^J(C_{G_i}\Lambda^{-1})(2\rho(x))^{2-n}\ge\frac{J}{K}c_0\Lambda^{-1} (2\rho(x))^{2-n}\\
    &\ge2^{2-n}\Lambda^{-5}c_0 \rho(x)^{2+k-n} = C\rho(x)^{2+k-n}.
\end{align*}	
Finally, note that $w_r(q)=1$, let $r\rightarrow 0$, and by passing to a subsequence if necessary,  we may assume that $\{w_r\}$ converges to a positive harmonic function $G$ in $\Omega\setminus\mathcal{S}$ which satisfies $G(x)\geq C\rho(x)^{2+k-n}$. Thus, we finish the proof of Proposition \ref{prop:growth of hf2}.
\end{proof}
\subsubsection{Extension of the singular harmonic function to the total space}

$\quad$

In this subsection, we extend the singular harmonic function to the whole singular space. We will use arguments presented in \cite{Li04} and \cite{guo2024}.
\begin{proposition}\label{prop: extend the singular harmonic function}
    Let $(M,d,\mu)$ be as in Theorem \ref{thm:pmt with singularity4}. Then for each $l\in\mathbf{Z}_+$, there exists a singular harmonic function $G^{(l)}\in W^{1,2}_{loc}(M\backslash\mathcal{S}_l)\cap C^{\alpha}(M\backslash\mathcal{S}_l)\cap C^{\infty}(M\backslash \mathcal{S})$ that satisfies
    \begin{itemize}
        \item $G^{(l)}$ is weakly harmonic in $M\backslash\mathcal{S}_l$;
        \item $\lim_{x\in E, |x|\to\infty}G^{(l)}(x) = 0$;
        \item $G^{(l)}(x)\ge C_ld(x,\mathcal{S}_l)^{2+k_l-n}$ for  $d(x,\mathcal{S}_l)$   \ small enough;
        \item $\sup_{\partial E}G^{(l)}\leq C$ for some uniform $C$.
    \end{itemize}
\end{proposition}

\begin{proof}
For given $\mathcal{S}_l$, let $\mathcal{S}_l\subset\Omega_0\subset \Omega_1\subset\dots$ be a compact exhaustion of $M$, such that $\Omega_0\cap E = \emptyset$ and $\bar{\Omega}_0\cap\bar{E} = \partial E$.

\textbf{Step 1: } Construct singular harmonic functions $G_i$ on $\Omega_i$, such that $G_i(x)\ge G_j(x)$ for $i>j$ and $x\in\Omega_j$.

By  Proposition \ref{prop:growth of hf2}, there is a singular harmonic function $G_0\in W^{1,2}_{loc}(\Omega_0\backslash\mathcal{S}_l)\cap C^{\alpha}(\Omega_0\backslash\mathcal{S}_l)\cap C^{\infty}(\Omega_0\backslash \mathcal{S})$ satisfying
\begin{align}\label{eq: 52}
    G_0(x) \ge Cd^{2-n+k_l}(x,\mathcal{S}_l),
\end{align}
when $d(x,\mathcal{S}_l)$ is sufficiently small.

Denote the $\rho$-neighborhood of $U_l$ by $B_{\rho}(U_l)$. Since $\partial U_l$ lies in the smooth part of $M$, we can select $\rho$ to be sufficiently small such that $(B_{\rho}(U_l)\backslash U_l)\cap \mathcal{S} = \emptyset$. Denote $\chi$ to be a cut off function that satisfies $\chi = 1$ in $U_l$, $\chi = 0$ in  $M\backslash B_{\rho}(U_l)$. Due to Lemma \ref{lem: Poisson equation}, there exists $h_i\in W^{1,2}(\Omega_i)$ solving
\begin{equation}\label{eq: 53}	
\left\{
\begin{aligned}
 & -\Delta h_i = G_0\Delta\chi + 2\nabla G_0\cdot\nabla\chi  \quad \text{in $\Omega_i$}, \\
  &  h_i|_{\partial \Omega_i}=0.
\end{aligned}
\right.
\end{equation}
By Harnack inequality, $h_i\in C^{\alpha}$ and hence is continuous on $\Omega_i$. Define $G_i = \chi G_0 +h_i$, it is straightforward to see $G_i\in W^{1,2}_{loc}(\Omega_i\backslash\mathcal{S}_l)\cap C^{\alpha}(\Omega_i\backslash\mathcal{S}_l)\cap C^{\infty}(\Omega_i\backslash \mathcal{S})$ is a singular harmonic function on $\Omega_i$ that satisfies the conclusion of  Proposition \ref{prop:growth of hf2}.

\textbf{Claim: } If $i>j$, then $h_i\ge h_j$ in $\Omega_j$.

From Lemma \ref{lem: weak maximum principle} we obtain $G_i\ge 0$ in $\Omega_i$. Since $\chi$ is supported in $B_{2\rho}(\mathcal{S}_l)$, we have $h_i = G_i\ge 0$ on $\partial\Omega_j$. On the other hand, $h_j = 0$ on  $\partial\Omega_j$. In conjunction with \eqref{eq: 53} and Lemma \ref{lem: weak maximum principle}, we see the Claim is true.

As a direct corollary, we see $G_i(x)\ge G_j(x)$ for $i>j$ and $x\in\Omega_j$.

\textbf{Step 2: } Obtain a uniform bound on $G_i$.

Let $s_i = \sup_{\partial\Omega_0} G_i(x)$. 

\textbf{Claim: }$s_i$ is uniformly bounded.

Assume, for the sake of contradiction that $s_i\to\infty$. Let $\varphi_i = s_i^{-1}G_i$, then by the maximum principle and the fact that the difference between  $G_i$ and $G_0$ is a bounded continuous function, we have
\begin{align*}
    s_i^{-1}G_0\le \varphi_i\le s_i^{-1}G_0 +1.
\end{align*}
Passing to a subsequence, we may assume that $\varphi_i$ converges smoothly to a function $\varphi$ with $0\le \varphi\le 1$ in $M\backslash\mathcal{S}$. By the removable singularity Lemma \ref{lem: removable singularity}, $\varphi$ extends to a $W^{1,2}$ weak harmonic function in $M$. 

 It is well known that there exists a harmonic function $f$ on the AF end $E$ that satisfies $f|_{\partial E} = 1$ and $\lim_{|x|\to\infty} f(x) = 0$. Note $\sup_{\partial E}\varphi_i\leq1$.
From the weak maximum principle,  $0\le \varphi_i\le f$ in $E\cap\Omega_i$. 
 Hence, passing to a subsequence, we have $0\le \varphi\le f$ in $E$. Once again, using the weak maximum principle, the function
\begin{align*}
    M_i(r) = \sup_{\partial B_r(\mathcal{S}_l)} \varphi_i
\end{align*}
is non-increasing in $r$. By passing to a limit, we see that
\begin{align*}
    M(r) = \sup_{\partial B_r(\mathcal{S}_l)} \varphi
\end{align*}
is also non-increasing, thereby attaining its  positive maximum in $\mathcal{S}_l$. This contradicts the weak maximum principle since $\lim_{x\in E,|x|\to\infty}\varphi(x) = 0$.

Once the Claim is verified, the desired extension $G^{(l)}$ comes by taking a limit $G^{(l)} = \lim_{i\to\infty}G_i$.
\end{proof}

\begin{proposition}\label{lem: extend harmonic function 2}
     Let $(M,d,\mu)$ be as in Theorem \ref{thm:pmt with singularity4}. Then there exists $G\in C^{\infty}(M\backslash\mathcal{S})$ that satisfies
    \begin{itemize}
        \item $\Delta_gG = 0$ in $M\backslash\mathcal{S}$;
        \item $\lim_{x\in E, |x|\to\infty}G(x) = 0$;
        \item For each $l\in\mathbf{Z}_+$, there holds $G(x)\ge C_ld(x,\mathcal{S}_l)^{2+k_l-n}$
        \ \ for $d(x,\mathcal{S}_l)$   small enough.
    \end{itemize}
\end{proposition}
\begin{proof}
It is direct to see that the function
\begin{align}\label{eq: 65}
    G = \sum_{l=1}^{\infty}2^{-l}G^{(l)}
\end{align}
has the properties listed in the lemma.
\end{proof}

\section{Positive mass theorem for AF manifolds with arbitrary ends and PSC in the strong spectral sense}

        In this section, we will prove Theorem \ref{thm: PMT arbitrary end spectral}. We begin with the following proposition, which is crucial in our construction of the ALF manifolds.
       \begin{proposition}\label{prop: eq1-arbitrary end}
		Let $(M^n,g)$ be a smooth AF manifolds with arbitrary ends which is not necessarily complete, $E$ being its AF end of order $\tau>\frac{n-2}{2}$. Suppose $(M^n,g)$ has   $\beta$-scalar curvature no less than $h$ in the strong spectral sense for some $\beta>0$, where $h\in C_{-\tau-2}(E)$ is a non-negative smooth and integrable function on $M$.

		(1) If $h(p)>0$ for some point  $p\in M$, then for any smooth function $Q$ with $0\leq Q\le h$ and $Q(p)<h(p)$, there is a positive and smooth function $u$ on $M$ that satisfies 
		$$
		-\Delta u +\beta(R_g-Q) u = 0
		$$
		and $u\to 1$ as $x\to \infty$ in $E$.
		
		(2) If $h\equiv 0$, then one of the following happens:

		\begin{itemize}
			\item $R_g\equiv 0$.
			\item for any $0<\beta'<\beta$, there is a smooth positive nonconstant function $u$ on $M$ that satisfies
			$$
			-\Delta u +\beta'R_g u = 0
			$$
			and $u\to 1$  as $x\to \infty$ in $E$.
		\end{itemize}
	\end{proposition}
	\begin{proof}
		
	Let's prove the first statement.  In fact, let $\{\Omega_i\}$ be a sequence of exhausting compact domains of $M$ with smooth boundaries. We use $\partial_+ \Omega_i$ and $\partial_- \Omega_i$ to denote the boundary of $\Omega_i$ in $E$ and $M \setminus E$ respectively. Then, by our assumption that $\lambda_1(-\Delta_g+\beta(R_g-Q))\ge 0$ and thanks to the Fredholm alternative, we know that for each $\Omega_i$ there is a  positive and smooth function $u_i$ on $\Omega_i$ that satisfies
		\begin{equation}\label{eq4}
			\left\{
			\begin{aligned}
				&-\Delta u_i +\beta(R_g-Q) u_i=0 \quad \text{in $\Omega_i$},\\
				&u_i|_{\partial_+\Omega_i}=1,\\
				&u_i|_{\partial_-\Omega_i}=0.
			\end{aligned}
			\right.	
		\end{equation}	
		Without loss of the generality, for each $i$, we may assume $\partial_+\Omega_i$ is the coordinated sphere of $M \cap E$,  and $p\in \Omega_i$ for each $i$. We claim the following holds:
		
		{\bf Claim:} {\it There is a constant $\Lambda>0$ independent on $i$ so that $0<u_i(p)\leq \Lambda$.}

		Suppose the Claim  is false, then there exists a subsequence(still denoted by $i$) such that $\Lambda_i:=u_i(p)\rightarrow \infty$. Let $w_i:=\Lambda^{-1}_i u_i$, we have $w_i(p)=1$ and 
		\begin{equation}\label{eq5}
			\left\{
			\begin{aligned}
				&-\Delta w_i +\beta(R_g-Q) w_i=0 \quad \text{in $\Omega_i$},\\
				&w_i|_{\partial_+\Omega_i}=\Lambda^{-1}_i,\\
				&w_i|_{\partial_-\Omega_i}=0.
			\end{aligned}
			\right.	
		\end{equation}	
		
		By Harnack inequality, we know that by passing to a subsequence, we may assume $\{w_i\}$ locally and smoothly converges to a positive and smooth function $w$ satisfying
		$$
		-\Delta w +\beta(R_g-Q) w=0 \quad \text{in $M$}, \lim_{x\in E,|x|\to\infty}w(x) = 0
		$$
and $w(p)=1$. Denote $\Omega_{ij}:=\Omega_i\setminus\{(\Omega_i\setminus \Omega_j)\cap E\}$ for $i>j\gg1$. Let $\hat{w}_i\in C^{0,1}(M)$ be given by
\begin{equation}
\hat{w}_i=\left\{
			\begin{aligned}
&\Lambda_i^{-1}\quad \text{ in $E\backslash \Omega_i$,}\\
& w_i\quad \text{ in $\Omega_i$,}\\
&0 \quad \text{ in $M\setminus(\Omega_i \cup E)$.}\end{aligned}
\right.	\nonumber	
\end{equation}
 By applying the strong spectral condition to $\hat{w}_i$ we obtain
		\begin{equation}\label{eq: 15}
			\begin{split}
				\mathcal{E}(\hat{w}_i)  := &\int_{E\backslash\Omega_i}\beta|R_g-Q|\Lambda_i^{-2}d\mu_g+\int_{(\Omega_{i}\setminus\Omega_{j})\cap E }|\nabla_M w_i|^2+\beta(R_g-Q) w_i^2d\mu_g\\
				+&\int_{\Omega_{ij} }|\nabla_M w_i|^2+\beta(R_g-Q) w_i^2d\mu_g\\
                \ge& \int_{\Omega_{ij} }\beta( h-Q) w_i^2d\mu_g\\
				\ge& C_0 = C_0(M,g_{M},R_g,Q)>0.
			\end{split}
		\end{equation}
		In the last inequality we have used the fact $w_i(p) = 1$ and the Harnack inequality for $w_i$.

Next, we will verify that { \it for any $\epsilon>0$, there holds $\mathcal{E}(\hat{w}_i)<4\epsilon$ for $i\gg j\gg1$.}
        
	 Note that $R_g$ and $Q$ is integrable on $(M^n,g)$ and $\Lambda_i\to\infty$.
     Then for any $\epsilon>0$, we have
			\begin{align}\label{eq: 14}
			\int_{E\backslash\Omega_i}\beta|R_g-Q|\Lambda_i^{-2}d\mu_g<\epsilon
		\end{align}
		for all sufficiently large $i$. Moreover, by multiplying $w_i$ on both sides of \eqref{eq5} in domain $ \Omega_{ij}$ and integrating by parts, we see that
		\begin{align*}
			\int_{\Omega_{ij}}|\nabla w_i|^2+\beta(R_g-Q) w_i^2=\int_{\partial_+\Omega_{j}} w_i \frac{\partial w_i}{\partial \Vec{n}_g}d\sigma
		\end{align*}
		By Lemma \ref{lem: asymptotic behavior1} we have
        \begin{align}\label{eq: 60}
		    \sup_{\partial B_r}r^{-1}|w|+|\nabla w|\leq C r^{-\tau-1+\epsilon}.
		\end{align}
        Here $B_r$ denotes the coordinate ball with radius $r$ in $E$.
        Using the smooth convergence $w_i\to w$, we know that for any fixed coordinate sphere $\partial B_r$ and $i$ sufficiently large, it holds
		\begin{align}\label{eq: 51}
		    \sup_{\partial B_r}r^{-1}|w_i|+|\nabla w_i|\leq C r^{-\tau-1+\epsilon}.
		\end{align}
		where $C$ is a constant independent on $i$ and $r$. Therefore, for fixed sufficiently large $j$ and $i\gg j$, there holds
        \begin{align}\label{eq: 19}
        \int_{\Omega_{ij}}|\nabla w_i|^2+\beta(R_g-Q) w_i^2<\epsilon.
        \end{align}

        \begin{figure}
    \centering
    \includegraphics[width = 12cm]{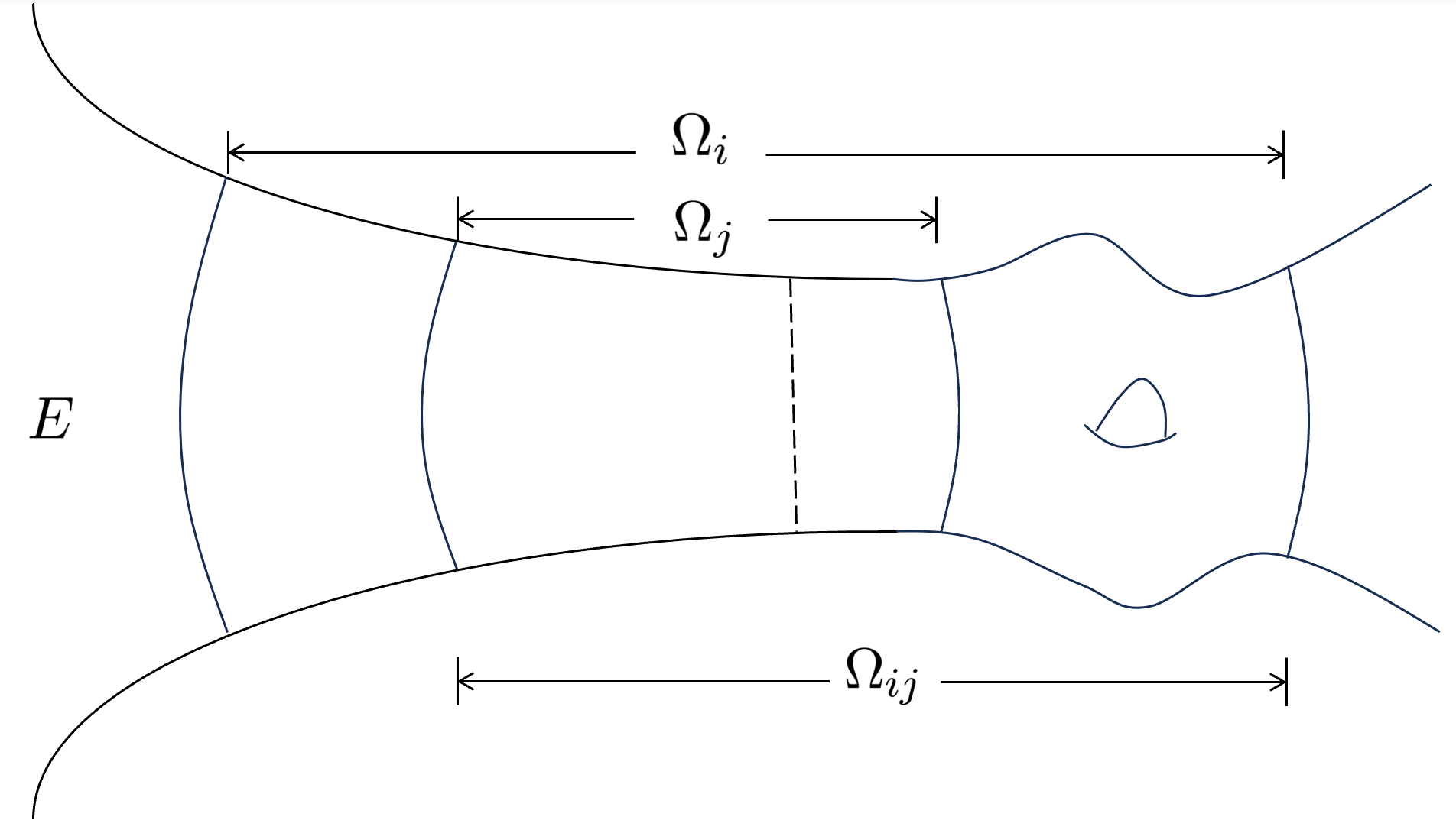}
    \caption{A schematic illustration of the proof of Proposition \ref{prop: eq1-arbitrary end}}
    \label{f3}
\end{figure}

		Finally, it suffice to show that for sufficiently large $j$ and $i\gg j$, there holds
		\begin{align}\label{eq: 18}
			\int_{(\Omega_{i}\setminus\Omega_{j})\cap E }|\nabla w_i|^2+\beta(R_g-Q) w_i^2d\mu_g<2\epsilon.
		\end{align}
		Set $\tilde{w}_i = w_i-\Lambda_i^{-1}$. Then 
		\begin{equation}\label{eq: 16}
			\left\{
			\begin{aligned}
				&-\Delta\tilde{w}_i+\beta(R_g-Q)\tilde{w}_i = -\beta(R_g-Q)\Lambda_i^{-1}\mbox{ in }\Omega_i,\\
				&\tilde{w}_i = 0\mbox{ on }\partial\Omega_i^+.
			\end{aligned}
			\right.	
		\end{equation}
Let $\eta$ be a cut off function with  $|\nabla\eta|<1$ and 			
 \begin{equation}
\eta(x)=\left\{
			\begin{aligned}
&0\quad \text{ for $x$ with  $d(x,E\backslash\Omega_j)>2$,}\\
&1\quad \text{ for $ x \in E\backslash\Omega_j$.}\\
\end{aligned}
\right.	\nonumber	
\end{equation}		
Without loss of generality, in the following, we assume $\eta(x)=0$ for all $ x\in \Omega_{j-1}$.		
Multiplying $\eta^2\tilde{w}_i$ on both sides of \eqref{eq: 16} and using
		\[
			\divv(\eta^2\tilde{w}_i\nabla\tilde{w}_i) 
            = \eta^2\tilde{w}_i\Delta\tilde{w}_i+\nabla(\eta^2\tilde{w}_i)\nabla\tilde{w}_i
			=\eta^2\tilde{w}_i\Delta\tilde{w}_i+|\nabla(\eta\tilde{w}_i)|^2-|\nabla\eta|^2\tilde{w}_i^2,
		\]
		we deduce that
		\begin{align*}
			&\int_{(\Omega_{i}\setminus\Omega_ {j-1})\cap E} |\nabla(\eta \tilde{w}_i)|^2+\beta\int_{(\Omega_{i}\setminus\Omega_{ j-1})\cap E }(R_g-Q)\eta^2\tilde{w}_i^2-\int_{(\Omega_{i}\setminus\Omega_{ j-1})\cap E }|\nabla\eta|^2\tilde{w}_i^2\\ 
            =& -\beta\Lambda_i^{-1}\int_{(\Omega_{i}\setminus\Omega_{ j-1})\cap E }(R_g-Q)\eta^2\tilde{w}_i.
		\end{align*}
		Therefore, 
		\begin{equation}\label{eq: 17}
		    \begin{split}
		        &\int_{(\Omega_{i}\setminus\Omega_{j})\cap E }|\nabla w_i|^2+\beta(R_g-Q) w_i^2d\mu_g\\\leq&\int_{(\Omega_{i}\setminus\Omega_{ j-1})\cap E }|\nabla(\eta \tilde{w}_i)|^2+\beta(R_g-Q) (\tilde{w_i}+\Lambda_i^{-1})^2d\mu_g\\
            \le& \int_{(\Omega_{i}\setminus\Omega_{ j-1})\cap E }|\nabla\eta|^2\tilde{w}_i^2
            +3\beta\int_{(\Omega_{i}\setminus\Omega_{ j-1})\cap E }|R_g-Q|(\tilde{w}_i^2 + 1),
		    \end{split}
		\end{equation}
	where we have used $\Lambda_i\geq 1$ for $i\gg1$. Since $\nabla\eta$ is supported on a compact set around $\partial\Omega_j$, using \eqref{eq: 51}, we see the first term on the last line of \eqref{eq: 17} can be bounded by $\epsilon$ for  $i\gg1$. 
    
    To obtain similar estimate for the last term, we only need to show that the $C^0$ norm of $\tilde{w}_i$ in $(\Omega_{i}\setminus\Omega_{ j-1})\cap E$ is uniformly bounded.  We first observe that $C_1 = \sup_{\partial E} \tilde{w}_i<+\infty$. Let $C_S$ be the Sobolev constant of $(E, g)$. By our assumption, 
we can choose some fixed $R\gg1$ such that
\begin{equation}\label{eq: intergral for R-Q}
  \beta C_S(\int_{ E\setminus B_R(O) }|R_g-Q|^{\frac{n}{2}})^{\frac{2}{n}}\leq\frac{1}{4}.  
\end{equation}
Let $0\leq\varphi\leq 1$ be another cut off function supported in $E$ with  $|\nabla\varphi|<2$
and 		
\begin{equation}
\varphi(x)=\left\{
			\begin{aligned}
&1\quad \text{ for $x\in E$ with  $|x|>R+1$,}\\
&0\quad \text{ for $x\in E$ with  $|x|\leq R$.}\\
\end{aligned}
\right.	\nonumber	
\end{equation}	
Multiplying $\varphi^2\tilde{w}_i$ on both sides of \eqref{eq: 16} and integrating by parts gives
    \begin{align*}
			\int_{ E} |\nabla(\varphi\tilde{w}_i)|^2
            +\beta\int_{ E }(R_g-Q)\varphi^2\tilde{w}_i^2-\int_{ E }|\nabla\varphi|^2\tilde{w}_i^2
            = -\beta\Lambda_i^{-1}\int_{ E }(R_g-Q)\varphi^2\tilde{w}_i.
		\end{align*}
    By Sobolev inequality and Holder inequality,
    \begin{equation}
    \begin{split}\label{eq: Sobolov ineq}
        (\int_{ E} |\varphi\tilde{w}_i|^{\frac{2n}{n-2}})^{\frac{n-2}{n}}\leq&C_S\int_{ E} |\nabla(\varphi\tilde{w}_i)|^2\\
    \leq&\beta C_S \Lambda_i^{-1}(\int_{  E\setminus B_R(O) }|R_g-Q|^{\frac{2n}{n+2}})^{\frac{n+2}{2n}}
    (\int_{ E} |\varphi\tilde{w}_i|^{\frac{2n}{n-2}})^{\frac{n-2}{2n}}\\
    &+C_S\int_{ B_{R+1}(O)\backslash B_R(O) }|\nabla\varphi|^2\tilde{w}_i^2
    +\beta  C_S(\int_{ E\setminus B_R(O) }|R_g-Q|^{\frac{n}{2}})^{\frac{2}{n}}
    (\int_{ E} |\varphi\tilde{w}_i|^{\frac{2n}{n-2}})^{\frac{n-2}{n}}.
    \end{split}
    \end{equation}
   By applying Harnack's inequality to \eqref{eq5} satisfied by $w_i$, we obtain $\tilde{w}_i$ is uniformly bounded on the support of $\nabla\varphi$. Combined with \eqref{eq: intergral for R-Q}, \eqref{eq: Sobolov ineq} and using Cauchy-Schwarz inequality yields
   \begin{equation}
       (\int_{ E} |\varphi\tilde{w}_i|^{\frac{2n}{n-2}})^{\frac{n-2}{2n}}\leq C_0.
   \end{equation}
    for some uniform $C_0$ independent of $i$.
    Now we are able to use $L^{\frac{2n}{n-2}}$ estimate and Moser iteration to obtain the uniform $C^0$ bound for $\tilde{w}_i$ in $E\setminus B_R(O)$. Combined with the fact that $R_g-Q$ is $L^1$ integrable we see \eqref{eq: 18} is true.  Thus,   the inequality $\mathcal{E}(\hat{w}_i)<\epsilon$  can be deduced from \eqref{eq: 14}, \eqref{eq: 19} and \eqref{eq: 17}.  Combining  this with \eqref{eq: 15}, we obtain a contradiction, which implies Claim.  Once Claim  is verified, in conjunction with Harnack inequality, we complete the proof of  the first statement of Proposition \ref{prop: eq1-arbitrary end}.

     Now, we are going to prove the second statement of Proposition \ref{prop: eq1-arbitrary end}. We assume $(M,g)$ is not scalar flat. Note that the condition of $\beta$-scalar non-negative in the strong spectral sense is equivalent to
        \begin{align}\label{eq: 32}
            \int_M |\nabla\phi|^2 +\beta'\phi^2d\mu_g\ge \int_M(1-\frac{\beta'}{\beta})|\nabla\phi|^2d\mu_g
        \end{align}
        for all locally Lipschitz function $\phi$ which is either compact supported or asymptotically constant in the sense of Definition \ref{defn: strong test function}. Let $\Omega_i$ and $\Omega_{ij}$ be as in (1), then we can  find a smooth function $u_i$ solving following equation:
        \begin{equation}\label{eq: 31}
			\left\{
			\begin{aligned}
				&-\Delta u_i +\beta' R_{ g} u_i=0 \quad \text{in $\Omega_i$},\\
				&u_i|_{\partial_+\Omega_i}=1,\\
				&u_i|_{\partial_-\Omega_i}=0.
			\end{aligned}
			\right.	
		\end{equation}
Next, following the same arguments as those in the proof of the first statement,  we can show that $u_i$ is locally uniformly bounded. 

In fact,  if  $u_i$ is not locally uniformly bounded we can define $w_i$ as before, then $w_i$ converges to a function $w$ solving
 \begin{align}\label{eq: 34}
            -\Delta w+\beta' R_{ g} w = 0
\end{align}
in $M$ and 
$$
\lim\limits _{ x\to\infty, x\in E}w(x) = 0
$$
with   $w(p)=1$ for some fixed $p\in M$. It follows from \eqref{eq: 32} and the same argument as before, we get  \begin{equation}\label{eq: 33}
    \begin{split}
&\int_{M\backslash\Omega_i}\beta'R\Lambda_i^{-2}d\mu_g + \int_{(\Omega_i\backslash\Omega_j)\cap E} |\nabla w_i|^2+\beta'Rw_i^2 d\mu_g
    +\int_{\Omega_{ij}} |\nabla w_i|^2 +\beta'Rw_i^2 d\mu_g\\
    \ge& \int_{\Omega_{ij}} (1-\frac{\beta'}{\beta})|\nabla w_i|^2d\mu_g.
    \end{split}
\end{equation}
By the same argument used in the proof of the first statement, we see that the left-hand side of \eqref{eq: 33} tends to zero as $i\gg j$ tends to infinity. Therefore, the right-hand side of \eqref{eq: 33} tends to zero. Passing to a limit we obtain $w\equiv \mbox{const}$. Plugging into \eqref{eq: 34} we obtain $R_g\equiv 0$, which contradicts our assumption.
 Thus, $\{u_i\}$ has a uniform bound.
 
Let $i\to\infty$ we see $u_i$ converges to the desired function $u$ satisfying the equation in (2) of Proposition \ref{prop: eq1-arbitrary end}.	
	\end{proof}

To apply Proposition \ref{prop: PMT for S^1 symmetric ALF}, we need to construct ALF manifolds with nonnegative scalar curvature. To this end, we consider two separate cases.

    \textbf{Case1: $(M^n, g)$ has $\beta$-scalar curvature no less than $h$ in the strong spectral sense for some $\beta\geq\frac{1}{2}$, where $h$ is a nonnegative smooth function with $h(p)>0$ for some $p\in M$.}
    
    Without loss of generality, we may assume $h\in C_{-2-\tau}(E)$.  By Remark \ref{re: NNSC}, $(M^n, g)$ has $\frac{1}{2}$-scalar curvature no less than $h$. By Proposition 
    \ref{prop: eq1-arbitrary end}, there exists some positive function $u$ satisfying
    \[
    Lu = -\Delta u +\frac{1}{2}R_g u -\frac{1}{4}hu=0\quad \text{in}\ \  M
    \]
     with $u\to 1$ as $x\to\infty$ in $E$.
     Let $ E\subset V_1\subset V_2\subset\dots$ be a sequence of exhausting domains of $M$.
     Since $\lambda_1(-\Delta_g+\frac{1}{2}(R_g-h))\ge 0$, the solution $v_i$ of the following equation exists:
	\begin{equation}\label{eq: 35}
		\left\{
		\begin{aligned}
			&Lv_{i} = -\Delta v_{i} +\frac{1}{2}R_g v_{i} -\frac{1}{4}hv_{i}=0 \quad \text{in $V_i$},\\
			&\lim\limits _{ x\to\infty, x\in E} v_{i}(x) = 1,\\
			&v_{i}|_{\partial V_i}=0.
		\end{aligned}
		\right.	
	\end{equation}
	By the maximum principle, we have $v_{i}<u$ for all $x\in M$ and  $v_{j}>v_{i}$ on $V_i$ for each $j>i$. Therefore, $v_{i}$ converges monotonically to an function $v$ in $C^2$ sense, where $v\in C^2(M)$ solves $Lv = 0$ and satisfies  $\lim\limits _{ x\to\infty, x\in E}v(x) = 1$. Moreover,  near the infinity of AF end $E$ there holds
	\begin{equation}\label{eq: asymp behaviour for $v_i$}
		|v-1|+|x||\partial v| = O(|x|^{-\tau+\epsilon}),
	\end{equation}
	and hence, for any fixed coordinated ball $B_r$ in AF end $E$, and $i$ large enough,  we have
	\begin{equation}\label{eq: asymp behaviour for $v$}
	    \sup_{\partial B_r}(|v_i-1|+r|\partial v_i|) = O(r^{-\tau+\epsilon}).
	\end{equation}

	Next, we construct a sequence of ALF manifolds as follows
	\begin{equation}\label{eq: 63}
		\begin{split}
			&(\hat{M}_{i},\hat{g}_{i}) = (V_i\times \mathbf{S}^1, g + v_{i}^2ds^2),\\
			&(\hat{M},\hat{g}) = (M\times \mathbf{S}^1, g + v^2ds^2).
		\end{split}
	\end{equation}
	Then we have 
    \[
    R_{\hat{g}}=R_{\hat{g}_i}=R_g-2v_i^{-1}\Delta v_i=\frac{h}{2}\geq0 \ \ \text{in}\ \  V_i.
    \]
    We remind the readers here $(\hat{M}_{i},\hat{g}_{i})$ is not necessarily complete. We use $\hat{E} = E\times \mathbf{S}^1$ to denote the distinguished end in these ALF manifolds. By \eqref{eq: asymp behaviour for $v$} we conclude that $\hat{E}$ is an ALF end of order $\tau-\epsilon$.
	
	The following lemma estimates the mass of the ALF end $(\hat{M}_{i},\hat{g}_{i},\hat{E})$.
	
	\begin{lemma}\label{lem: mass decay ALF}
		\begin{align}\label{eq: 42}
			m_{ADM}(\hat{M}_{i},\hat{g}_{i},\hat{E}) \le (n-1)m_{ADM}(M,g,E) - \frac{1}{4\omega_{n-1}}\int_{V_i}hv_i^2d\mu_{g}.
		\end{align} 
	\end{lemma}
	\begin{proof}
		
		By the definition of the ADM mass for ALF manifolds, we have
		\begin{equation}\label{eq: 36}
			\begin{split}
				m_{ADM}(\hat{M}_{i}, \hat{g}_i, \hat{E}) = &\frac{1}{4\pi\omega_{n-2}}\lim_{\rho\to\infty}\int_{S^{n-1}(\rho)\times \mathbf{S}^1}(\partial_k(\hat{g}_i)_{kj}-\partial_j (\hat{g}_i)_{aa})\nu^jd\sigma_xds\\
				=&\frac{1}{4\pi\omega_{n-1}}\lim_{\rho\to\infty}\int_{S^{n-1}(\rho)\times \mathbf{S}^1}(\partial_k(\hat{g}_i)_{kj}-\partial_j (\hat{g}_i)_{kk}-\partial_jv_i^2)\nu^jd\sigma_xds\\
				=&\frac{1}{2\omega_{n-1}}\lim_{\rho\to\infty}\int_{S^{n-1}(\rho)}(\partial_k(\hat{g}_i)_{kj}-\partial_j (\hat{g}_i)_{kk}-\partial_jv_i^2)\nu^jd\sigma_x\\
				=&(n-1)m_{ADM}(M,g,E)-\frac{1}{\omega_{n-1}}\lim_{\rho\to\infty}\int_{S^{n-1}(\rho)}v_i\frac{\partial v_i}{\partial \Vec{n}_{euc}}d\sigma_x\\
				=&(n-1)m_{ADM}(M,g,E)-\frac{1}{\omega_{n-1}}\lim_{\rho\to\infty}\int_{S^{n-1}(\rho)}v_i\frac{\partial v_i}{\partial \Vec{n}_g}d\sigma_g.\\
			\end{split}
		\end{equation}
		
		We claim the last term is well defined. In fact, for $\rho_1>\rho_2$, we have
		\begin{align*}
			&\int_{S^{n-1}(\rho_1)}v_i\frac{\partial v_i}{\partial \Vec{n}_g}d\sigma_g - \int_{S^{n-1}(\rho_2)}v_i\frac{\partial v_i}{\partial \Vec{n}_g}d\sigma_g\\
			= &\int_{B_{\rho_1}\backslash B_{\rho_2}} (|\nabla v_i|^2+v_i\Delta v_i)d\mu_g\\
			= &\int_{B_{\rho_1}\backslash B_{\rho_2}} (|\nabla v_i|^2+ \frac{1}{2}R_g v_i^2 -\frac{1}{4}hv_i^2)d\mu_g.
		\end{align*}
		Since $|v_{i}(x)-1|+|x||\partial v_{i}(x)| = O(|x|^{-\tau+\epsilon})$  for any $x\in B_{\rho_1}\backslash B_{\rho_2}$ and $R_g,h\in L^1(E)$, we see the last term in \eqref{eq: 36} is well defined.
		
		By the strong spectral PSC condition and \eqref{eq: 35}, we calculate
		\begin{equation}\label{eq: 37}
			\lim_{\rho\to\infty}\int_{S^{n-1}(\rho)}v_i\frac{\partial v_i}
			{\partial \Vec{n}_g}d\mu_{g}
			=\int_{V_i}|\nabla v_i|^2+(\frac{1}{2}R_g-\frac{1}{4}h)v_i^2d\mu_{g}
			\ge \int_{V_i}\frac{1}{4}hv_i^2d\mu_{g}.
		\end{equation}
		The conclusion then follows immediately.
	\end{proof}
	Next, we consider 

    \textbf{Case 2: $(M^n, g)$ has $\beta$-scalar curvature nonnegative  in the strong spectral sense for some $\beta>\frac{1}{2}$.}

    By Proposition 
    \ref{prop: eq1-arbitrary end}, there exists some positive function $u$ satisfying
    \[
    \mathcal{L}u = -\Delta u +\frac{1}{2}R_g u =0\quad \text{in}\ \  M
    \]
     with $u\to 1$ as $x\to\infty$ in $E$.
     Using the similar argument as before, we can solve
		\begin{equation}\label{eq: 38}
			\left\{
			\begin{aligned}
				&\mathcal{L}v_{i} = -\Delta v_{i} +\frac{1}{2}R_g v_{i} = 0 \quad \text{in $V_i$},\\
				&\lim\limits _{ x\to\infty, x\in E} v_{i}(x) = 1,\\
				&v_{i}|_{\partial V_i}=0,
			\end{aligned}
			\right.
		\end{equation}
     where $ E\subset V_1\subset V_2\subset\dots$ is a  sequence of exhausting domains of $M$.
 Also,  $v_{i}$ converges monotonically to a function $v$ in $C^2$ sense, where $v\in C^2(M)$ solves $\mathcal{L}v = 0$ and satisfies  $\lim\limits _{ x\to\infty, x\in E}v(x) = 1$. Moreover,  near the infinity of AF end $E$, $v_i$ and $v$ satisfy similar asymptotic behaviors as 
\eqref{eq: asymp behaviour for $v_i$} and \eqref{eq: asymp behaviour for $v$}
Let $(\hat{M}_{i},\hat{g}_{i},\hat{E})$ and $(\hat{M},\hat{g},\hat{E})$  be constructed as in \eqref{eq: 63}. Then
    \[
    R_{\tilde{g}}=R_{\tilde{g}_i}=R_g-2v_i^{-1}\Delta u_i=0 \ \ \text{in}\ \  V_i.
    \]
  Compared with Lemma \ref{lem: mass decay ALF}, we have
  \begin{lemma}\label{lem: mass decay ALF 2}
        
		\begin{align}\label{eq: 39}
			m_{ADM}(\hat{M}_{i},\hat{g}_{i},\hat{E}) \le (n-1)m_{ADM}(M,g,E) - \frac{2\beta-1}{2\beta\omega_{n-1}}\int_{V_i}|\nabla v_i|^2d\mu_{g}.
		\end{align} 
	\end{lemma}
    \begin{proof}
        Multiplying $v_i$ on both sides of \eqref{eq: 38} and integrating by parts gives
        \begin{equation}\label{eq: integral}
			\lim_{\rho\to\infty}\int_{S^{n-1}(\rho)}v_i\frac{\partial v_i}
			{\partial \Vec{n}_g}d\mu_{g}
			=\int_{V_i}|\nabla v_i|^2+\frac{1}{2}R_gv_i^2d\mu_{g}
			\ge \int_{V_i}(1-\frac{1}{2\beta})|\nabla v_i|^2d\mu_{g},
		\end{equation}
        where we have used \eqref{eq: 32}. Combining \eqref{eq: 36} with \eqref{eq: integral}, we get the desired estimate.
    \end{proof}
    The next Lemma considers the convergence of the ADM mass of $(\hat{M}_i,\hat{g}_i,\hat{E})$ to that of $(\hat{M},\hat{g},\hat{E})$.

    \begin{lemma}\label{lem: convergence ADM mass on ALF}
    Let $(\hat{M}_{i}, \hat{g}_{i}, \hat{E})$  and  $(\hat{M},\hat{g},\hat{E})$ be as in  \eqref{eq: 63}  with $v_i$ being the solution of \eqref{eq: 35} or \eqref{eq: 38}. In both cases, we have
        $$\lim_{i\to\infty}m_{ADM}(\hat{M}_{i}, \hat{g}_{i}, \hat{E}) = m_{ADM}(\hat{M},\hat{g},\hat{E}).$$
    \end{lemma}
    \begin{proof}
        By equations satisfied by $v_{i}$ and $v$ and using a standard Green's function argument, we have $v_{i}\to v$ in $C^0_{-\tau + \epsilon}(E)$. By Schauder estimates we obtain $v_{i}\to v$ in $C^2_{-\tau + \epsilon}(E)$, which shows that $\hat{g}_{i}\to \hat{g}$ in $C^2_{-\tau + \epsilon}(E)$.

        By the calculation of  (4.2) in \cite{Bar86}, we have
        \begin{align*}
            R(\hat{g}_{i})|\hat{g}_{i}|^{\frac{1}{2}} = \partial_p\left((\hat{g}_{i})_{pq,q}-(\hat{g}_{i})_{qq,p}+Q_1(\hat{g}_{i})\right)+Q_2(\hat{g}_{i}),
        \end{align*}
        where
        \begin{align*}
            &g_{pq,q}-g_{qq,p}+Q_1(g) = |g|^{\frac{1}{2}}g^{pq}(\Gamma_q-\frac{1}{2}\partial_q(log|g|)),\\
            &Q_2(g)|g|^{-\frac{1}{2}} = -\frac{1}{2}g^{pq}\Gamma_p\partial_q(log|g|)+g^{pq}g^{kl}g^{rs}\Gamma_{pkr}\Gamma_{qls},\\
            &\Gamma_{ijk} = \frac{1}{2}(g_{jk,i}+g_{ik,j}-g_{ij,k})\mbox{ and }\Gamma_k = g^{ij}\Gamma_{ijk}.
        \end{align*}
        In our case, it is straightforward to see $Q_1(\hat{g}_{i}) = O(r^{-2\tau-1+2\epsilon})$, $Q_2(\hat{g}_{i}) = O(r^{-2\tau-2+2\epsilon})$, and
        \begin{align*}
            &m_{ADM}(\hat{M}_{i}, \hat{g}_{i}, \hat{E})\\
            =&\frac{1}{4\pi\omega_{n-1}}\lim_{\rho\to\infty} \int_{S^{n-2}(\rho)\times \mathbf{S}^1}(\hat{g}_i)_{pq,q}-(\hat{g}_{i})_{qq,p}d\sigma_xds\\
            =&\frac{1}{4\pi\omega_{n-1}}\int_{S^{n-2}(\rho_0)\times \mathbf{S}^1}(\hat{g}_{i})_{pq,q}-(\hat{g}_{i})_{qq,p}+Q_1(\hat{g}_{i})d\sigma_xds +\frac{1}{2\omega_{n-1}}\int_{\rho>\rho_0}R_{\hat{g}_{i}}|\hat{g}_{i}|^{\frac{1}{2}}-Q_2(\hat{g}_{i})dx.
        \end{align*}
        By the dominate convergence theorem we have
        \begin{align*}
            &(\hat{g}_{i})_{pq,q}-(\hat{g}_{i})_{qq,p}+Q_1(\hat{g}_{i})\to \hat{g}_{pq,q}-\hat{g}_{qq,p}+Q_1(\hat{g}_{}),\\
            &\int_{\rho>\rho_0}Q_2(\hat{g}_{i})dx\to\int_{\rho>\rho_0}Q_2(\hat{g}_{})dx \mbox{ as }i\to\infty.
        \end{align*}
        From \eqref{eq: 35} and \eqref{eq: 38}, in both cases we have $R_{\hat{g}_{i}} = R_{\hat{g}}$ in $E$, which gives the desired conclusion.
        \end{proof}
        As the final step of the preparation, we need the following conformal deformation theorem that applies to the strong spectral PSC condition.

       \begin{lemma}\label{lem: make a point strictly positive}
    Let $(M^n,g)$ be a smooth AF manifold with arbitrary ends and let $E$ be its AF end. Suppose $(M^n,g)$ has negative mass and $\frac{1}{2}$-scalar curvature non-negative  in the strong spectral sense. Then there exists $\varphi\in C^{\infty}(M)$, such that $(M^n,g_1:=\varphi^{\frac{4}{n-2}}g)$ is also a complete manifold with the AF end $E$ and negative mass, and for any neighborhood $\mathcal{U}$ of $E$ such that $\mathcal{U}\Delta E$ is compact, any $\phi\in C^1(\mathcal{U})$ with 
    \[
    \phi|_{\partial \mathcal{U}}=0, \ \ \lim_{x\to\infty,x\in E}\phi = 1\ \  \text{and} \ \ 
    \phi-1\in W^{1,2}_{\frac{2-n}{2}}(\mathcal{U}),
    \]
    it holds
     \begin{align}\label{eq: 30}
        \int_{M}|\nabla_{g_1}\phi|^2+\frac{1}{2}R_{g_1}\phi^2d\mu_{g_1}\ge \int_{M}\frac{1}{2}h\phi^2d\mu_{g_1}
    \end{align}
    for some $h\in C^{\infty}(M)$ satisfying $h\ge 0$ everywhere and $h(p)>0$ at some point $p$.
\end{lemma}

\begin{proof}
    We first construct a positive function $u\in C^{\infty}(M)$ such that
    \begin{equation}\label{eq: 27}
			\left\{
			\begin{aligned}
				&\Delta_gu\le 0 \text{ for all $x\in M$},\\
				&\Delta_gu(p)<0 \text{ at some point $p\in M$},\\
				&\lim\limits _{x\to\infty, x\in E} u(x) = 1.
			\end{aligned}
			\right.	
		\end{equation}
    Choose $c_S$ to be the constant that depends on $E$ as in \cite[Lemma 2.1]{Zhu23}. Then, we can construct a nonpositive function $f$ with compact support in $E$, such that $f(p)<0$ at some point $p\in E$, and
    \begin{align*}
        (\int_E|f|^{\frac{n}{2}})^{\frac{2}{n}}\le\frac{c_S}{2}.
    \end{align*}
    By \cite[Proposition 2.2]{Zhu23}, there is a positive function $u\in C^{\infty}(M)$ that satisfies $\Delta_gu = fu\le 0$ for all $x\in M$ and $\Delta_gu(p) = f(p)u(p)<0$. Moreover, From Lemma \ref{lem: asymptotic behavior2} we know that $u$ has the expansion 
    \begin{align}\label{eq: 28}
        u(x) = 1+A|x|^{2-n} + o(|x|^{2-n}),
    \end{align}
    where
    \begin{align}\label{eq: 29}
        A = -\frac{1}{(n-2)\omega_{n-1}}\int_E fud\mu_g >0.
    \end{align}
    It follows that $u$ satisfies \eqref{eq: 27}.

    Let $\varphi_{\epsilon} = \frac{1+\epsilon u}{1+\epsilon}$ with $\varepsilon>0$ to be determined. Then $(M^n,g_{\epsilon}) = (M^n,\varphi_{\epsilon}^{\frac{4}{n-2}}g)$ is complete. For any neighborhood $\mathcal{U}$ of $E$ such that $\mathcal{U}\Delta E$ is compact, any $\phi\in C^1(\mathcal{U})$ with 
    \[
    \phi|_{\partial \mathcal{U}}=0,\ \ \lim_{|x|\to\infty,x\in E}\phi = 1 
    \ \ \text{and}\ \ \phi-1\in W^{1,2}_{\frac{2-n}{2}}(\mathcal{U}),
    \]
    we calculate
    \begin{align*}
        &\int_M|\nabla_{g_\epsilon}\phi|^2+\frac{1}{2}R_{g_\epsilon}\phi^2d\mu_{g_\epsilon}\\
        =&\int_M \varphi_{\epsilon}^2|\nabla_{g}\phi|^2+\frac{1}{2}\varphi_{\epsilon}(R_{g}\varphi_{\epsilon}-\frac{4(n-1)}{n-2}\Delta_g\varphi_{\epsilon})\phi^2d\mu_{g}\\
        =&\int_M |\nabla_{g}(\varphi_{\epsilon}\phi)|^2+\frac{1}{2}R_g(\varphi_{\epsilon}\phi)^2 d\mu_g- \lim_{\rho\to\infty} \int_{\partial B_{\rho}}\phi^2\varphi_{\epsilon}\frac{\partial \varphi_{\epsilon}}{\partial \Vec{n}}d\sigma_g -\int_M \frac{2(n-1)}{n-2}\varphi_{\epsilon}\phi^2\Delta_g \varphi_{\epsilon}d\mu_g\\
        \ge& -\int_M \frac{2(n-1)}{n-2}\phi^2\varphi_{\epsilon}^{-\frac{n+2}{n-2}}\Delta_g \varphi_{\epsilon}d\mu_g,
    \end{align*}
    where in the last inequality we have used that $(M^n,g)$ has $\frac{1}{2}$-scalar curvature
    non-negative  in strong spectral condition and \eqref{eq: 29}. It follows from \eqref{eq: 27} that $(M^n,g_1:=g_{\varepsilon})$ satisfies \eqref{eq: 30}. Furthermore, when $\epsilon$ is sufficiently small, we have
    \begin{align*}
        m_{ADM}(M^n,g_1,E) 
        =&m_{ADM}(M^n,g,E)+\frac{2A\epsilon}{1+\epsilon} < 0.
    \end{align*}
    This completes the proof.
\end{proof}

\begin{proof}[Proof of Theorem \ref{thm: PMT arbitrary end spectral}]

    Let us  assume $h(p)>0$ first. By Proposition \ref{prop: eq1-arbitrary end}, we can solve \eqref{eq: 35} and construct the ALF manifolds as described in \eqref{eq: 63}. Without loss of generality, we assume $p\in V_1$. It follows from Lemma \ref{lem: mass decay ALF}  and the monotonicity of $v_i$ with respect to $i$ that
    \begin{align*}
        (n-1)m_{ADM}(M,g,E)\ge& m_{ADM}(\hat{M}_{i},\hat{g}_{i},\hat{E})+\frac{1}{4\omega_{n-1}}\int_{V_i}hv_i^2d\mu_{g}\\
        \ge &m_{ADM}(\hat{M}_{i},\hat{g}_{i},\hat{E})+\frac{1}{4\omega_{n-1}}\int_{V_1}hv_1^2d\mu_{g}\\
        \ge &m_{ADM}(\hat{M}_{i},\hat{g}_{i},\hat{E})+\sigma_0,
    \end{align*}
    where $\sigma_0>0$ is a positive number. The conclusion  then follows from Lemma \ref{lem: convergence ADM mass on ALF} and the positive mass theorem for ALF manifolds (Proposition \ref{prop: PMT for S^1 symmetric ALF} and \cite[Theorem 1.8]{CLSZ2021}).

     For the general case, we can use Lemma \ref{lem: make a point strictly positive} to reduce the case that $h(p)>0$. To be more precise, if $m_{ADM}(M,g,E)<0$, then by utilizing Lemma \ref{lem: make a point strictly positive} we can find another metric $g_1$ on $M$ such that $m_{ADM}(M,g_1,E)<0$ and $(M,g_1)$ satisfies the condition that $h(p)>0$. This is a contradiction.

    Next, let us assume $\beta>\frac{1}{2}$. It suffices to prove the rigidity of $(M,g)$ under the assumption that $m_{ADM}(M,g,E) = 0$. We first show that $(M,g)$ is scalar flat. If it is not the case, by Proposition \ref{prop: eq1-arbitrary end}, we can solve \eqref{eq: 38} and construct a sequence of ALF manifolds as described in \eqref{eq: 63}. Therefore, by applying Lemma \ref{lem: mass decay ALF 2} we see that
    \begin{align*}
        m_{ADM}(\hat{M}_{i},\hat{g}_{i},\hat{E}) \le (n-1)m_{ADM}(M,g,E) - \frac{2\beta-1}{2\beta\omega_{n-1}}\int_{V_i}|\nabla v_i|^2d\mu_{g}.
    \end{align*}
    Since $(M,g)$ is not scalar flat, the limit function $v$ of $v_i$ is not a constant. Consequently, we have     
    $$\int_{V_i}|\nabla v_i|^2d\mu_{g}\geq a>0,$$ 
    where $a$ is a  constant independent of $i$.  Letting $i\to\infty$ and using Lemma \ref{lem: convergence ADM mass on ALF} and the positive mass theorem for ALF manifold (Proposition \ref{prop: PMT for S^1 symmetric ALF} and \cite[Theorem 1.8]{CLSZ2021}), we get a contradiction. Therefore, $(M,g)$ is scalar flat. The conclusion then follows directly from \cite[Theorem 1.2]{Zhu23}.
\end{proof}

\section{Positive mass theorem on singular spaces}
In this section, we give the proof of Theorem \ref{thm:pmt with singularity4} and Theorem
\ref{thm:non-existence psc on singular space1}. In Sec 4.1 and Sec 4.2, we always assume
 $(M,d,\mu)$ satisfies the conditions $(1)-(6)$ in Theorem \ref{thm:pmt with singularity4}.
\subsection{Properties of the conformally deformed manifolds}
Let $G$ be a positive singular harmonic function as in Proposition \ref{lem: extend harmonic function 2}. Recall that for each $i\in \mathbf{Z}_+$ it holds
      \[
      G(x)\ge C_id(x,\mathcal{S}_i)^{2+k_i-n}
        \ \ \text{for}\ \  d(x,\mathcal{S}_i)\ \   \text{small enough}.
      \]
By Lemma \ref{lem: asymptotic behavior2},  $G$ has an expansion at infinity of the AF end $E$(for any $\varepsilon'>0$) 
\begin{equation}\label{eq: 25}
G = a|x|^{2-n}+\omega \ \ \text{with}
\ \ |\omega| +|x||\partial\omega| = O(|x|^{1-n})+O(r^{-\tau+2-n+\varepsilon'})\ \ \text{as} \quad |x|\to\infty
	\end{equation}
For any $\delta>0$, let $G_\delta(x):=1+\delta G$, then $G_\delta$ is also a  positive singular harmonic function in $M\setminus \mathcal{S}$. Let $ g_\delta:=G_\delta^{\frac{4}{n-2}}g$. We collect basic properties of the conformally deformed manifolds $(M\backslash\mathcal{S},g_{\delta})$.
\begin{lemma}\label{lem: completeness}
		$(M\backslash\mathcal{S},g_{\delta})$ is a complete manifold with an AF end $E$ of the order $\tau$ and some arbitrary ends.
        \end{lemma}
	\begin{proof}
By (4) in the statement of Theorem \ref{thm:pmt with singularity4}, each connected component $\hat{\mathcal{S}}$ is contained in an open set $U$, such that $\partial U\cap\mathcal{S} = \emptyset$.  For simplicity, we may assume that the singular set $\hat{\mathcal{S}} = \mathcal{S}_1\cup\mathcal{S}_2\cup\dots\cup\mathcal{S}_L$. Let 
\begin{align}\label{eq: 72}
    U^g_r: = \{x\in M: d_g (x, \mathcal{S})<r\},
\end{align}
     where $d_g (x, \mathcal{S})$ denotes the distance function to $\mathcal{S}$ with respect to the metric $g$. Now for any small $r_1$, $r_2$ with $0<r_1 <r_2$, choose an arbitrary unit speed curve (with respect to the metric $g$) $\gamma:[0,l]\longrightarrow U^{g}_{r_2}\backslash U^{g}_{r_1}$, $\gamma(0)\in \partial U^{g}_{r_1}, \gamma(l)\in \partial U^{g}_{r_2}$. From \eqref{eq: 72} we know that $\gamma(l)\in \partial \mathcal{B}_{r_2}(\mathcal{S}_\alpha)$ for some $\alpha\in \{1,2,\dots,L\}$.
     
      Denote $D(t) = d_{g}(\gamma(t),\mathcal{S}_\alpha)$. The triangle inequality yields $|D'(t)|\le 1$ almost everywhere. 
      We calculate
        \begin{align*}
            \mbox{length}_{g_{\delta}}(\gamma) = &\int_0^l G_{\delta}^{\frac{2}{n-2}}|\gamma'(t)|_{g}dt \ge \int_0^l D(t)^{\frac{2}{n-2}(2+k_\alpha-n)}|D'(t)|dt\\
            \ge& \int_0^l D(t)^{-1}|D'(t)|dt\ge\ln(\frac{d_{g}(\gamma(l),\mathcal{S}_\alpha)}{d_{g}(\gamma(0),\mathcal{S}_\alpha)}) \ge\ln(\frac{r_2}{r_1}),
        \end{align*}
        where we have used $k_\alpha\le \frac{n}{2}-1$. This shows $d_{g_{\delta}}(\partial U^{g}_{r_1},\partial U^{g}_{r_2})\ge \ln(\frac{r_2}{r_1})$, which immediately implies the completeness.       
        \end{proof}

\begin{proposition}\label{conformal deformation1}
		Suppose $(M\backslash\mathcal{S},g)$ has $\beta$-scalar curvature no less than $h$ in the strong spectral sense, then $(M\backslash\mathcal{S},g_{\delta})$ has $\beta$-scalar curvature no less than $hG_\delta^{-\frac{4}{n-2}}$ in the strong spectral sense, \textit{i.e.} for any $\phi\in Lip(M\backslash\mathcal{S})$ with compact support set or approaches to a constant at the infinity of  AF end $E$, there holds
        \begin{equation}\label{eq:spectrum psc1}
\int_{M\setminus \mathcal{S}}(|\nabla_{g_\delta} \phi|_{g_\delta}^2+\beta R_{g_\delta}\phi^2)d\mu_{g_\delta}\geq \int_{M\setminus \mathcal{S}}\beta hG_\delta^{-\frac{4}{n-2}}\phi^2 d\mu_{g_\delta}.
\end{equation}
	\end{proposition}
	\begin{proof}
		 Owing to the well-known formula, we know that
		\begin{align}\label{eq: 48}
		    R_{g_\delta}=G_\delta^{-\frac{n+2}{n-2}}(R_g \cdot
		G_\delta-\frac{4(n-1)}{n-2}\Delta_g G_\delta)=G_\delta^{-\frac{4}{n-2}}\cdot R_g.
		\end{align}
		Therefore, for $\phi$ with compact support in $M\backslash\mathcal{S}$, we have
		\begin{equation}
			\begin{split}
				&\quad\int_{M\backslash\mathcal{S}}(|\nabla_{g_{\delta}} \phi|^2+\beta  R_{g_{\delta}} \phi^2)d\mu_{g_{\delta}}	\\
				&=\int_{M\backslash\mathcal{S}}(G_{\delta}^2|\nabla_g \phi|^2 +\beta R_g (G_{\delta}\phi)^2)d\mu_g\\
				&=\int_{M\backslash\mathcal{S}}(G_{\delta}^2|\nabla_g \phi|^2 +\beta R_g (G_{\delta}\phi)^2)d\mu_g-\int_{M\backslash\mathcal{S}}(\phi^2 G_\delta \Delta_g G_\delta)d\mu_g\\
				&=\int_{M\backslash\mathcal{S}}(G_{\delta}^2|\nabla_g \phi|^2 +\beta R_g (G_{\delta}\phi)^2)d\mu_g+\int_{M\backslash\mathcal{S}}(\phi^2|\nabla_g
				G_\delta|^2+2 \phi G_\delta \nabla_g G_\delta \cdot \nabla_g \phi)d\mu_g \\
				&=	\int_{M\backslash\mathcal{S}}(|\nabla_g (G_\delta\phi)|^2 +\beta R_g (G_{\delta}\phi)^2)d\mu_g.
			\end{split}
		\end{equation} 
		Now, we consider the case that $\phi$ approaches to a constant at the infinity of  AF end $E$. In this case, we have	
		\begin{align*}
			&\int_{M\backslash\mathcal{S}}(|\nabla_{g_{\delta}} \phi|^2+\beta  R_{g_{\delta}} \phi^2)d\mu_{g_{\delta}}\\
			=&\int_{M\backslash\mathcal{S}}G_{\delta}^2|\nabla_g\phi|^2+\beta R_g(G_{\delta}\phi)^2d\mu_g\\
			=&\int_{M\backslash\mathcal{S}}|\nabla_gG_{\delta}\phi|^2+\beta R_g(G_{\delta}\phi)^2d\mu_g - \int_{M\backslash\mathcal{S}}\divv_g(\phi^2G_{\delta}\nabla_gG_{\delta})d\mu_g\\
			\ge& \int_{M\backslash\mathcal{S}}\beta h(G_{\delta}\phi)^2d\mu_g-\lim_{\rho\to +\infty}\int_{\partial B_{\rho}}\phi^2 G_{\delta}\frac{\partial G_{\delta}}{\partial \Vec{n}} d\sigma.
		\end{align*}
		By applying the maximum principle to the singular harmonic function $G$, we know  the expansion coefficient $a$ in \eqref{eq: 25} is positive. It follows that
		\begin{align*}
			\lim_{\rho\to +\infty}\int_{\partial B_{\rho}}\phi^2 G_{\delta}\frac{\partial G_{\delta}}{\partial \Vec{n}} d\sigma = (2-n)a\delta\le 0.
		\end{align*}
		Thus, in both cases, we get the desired inequality.
	\end{proof}

    \subsection{Proof of Theorem \ref{thm:pmt with singularity4}}
In this subsection, we give the proof of Theorem \ref{thm:pmt with singularity4}. We first consider the case that  $h$ is nonnegative everywhere and  positive somewhere.
By Proposition \ref{prop: eq1-arbitrary end} and the paragraph behind it, for a sequence of  exhausting  domains $ E\subset V_1\subset V_2\subset\dots$ of $M\backslash\mathcal{S}$, we can solve the following equations:
	\begin{equation}\label{eq: 43}
		\left\{
		\begin{aligned}
&\bar{L}_{\delta}v_{i,\delta} = -\Delta_{g_{\delta}} v_{i,\delta} +\frac{1}{2}R_{g_{\delta}} v_{i,\delta} -\frac{1}{4}hG_{\delta}^{-\frac{4}{n-2}}v_{i,\delta}=0 \quad \text{in $V_i$},\\
&\lim\limits _{ x\to\infty, x\in E} v_{i,\delta}(x) = 1,\\
&v_{i,\delta}|_{\partial V_i}=0.
		\end{aligned}
		\right.	
	\end{equation}
  For any fixed coordinated ball $B_r$ in AF end $E$, and $i$ large enough,  we have
	$$
	 \sup_{\partial B_r}(|v_{i,\delta}-1|+r|\partial v_{i,\delta}|) = O(r^{-\tau+\epsilon}).
	 $$
Moreover, we have $v_{j,\delta}>v_{i,\delta}$ on $V_i$ for each $j>i$, and $v_{i,\delta}$ converges monotonically to an function $v_{\delta}\in C^2(M\setminus \mathcal{S})$ in $C^2$ sense with  $\bar{L}_{\delta}v_{\delta} = 0$ and  $|v_{\delta}(x)-1|+|x||\partial v_{\delta}(x)| = O(|x|^{-\tau+\epsilon})$  for $x$ near the infinity of $E$.

Now, we can follow the argument in Section 3 to construct a sequence of ALF manifolds with 
nonnegative scalar curvature and  arbitrary ends by
\begin{align*}
		&(\hat{M}_{i,\delta},\hat{g}_{i,\delta}) = (V_i\times \mathbf{S}^1, g_{\delta} + v_{i,\delta}^2ds^2),\\
		&(\hat{M}_{\delta},\hat{g}_{\delta}) = (M\times \mathbf{S}^1, g_{\delta} + v_{\delta}^2ds^2),
	\end{align*}
where $s\in \mathbf{S}^1$. 
Note that $(\hat{M}_{\delta},\hat{g}_{\delta})$  is complete but $(\hat{M}_{i,\delta},\hat{g}_{i,\delta})$ is not necessarily complete.
Without loss of generality, we assume $p\in V_1$. Therefore, we have $h>0$ in a small neighborhood $\Omega$ of $p$.

\begin{figure}
    \centering
    \includegraphics[width = 15cm]{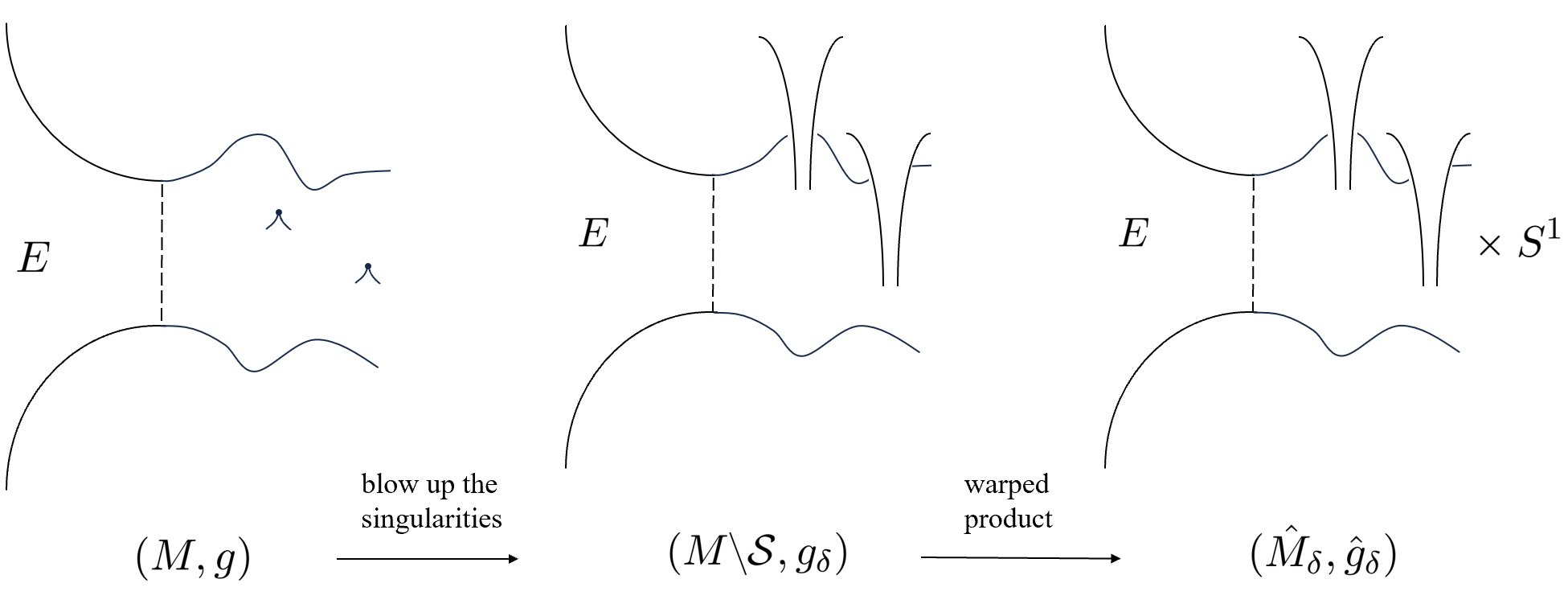}
    \caption{The "blow up - warped product" procedure}
    \label{f4}
\end{figure}

	\begin{lemma}\label{lem: negative mass 2}
		There exists $\delta_0\in (0,1)$, $\sigma_0 = \sigma_0(V_1,g|_{V_1},h|_{V_1})>0$, such that for all $0<\delta<\delta_0$, we have
		\begin{align*}
			m_{ADM}(\hat{M}_{\delta},\hat{g}_{\delta},\hat{E})+\sigma_0 < m_{ADM}(M,g,E).
		\end{align*}
        Here $\cdot|_{V_1}$ denotes the restriction on $V_1$.
	\end{lemma}
	\begin{proof}
        By Lemma \ref{lem: mass decay ALF} and Proposition \ref{conformal deformation1}, we deduce that
        \begin{align*}
            m_{ADM}(\hat{M}_{i,\delta}, \hat{g}_{i,\delta}, \hat{E})
            \le &(n-1)m_{ADM}(M\backslash\mathcal{S}, g_{\delta}, E) - \frac{1}{4\omega_{n-2}}\int_{V_i}hG_{\delta}^{-\frac{4}{n-2}}v_{i,\delta}^2d\mu_{g_\delta}\\
            \le &(n-1)m_{ADM}(M,g,E)+2(n-1)a\delta - \frac{1}{4\omega_{n-2}}\int_{V_i}h(G_{\delta}v_{i,\delta})^2d\mu_g.
        \end{align*}
        Let $\tilde{v}_{i,\delta} = G_{\delta}v_{i,\delta}$. The direct calculation gives $\Delta_{g}\tilde{v}_{i,\delta} =  G_\delta ^{\frac{n+2}{n-2}}\Delta_{g_{\delta}}v_{i,\delta}$. Plugging this into \eqref{eq: 43} and using \eqref{eq: 48}, it is straightforward to see that $\tilde{v}_{i,\delta}$ satisfies
        \begin{equation}\label{eq: 23}
		\left\{
		\begin{aligned}
			&\tilde{L}_{\delta}\tilde{v}_{i,\delta} = -\Delta_g \tilde{v}_{i,\delta} +\frac{1}{2}R_g \tilde{v}_{i,\delta} -\frac{1}{4}h\tilde{v}_{i,\delta}=0 \quad \text{in $V_i$},\\
			&\lim\limits _{ x\to\infty, x\in E} \tilde{v}_{i,\delta}(x) = 1,\\
			&\tilde{v}_{i,\delta}|_{\partial V_i}=0.
		\end{aligned}
		\right.	
	\end{equation}
        Clearly, $\tilde{v}_{i,\delta}$ is independent of $\delta$, meaning that it depends only on the geometry of $(V_i, g|_{V_i})$, so we simply denote it by $\tilde{v}_i$. Since $(M,g)$ has $\frac{1}{2}$-scalar curvature at least $h$ in the strong spectral sense, combined with the maximum principle, we have $\tilde{v}_{j}>\tilde{v}_{i}$ on $V_i$ for any $j>i\ge 1$. It follows that
        \begin{equation}
            \frac{1}{4\omega_{n-2}}\int_{V_i}h(G_{\delta}v_{i,\delta})^2d\mu_g =  \frac{1}{4\omega_{n-2}}\int_{V_i}h\tilde{v}_i^2d\mu_g \ge  \frac{1}{4\omega_{n-2}}\int_{V_1}h\tilde{v}_1^2d\mu_g := 2\sigma_0>0.
        \end{equation}
        Here we note that $\sigma_0$ is independent of $i,\delta$. 
        Therefore, there exists $\delta_0>0$, such that
        \begin{align}\label{eq: 24}
            m_{ADM}(\hat{M}_{i,\delta}, \hat{g}_{i,\delta}, \hat{E})\le m_{ADM}(M,g,E)-\sigma_0<m_{ADM}(M,g,E)
        \end{align}
        for all $0<\delta<\delta_0$.
        Leaving $i\to\infty$ and using Lemma \ref{lem: convergence ADM mass on ALF}, we obtain the desired result.
	\end{proof}

Now, the strict inequality part of   Theorem \ref{thm:pmt with singularity4} follows from  Lemma \ref{lem: negative mass 2} and Proposition \ref{prop: PMT for S^1 symmetric ALF};  

Next, we show that $m_{ADM}(M,g,E)\ge 0$ if $(M, g)$ has non-negative $\beta$-scalar curvature in the strong spectral sense for some $\beta\ge\frac{1}{2}$. Assume that $m_{ADM}(M,g,E)<0$. Let $G$, $G_{\delta}$ and $g_{\delta}$ be as in Sec. 4.1. By taking $\delta$ sufficiently small, we have $m_{ADM}(M\backslash\mathcal{S},g_{\delta},E)<0$. On the other hand, by a similar argument as in the proof of Proposition \ref{conformal deformation1}, we see that $(M\backslash\mathcal{S},g_{\delta},E)$ has non-negative $\beta$-scalar curvature in the strong spectral sense. It follows from Theorem \ref{thm: PMT arbitrary end spectral} that $m_{ADM}(M\backslash\mathcal{S},g_{\delta},E)\ge 0$, which is a contradiction.

We turn to the rigidity part of Theorem \ref{thm:pmt with singularity4}.
We need the following Lemma.

\begin{lemma}\label{lmm:spectrum psc2}
Suppose $(M^n\setminus \mathcal{S},g) $ has  $\beta$-scalar curvature no less than $h$ in the strong spectral sense for some $\beta>\frac{1}{2}$. For some $\beta'\in(\frac{1}{2}, \beta)$, let $u$ be a positive and smooth solution to the following equation
\begin{equation*}
		\left\{
		\begin{aligned}
& -\Delta u+\beta' R_g u=0,\\
&\lim_{|x|\to\infty,x\in E}u(x) = 1.
		\end{aligned}
		\right.	
	\end{equation*}
which is the limit of the positive solution $u_i$ to the following equation
\begin{equation}
		\left\{
		\begin{aligned}
& -\Delta u_i+\beta' R_g u_i=0, x\in V_i\\
&\lim_{|x|\to\infty,x\in E}u_i(x) = 1,\\
&u_i = 0, x\in \partial V_i.
		\end{aligned}
		\right.	
	\end{equation}
Then for any locally Lipschitz function $\phi$ which is either compactly supported or asymptotically constant in the sense of Definition \ref{defn: strong test function}, it holds
\begin{equation}\label{eq: preserving SC NNSC}
    \int_{M\backslash\mathcal{S}} (|\nabla \phi|^2+\frac{R_g}{2}\phi^2)d\mu_g
    \geq \frac{1}{4(\beta')^2}(2\beta'-1)\int_{M\backslash\mathcal{S}} u^{-2}|\nabla u|^2 \phi^2 d\mu_g.
\end{equation}
\end{lemma}
\begin{proof}
Using
$$
div(\frac{1}{2\beta'}u^{-1}\phi^2 \nabla u )=\frac{1}{2\beta'}u^{-1}\phi^2\Delta u+ \frac{1}{2\beta'}\phi^2\nabla u^{-1}\cdot \nabla u + \frac{1}{2\beta'}u^{-1}\nabla u\cdot \nabla \phi^2$$
and
$$
R_g=\frac{\Delta u}{\beta u},
$$	
we have
\begin{align*}
&\int_{M\backslash\mathcal{S}} (|\nabla \phi|^2+\frac{R_g}{2}\phi^2)d\mu_g\\
=&	\int_{M\backslash\mathcal{S}}	(|\nabla \phi|^2- \frac{1}{2\beta'}\phi^2\nabla u^{-1}\cdot \nabla u - \frac{1}{2\beta'}u^{-1}\nabla u\cdot \nabla \phi^2)d\mu_g+\int_{M\backslash\mathcal{S}}div(\frac{1}{2\beta'}u^{-1}\phi^2 \nabla u )\\
=&\int_{M\backslash\mathcal{S}} |\nabla \phi-\frac{1}{2\beta'}u^{-1}\phi \nabla u|^2 d\mu_g +\frac{1}{4(\beta')^2}(2\beta'-1)\int_{M\backslash\mathcal{S}} u^{-2}|\nabla u|^2 \phi^2 d\mu_g+\lim_{\rho\to\infty}\int_{S^{n-1}(\rho)}\frac{1}{2\beta'}u^{-1}\phi^2 \frac{\partial u}{\partial \Vec{n}}\\
\geq& \frac{1}{4(\beta')^2}(2\beta'-1)\int_{M\backslash\mathcal{S}} u^{-2}|\nabla u|^2 \phi^2 d\mu_g+\lim_{\rho\to\infty}\int_{S^{n-1}(\rho)}\frac{1}{2\beta'}u^{-1}\phi^2 \frac{\partial u}{\partial \Vec{n}}.
\end{align*}	
We only need to show the last term in the 
equality above is nonnegative. It suffices to consider the case that $\phi$ is asymptotically constant in the sense of Definition \ref{defn: strong test function}. By integration by parts and using the strong spectral condition, we obtain
\begin{align*}
    \lim_{\rho\to\infty}\int_{S^{n-1}(\rho)}u_i \frac{\partial u_i}{\partial \Vec{n}} = \int_{V_i} |\nabla u_i|^2+\beta'R_g u_i^2 d\mu_g \ge 0.
\end{align*}
 By the maximum principle,
\begin{align*}
    \sup_{x\in E}|u_i-u| = \sup_{x\in\partial E}|u_i-u|
\end{align*}
so $u_i$ converges to $u$ uniformly in $E$. The Schauder estimate improves this to the $C^2$ convergence. Thus, using the asymptotic expansion of $u$ near the infinity of the AF end $E$ and  letting $i\to\infty$, we conclude that
\begin{align*}
    \lim_{\rho\to\infty}\int_{S^{n-1}(\rho)}u^{-1}\phi^2 \frac{\partial u}{\partial \Vec{n}} = \lim_{\rho\to\infty}\int_{S^{n-1}(\rho)}u \frac{\partial u}{\partial \Vec{n}}\ge 0,
\end{align*}
which completes the proof.
\end{proof}
\begin{proof}[Proof of Theorem \ref{thm:pmt with singularity4} (the rigidity part)]
Assume $m_{ADM}(M,g,E) = 0$ and $\beta>\frac{1}{2}$. We first show that $(M\backslash\mathcal{S},g)$ is scalar flat. Assume $R_g\not\equiv 0$. By Proposition \ref{prop: eq1-arbitrary end}, for any $\beta'\in(\frac{1}{2}, \beta)$, there is a positive and smooth function $u$ on $M$ that satisfies
		 $$
		 -\Delta u +\beta'R_g u = 0 \text{ in }M\backslash\mathcal{S}
		 $$
		with $u\to 1$ as $x\to \infty$  in the AF end  $E$.
    Since $R_g\not\equiv 0$, $u$ is not a constant. In conjunction with Lemma \ref{lmm:spectrum psc2}, we see that $(M,d,\mu)$  has  $\frac{1}{2}$-scalar curvature no less than $h$ in the strong spectral sense for some nonnegative function $h$ with $h(p)>0$, which yields that $m_{ADM}(M,g,E)>0$, leading to a contradiction.

Next, we show $(M\backslash\mathcal{S},g)$ is Ricci flat. The idea is to perform Kazdan's conformal deformation \cite{Kazdan82}. Denote
\begin{align*}
    L_g = -\Delta_g+\frac{n-2}{4(n-1)}R_g.
\end{align*}
Assume $Ric_g(p)\ne 0$ at a point $p\in M\backslash\mathcal{S}$. Let $\eta$ be a cut off function which equals to $1$ near $p$ and is supported in $\Omega$, where $\Omega$ is an open set containing $p$. Let $\bar{g} = g-\epsilon\eta Ric_g$. By the argument in \cite[p.9898]{Zhu23}, there exists an open set $\tilde{\Omega}\supset\Omega$ and a function $u\in M\backslash\mathcal{S}$ that satisfies
\begin{enumerate}
    \item\label{item 1} $L_{\bar{g}}u\ge 0$ in $M\backslash\mathcal{S}$;
    \item\label{item 2} $L_{\bar{g}}u> 0$ in $\tilde{\Omega}$. Hence, $L_{\bar{g}}u > b_0>0$ in $\Omega$ for some $b_0$;
    \item\label{item 3} $u\ge\tau_0>0$ in $M\backslash\mathcal{S}$ for some $\tau_0$;
    \item\label{item 4}  $u=1+A|x|^{2-n}+O(|x|^{1-n})$ in the AF end $E$ and $$\lim_{\rho\to\infty}\int_{S^{n-1}(\rho)} u\frac{\partial u}{\partial \Vec{n}_{\bar{g}}}>0.$$
\end{enumerate}
Following the argument of \cite[Lemma 2.9]{Kazdan82}, let $w = 1-e^{-cu}$. Then
\begin{align*}
    L_{\bar{g}}w\ge e^{-cu}(cL_{\bar{g}}u+\frac{n-2}{4(n-1)}R_{\bar{g}}(e^{cu}-1-cu)).
\end{align*}
Using $e^x\ge x+1$ and  \eqref{item 2}, we have $ L_{\bar{g}}w\ge 0$ in $(M\backslash\mathcal{S})\backslash\Omega$. By Taylor's theorem, we obtain
\begin{align*}
    L_{\bar{g}}w\ge ce^{-cu}(L_{\bar{g}}u-\frac{1}{2}e^{c\gamma}c\gamma^2|R_{\bar{g}}|)\quad\mbox{  in  }\Omega,
\end{align*}
where $\gamma = \sup_{\bar{\Omega}}u(x)$. Therefore, by taking $c$ sufficiently small and using \eqref{item 2}, we have $L_{\bar g}w\ge 0$ in $M\backslash\mathcal{S}$.

As a consequence of \eqref{item 3} and \eqref{item 4}, we know that
 \begin{equation}\label{eq: estimate for w}
        0<a = 1-e^{-c\tau_0}\le w< 1
       \end{equation}
       and
        \begin{equation}\label{eq: intergral for w}
        \lim_{\rho\to\infty}\int_{S^{n-1}(\rho)} w\frac{\partial w}{\partial \Vec{n}_{\bar{g}}}>0.
\end{equation}

Set $\tilde{g} = (\frac{w}{1-e^{-c}})^{\frac{4}{n-2}}g$, then $(M,\tilde{g})$ is an AF manifold 
 of order $\tau$. Let $d_g,d_{\tilde{g}}$ be the metrics on $M\backslash\mathcal{S}$, then by the definition of almost manifold condition it holds $d_g = d|_{M\backslash\mathcal{S}}$. Define
\begin{align}\label{eq: 68}
    \tilde{d}(x,y) = \lim_{i\to\infty}d_{\tilde{g}}(x_i,y_i)
\end{align}
for any sequence  $x_i,y_i\in M\backslash\mathcal{S}$ with $d(x_i,x)\rightarrow0$ and $d(y_i,y)\rightarrow 0$. Since
\begin{align}\label{eq: 69}
    (\frac{a}{1-e^{-c}})^{\frac{4}{n-2}}d_g\le d_{\tilde{g}}
    \le (\frac{1}{1-e^{-c}})^{\frac{4}{n-2}}d_g,
\end{align}
we see that the limit in \eqref{eq: 68} exists, and does not depend on the choice of the sequence $\{x_i\},\{y_i\}$. This shows \eqref{eq: 68} is well defined. By taking a limit for the triangle inequality for $d_{\tilde{g}}$, we are able to verify the triangle inequality for $\tilde{d}$. Furthermore, since $M$ is a Hausdorff space, for any $x,y\in M$ that satisfies $x\ne y$, we can pick two open sets $U_x$ and $U_y$ containing $x,y$ respectively, such that $\bar{U}_x\cap\bar{U}_y = \emptyset$. Using \eqref{eq: estimate for w} and \eqref{eq: 69}, we conclude $\tilde{d}(x,y)>0$, so $\tilde{d}$ is positive definite.

Finally, we define
\begin{align*}
    \tilde{\mu}(A) = \int_A w^{\frac{2n}{n-2}}d\mu.
\end{align*}
Due to \eqref{eq: estimate for w}, it is straightforward to verify that $(M,\tilde{d},\tilde{\mu})$ is a metric measure space that satisfies the almost manifold condition and the conditions in Theorem \ref{thm:pmt with singularity4}. By the inequality part of Theorem \ref{thm:pmt with singularity4}, we have $m_{ADM}(M,\tilde{g},E)\ge 0$. However, using \eqref{eq: intergral for w}, we obtain $m_{ADM}(M,\tilde{g},E)<m_{ADM}(M,g,E) = 0$, a contradiction. This concludes the proof of the rigidity parts of Theorem \ref{thm:pmt with singularity4}.
\end{proof}

\subsection{Proof of Theorem \ref{thm:non-existence psc on singular space1}}

In this subsection, we present the proof of Theorem \ref{thm:non-existence psc on singular space1}.

Let $\mathcal{U}\subset M$ be a  small neighborhood of $p\in M$ such that $h(x)>0$ for any $x\in \mathcal{U}$, we always assume $\mathcal{U}\cap \mathcal{S}=\emptyset$.   We need the following lemma:

\begin{lemma}\label{lmm:singular function}
Let  $(M^n,d, \mu)$ be as in Theorem \ref{thm:non-existence psc on singular space1} and
$f$ be a  non-negative smooth function on $(M^n,d, \mu)$ which is strictly positive somewhere, satisfying  $supp(f)\subset\subset \mathcal{U}$. Then there is a positive smooth function $u$ on $M\setminus\mathcal{S}$ with
$$
-\Delta u+fu=0 \quad \text{in $M\setminus\mathcal{S}$},
$$	
and 
$$
u|_{\partial B(p_i,r)}\geq Cr^{2-n}, \quad \text{ for any  $p_i\in \mathcal{S}$ and any $r>0$}.
$$
\end{lemma}

\begin{proof}
	
For simplicity we assume $\mathcal{S} $	consists of a single point, say $p_1$.  Choose $\{r_i\}$ decreasing to $0$, consider the following Dirichlet problem

\begin{equation}\label{eq:Dirichlet prob3}
\left\{
\begin{aligned}
&-\Delta u_i +fu_i=0 \quad \text{in $M\setminus B(p_1,r_i)$ },\\
&u_{i}|_{\partial  B(p_1,r_i)}=1\\
\end{aligned}
\right.
\end{equation}
Note that $f\geq 0$, we see that  \eqref{eq:Dirichlet prob3}	admits a positive  smooth solution $u_i$. Let $q\in M\setminus\mathcal{S}$, define 
$$
w_i:=(u_i(q))^{-1}u_i.
$$	
Clearly,  $w_i$	also satisfies equation \eqref{eq:Dirichlet prob3} and $w_i(q)=1$. By the Harnack inequality, passing by a subsequence if necessary, we may assume $w_i$ converges smoothly and locally to $u$ in $M\setminus\mathcal{S}$	. If  $u$ is bounded on $M$, then by Lemma \ref{lem: removable singularity} we have
$$
-\Delta u +fu =0.
$$ in $M$.
 By choosing the cut-off function $\eta$ in Lemma \ref{lem: estimate for eta}, and multiplying two sides of the above equation by $\eta^2u$, we obtain
$$
\int_M |\nabla u|^2 d\mu<\infty,
$$		
and 	
$$
\int_M (|\nabla u|^2 +fu^2)d\mu	=0,
$$	
which implies $u=0$ and contradicts to $u(q)=1$. Therefore, $u(x)\rightarrow \infty$ as $x\rightarrow p_1$.

Note that $u$ is a positive harmonic function near $p_1$. In conjunction with Proposition \ref{prop:existence of singular positive hf} and Proposition \ref{prop:growth of hf1}, we complete the proof of Lemma \ref{lmm:singular function}.
\end{proof}

The next Lemma shows that with a suitable choice of $f$, the conformal metric $\bar g= u^{\frac{4}{n-2}}g$ has $\beta$-scalar curvature  non-negative in the spectral sense.

\begin{lemma}\label{lmm:spectrum psc1}
Let $u$ be as in Lemma \ref{lmm:singular function} and  $\bar g= u^{\frac{4}{n-2}}g$. Then for any $\phi \in C^\infty_0(M\setminus\mathcal{S})$ there holds
$$
\int_{M\setminus\mathcal{S}} (|\bar\nabla \phi|^2+\beta R_{\bar g}\phi^2)d\mu_{\bar g}
\geq \int_{M\setminus\mathcal{S}} \left(h-\frac{(4\beta-1)n-(4\beta-2)}{n-2} f \right)u^{-\frac{2n}{n-2}}\phi ^2 d\mu_{\bar g}.
$$	
\end{lemma}

\begin{proof}
The scalar curvature of the conformal metric $\bar{g}$ is given by
\begin{equation}
R_{\bar g}=u^{-\frac{n+2}{n-2}}(R_g u-\frac{4(n-1)}{n-2}\Delta_g u)
=u^{-\frac{4}{n-2}}(R_g -\frac{4(n-1)}{n-2}f).
\end{equation}
Note that 
$$
\nabla u^2 \cdot\nabla\phi^2=div(\phi^2 \nabla u^2)-2\phi^2 |\nabla u|^2-2f(u\phi)^2,
$$
we obtain
$$
\int_M |\nabla (u\phi) |^2 d\mu_g =\int_M u^2|\nabla \phi |^2 d\mu_g-\int_M f(u\phi)^2 d\mu_g$$
Therefore, we have
\begin{equation}
	\begin{split}
\int_M (|\bar\nabla \phi|^2+\beta R_{\bar g}\phi^2)d\mu_{\bar g}&=\int_M(u^2|\nabla_g\phi|^2+\beta(R_g -\frac{4(n-1)}{n-2}f)(u\phi)^2	)d\mu_g\\
&=\int_M(|\nabla_g(u\phi)|^2+\beta(R_g -\frac{4(n-1)}{n-2}f)(u\phi)^2	)d\mu_g	+\int_M f(u \phi)^2d\mu_g\\
&\geq \int_M \left(\beta h-\frac{(4\beta-1)n-(4\beta-2)}{n-2} f \right)(\phi u)^2 d\mu_g\\
&=\int_M \left(\beta h-\frac{(4\beta-1)n-(4\beta-2)}{n-2} f \right)u^{-\frac{4}{n-2}}\phi ^2 d\mu_{\bar g}.
\end{split}\nonumber
\end{equation}
\end{proof}
Now, we are in the position of proving Theorem \ref{thm:non-existence psc on singular space1}.

\begin{proof}[Proof of  Theorem \ref{thm:non-existence psc on singular space1}]
First, we consider the case that $M$ satisfies the item $(1)-(4)$ in the Theorem \ref{thm:non-existence psc on singular space1}. Suppose the conclusion of Theorem \ref{thm:non-existence psc on singular space1} is false, then there is a Riemannian metric $g$ on $M\setminus\mathcal{S}$ satisfying  item $(1)-(4)$ listed in Theorem \ref{thm:non-existence psc on singular space1}. Let $f\ge 0$ be a compactly supported smooth function in $M\backslash\mathcal{S}$. We may choose $f$ as in Lemma \ref{lmm:singular function} small enough so that 
\begin{align*}
    \eta(x):=\left(\beta h-\frac{(4\beta-1)n-(4\beta-2)}{n-2} f \right)u^{-\frac{4}{n-2}} (x)\geq 0, \quad \text{for all $x\in M\setminus \mathcal{S}$}, 
\end{align*}
with
\begin{align*}
    \eta(p)=\left(\beta h-\frac{(4\beta-1)n-(4\beta-2)}{n-2} f \right)u^{-\frac{4}{n-2}}(p)> 0.
\end{align*}
Let $u$ be as in Lemma \ref{lmm:singular function} and $\bar g= u^{\frac{4}{n-2}}g$.
Then $(M\setminus \mathcal{S}, \bar g)$ is a complete Riemannian manifold. By Definition \ref{defn: degree 2}, we see
\begin{align*}
    \psi:M\setminus \mathcal{S}\mapsto X
\end{align*}
is a proper map with $deg(\psi)\neq 0$ in the sense of Definition \ref{defn: non-zero degree}, where $X$ is the enlargeable manifold in the statement of Theorem \ref{thm:non-existence psc on singular space1}.

Owing to  Lemma \ref{lmm:spectrum psc1}, we know that $\lambda_1(-\Delta_{\bar{g}}+\frac{1}{2}R_{\bar{g}})\ge 0$.  Hence, we may find a positive smooth function $\varphi$ on $M\setminus \mathcal{S}$ such that 
$$
\hat g:=\bar g+ \varphi^2 d\theta^2, \quad \theta \in \mathbf{S}^1
$$
is a complete $\mathbf{S}^1$-invariant Riemannian metric on $(M\setminus \mathcal{S})\times \mathbf{S}^1 $ with $R_{\hat g}\geq 0$ and $R_{\hat g}(p)> 0$. Then by Lemma \ref{G-invariant Kazdan}, we may deform $\hat g$ to get a new complete and $\mathbf{S}^1$-invariant Riemannian metric on $(M\setminus \mathcal{S})\times \mathbf{S}^1 $ with positive scalar curvature everywhere, which contradicts with Proposition \ref{prop: noncompact dominate enlargeable 2}. 

Next we turn to the proof of the rigidity part of the theorem.
Since there exists some $\beta>\frac{1}{2}$ such that
\begin{align*}
    \int_{M\setminus\mathcal{S}} (|\nabla \phi|^2+\beta R_{ g}\phi^2)d\mu_{ g}\geq 0,
\end{align*}
 we can use the same strategy as in  Proposition \ref{prop: eq1-arbitrary end} to obtain a positive and smooth function $w$ on $M\setminus\mathcal{S} $ with
\begin{align*}
    -\Delta_{g}w+\beta R_{g} w=0.
\end{align*}
Hence, in conjunction with the argument of Lemma \ref{lmm:spectrum psc2}, we see that for any $\phi \in C^\infty_0(M\setminus \mathcal{S})$, it holds
\begin{align*}
    \int_M (|\nabla_{ g} \phi|^2+\frac{R_{ g}}{2}\phi^2)d\mu_{ g}\geq \frac{1}{4\beta^2}(2\beta-1)\int_M w^{-2}|\nabla_{ g} w|^2 \phi^2 d\mu_{ g},
\end{align*}
which implies that there exists a smooth and positive $h'$ on $M\setminus \mathcal{S}$ satisfying the condition satisfied by $h$ in (3) in Theorem \ref{thm:non-existence psc on singular space1}. Then, proceeding as in the proof of the first part of Theorem \ref{thm:non-existence psc on singular space1}, we conclude that $w$ must be a constant. Thus, $R_{ g}=0$.

If there is $x\in M\setminus \mathcal{S}$ with $|Ric_g(x)|\neq 0$, then by Theorem B in \cite{Kazdan82}, we can deform $g$ to get a PSC metric on  $M\setminus \mathcal{S}$ which satisfies conditions $(1)-(4)$ in Theorem \ref{thm:non-existence psc on singular space1}. However, this contradicts with the first part of Theorem \ref{thm:non-existence psc on singular space1}. Therefore, $(M\setminus \mathcal{S}, g)$ is Ricci flat. 
\end{proof}

\section{Foliation of area-minimizing hypersurfaces in AF manifolds}

In this section, we consider an asymptotically flat (AF) manifold $(M^{n+1}, g)$ of dimension $n+1$ with asymptotic order $\tau > \frac{n}{2}$. Throughout our discussion, we use $(y,z),y\in\mathbf{R}^n$ to denote the points in the AF end $E$ and $S_t$ to represent the coordinate $z$-hyperplane. Within the AF end $E$,  the hypersurface $(\partial B^{n+1}_R(O) \cap S_0) \times \mathbf{R}$ partitions $E$ into two distinct regions: an inner part and an outer part. Let $C_r$ be the  inner part.
To better illustrate the ideas, throughout this section our AF manifolds are assumed to have a single end, and at the end of this section we will explain how to extend the results to AF manifolds with arbitrary ends by using exactly the same idea.
\subsection{Plateau problem in cylinder $C_r$}
We want to solve the Plateau problem in cylinder $C_r$ with given boundary $S_{r,t}:=\partial C_r\cap S_t$. Let us reformulate the problem in terms of Geometry measure theory, all notations below are from \cite{Simon83}. 

Let $\Omega_{r,t}$ be a fixed bounded smooth domain  that satisfies $\partial \Omega_{r,t}\cap \partial C_r $ is the portion in $\partial C_{r}$ that lies below $S_{r,t}$, {\it  we want to find a Caccioppoli set $E_{r,t}$ with the least boundary volume $\mathcal{H}^{n-1}(\partial E_{r,t}\cap  C_r) $ and the symmetric difference $E_{r,t}\triangle \Omega_{r,t}\subset C_r$.}	

Let us recall some basic facts of Caccioppoli set and rectifiable current (cf p.146 in \cite{Simon83}).  Let $U\subset M^n$ be an open set, then
\begin{lemma}[Theorem 14.3 in \cite{Simon83}]\label{lmm: rectifiable}
Let $E\subset M^{n+1}$ be a Caccioppoli set, then its reduced boundary $\partial^* E$ is countably $n$-rectifiable and for any $C^1$-tangential vector $g $ with $supp(g)\subset U$

$$
\int_{E\cap U}div~ g~  d\mathcal{L}^{n+1}=-\int_U g\cdot v d\mu_E,
$$
where $\mu_E=\mathcal{H}^{n-1}\lfloor \partial^*E$, $v=(v^1,\cdots, v^{n+1})$ is $\mu_E$-measurable function with $|v|=1$ $\mu_E$-a.e. in $U$. Moreover, for any $x\in  \partial^*E$ the approximate tangent space $T_x$ of 	$\mu_E$ exists and has multiplicity $1$.
\end{lemma}
\begin{remark}\label{remark: bdry of Caccioppoli set}
For any $x\in  \partial^*E$ we may define an orientation $\xi(x)$ via $v(x)$. Therefore, for any   Caccioppoli set $E\subset M^n$, we  may associate an integer multiplicity $n$-rectifiable current $T\in \mathcal{D}_{n}(U)$ as following, for any $\omega \in \mathcal{D}^{n}(U)$
$$
T(\omega):=\int_{\partial^*E} <\omega(x), \xi(x) >d\mathcal{H}^{n}(x).
$$	
By this reason, we regard $[|\partial^*E|]:=T:=\underline{\tau}(\partial^*E,\xi,\theta)$ as an integer multiplicity $n$-rectifiable current, and $\theta=1$ in this case. Indeed, for such a  $T$ we may associate it with an integer multiplicity varifold $V=\underline V(\partial^*E, \theta)$ (see p.146 in \cite{Simon83})
\end{remark}

\begin{proposition}[Theorem 1.24 in \cite{Giu1984}]\label{prop: approximate}
Every bounded Caccioppoli set $E$ can be approximated by a sequence of $C^\infty$ domain $E_j$ such that
\begin{equation}\label{eq1}
\int|\chi_{E_j}-\chi_{E}|dx\rightarrow 0,	
\end{equation}
and
\begin{equation}\label{eq2}
\int|D\chi_{E_j}|\rightarrow \int|D\chi_{E}|.
\end{equation}
Here and in the sequel $\chi_A$ denotes the characteristic function of set $A$.
\end{proposition} 
As we are interested in minimizing sequence of Plateau problem, and due to the above Proposition \ref{prop: approximate}, we will assume  Caccioppoli set $E\subset M^{n+1}$ is open in the following, and then for any $\omega \in \mathcal{D}^{n+1}(U)$, we can define integer multiplicity $n+1$-rectifiable current $T\in \mathcal{D}_{n+1}(U)$ as following
$$
T(\omega):=\int_{E} <\omega(x), \xi(x) >d\mathcal{H}^{n}(x),
$$
where $\xi$ denotes the orientation of $M^n$. 
\begin{proposition}\label{prop: existence of Plateau problem}
Let $(M^{n+1}, g)$ be an AF manifold; then there is an area-minimizing current $T_{r,t}\in \mathcal{D}_{n}(U)$ with $\partial T_{r,t}=[|S_{r,t}|] \in \mathcal{D}_{n-1}(U)$ for some open domain $U$ containing $C_r$. Additionally, $T_{r,t}=\partial [|E_{r,t}|]\lfloor C_r$ for some $[|E_{r,t}|]\in \mathcal{D}_{n+1}(U)$.
\end{proposition}
\begin{proof}
Due to Theorem 1.20 in \cite{Giu1984}, we know that there is a 	Caccioppoli set $E_{r,t}\subset U\subset M^{n+1}$ that has the least boundary volume and $E_{r,t}\triangle \Omega_{r,t}\subset C_r$. By Lemma \ref{lmm: rectifiable} and the remark after that, we know that $\partial^* E_{r,t}$ is multiplicity $1$ rectifiable varifold and current. 
Let $T_{r,t}=\partial [|E_{r,t}|]\lfloor C_r$, hence,  $T_{r,t}\in \mathcal{D}_{n}(U)$  for some open domain $U$ containing $C_r$. Since the boundary of the coordinate balls $\partial B_s^{n+1}(O)$ forms a mean convex foliation for $s\ge r$, we know that $spt (T_{r,t})\subset B^{n+1}_r\subset C_r$, so it holds $\partial T_{r,t}=[|S_{r,t}|]$. 
\end{proof}

\subsection{Density estimate for area-minimizing  hypersurfaces  in asymptotically flat manifolds}

In the following we assume $\Sigma:=spt (T)$ to be the area-minimizing current in $(M^{n+1},g)$ and $T=\partial[|E|]$ where $E$ is a Caccioppoli set in $M^{n+1}$. Hence, $T$ can be regarded as a proper integer $n$-rectifiable varifold of multiplicity one in $(M^{n+1},g)$ which is still denoted by $\Sigma$.
In this subsection, we will derive the density estimate for area-minimizing  hypersurfaces  in asymptotically flat manifolds. As the first step, We shall establish monotonicity inequalities for minimal hypersurfaces in both local and global cases.

$\quad$

\textbf{1.Local inequalities}

Choose some fixed $r_0>1$ such that outside $B^{n+1}_{r_0}(O)$,
the injective radius of $(M^{n+1},g)$
is at least $1$
and we have $|sec|\leq k_0$. 
\begin{lemma}\label{lm: monotonicity ineq2} 
 Let $\Sigma^{n}$
be an area-minimizing hypersurface in $(M^{n+1},g)$. 
Then for all $\xi\in  M$ satisfying 
$\mathcal{B}^{n+1}_1(\xi)\subset M^{n+1}\setminus B^{n+1}_{r_0}(O) $, there holds
\begin{equation}\label{eq: monotonicity ineq2}
    e^{k_0\rho_2}\rho_2^{-n}\mathcal{H}^{n}(\mathcal{B}^{n+1}_{\rho_2}(\xi)\cap \Sigma)
    \geq e^{k_0\rho_1}\rho_1^{-n}\mathcal{H}^{n}(\mathcal{B}^{n+1}_{\rho_1}(\xi)\cap \Sigma)
\end{equation}
for any $0<\rho_1<\rho_2<1$. Here $\mathcal{B}^{n+1}_{\rho}(\xi)$ denotes the geodesic ball in $(M^{n+1},g)$.
\end{lemma}
\begin{proof}
    Let $r$ be the distance of $x$ to $\xi$  in $(M^{n+1},g)$. Then
    \begin{equation}
        div_{\Sigma}(\nabla^M r^2)=\sum_{i=1}^{n}2r\Hess^Mr(e_i,e_i)
        +2\langle \nabla^Mr, e_i\rangle^2.
    \end{equation}
    By Hessian comparison theorem (see \cite[p.234]{CM11}) we know in 
    $\mathcal{B}^{n+1}_1(\xi)$, for any vector  $V$
with  $|V|=1$, there holds
\[
| \Hess^M  r (V,V)-\frac{1}{r}(1-\langle V, \nabla^M r\rangle^2)|\leq\sqrt{k_0}.
\]
It follows
\[
|div_{\Sigma}(\nabla^M r^2)-2n|\leq2n\sqrt{k_0}r.
\]
The rest of the proof follows from the standard argument.
\end{proof}
\textbf{2. Global inequalities}

In order to get the regularity of area-minimizing hypersurfaces, we need a global monotonicity formula for volume density as well. 
Given $t$, we use
 $\Sigma_{a,t}$ to denote the solution of Plateau's problem in the coordinate cylinder $C_a$ with $\partial \Sigma_{a,t}= S_t \cap C_a$ for the coordinate hyperplane $S_t$.
 Let $\xi=(\xi^1, \cdots, \xi^{n+1})$ be a  point in the AF end $E$ with $|\xi|\geq \rho_0$ where $\rho_0$ is a constant to be determined.
 For simplicity,
 set 
 \[
 \theta_{\xi}(\rho,a)=\frac{\mathcal{H}^{n}(\Sigma_{a,t}\cap B^{n+1}_{\rho}(\xi))}{\omega_{n}\rho^{n}},
 \]
where $B^{n+1}_{\rho}(\xi):=\{x\in M^{n+1}: |x-\xi|<\rho\}$  denotes the coordinate ball with centre $\xi$ and radius $\rho$. 
  Then we have
\begin{lemma}\label{lmm: monotonicity ineq2}
For any $\rho_1<\rho_2$ with 
 $(B^{n+1}_{\rho_2}(\xi)\setminus B^{n+1}_{\rho_1}(\xi))\subset M\backslash B^{n+1}_{\rho_0}(O)$ and $B^{n+1}_{\rho_2}(\xi)\cap\partial\Sigma_{a,t}=\emptyset$,
it holds
\begin{equation}\label{eq: monotonicity formula}
	\begin{split}
	 \theta_{\xi}(\rho_2,a)- \theta_{\xi}(\rho_1,a)
    \geq
	\int^{\rho_2}_{\rho_1}s^{-n-1}\int_{\Sigma_{a,t}\cap B^{n+1}_{s}(\xi)}(u-n|v|)ds-\rho_2^{-n}\int_{\Sigma_{a,t}\cap  B^{n+1}_{\rho_2}(\xi)}|v|,
    \end{split}
\end{equation}
where $u,v$ satisfy
\begin{equation}
 u(x)=\left\{
\begin{aligned}
&O(|x|^{-1-\tau})|\xi|+O(|x|^{-\tau}), \quad \text{for $x\in M\setminus B^{n+1}_{R_0}(O)$, }\\
&O(1)|\xi|,\quad\ \quad \quad\quad\quad\quad\quad\ \ \  \ \text{for $x\in B^{n+1}_{R_0}(O)$. }
\end{aligned}
\right.\nonumber
\end{equation}
and
\begin{equation}
 v(x)=\left\{
\begin{aligned}
&O(|x|^{-\tau}), \quad\text{for $x\in M\setminus B^{n+1}_{R_0}(O)$, }\\
&O(1),\quad\ \quad \   \text{for $x\in B^{n+1}_{R_0}(O)$. }
\end{aligned}
\right.\nonumber
\end{equation}
\end{lemma}

\begin{proof}
The arguments are almost the same as those of 	Theorem 17.6 in \cite{Simon83}.
We include a proof here for completeness. 
Let $\{x^i\}$($1\leq i\leq n+1$) be the coordinate function on AF end $E$
and  $f^i$ be a smooth  extension of $x^i-\xi^i$ satisfying
\begin{equation}
f^i=\left\{
\begin{aligned}
&x^i-\xi^i, \quad \text{for $x\in M^{n+1}\setminus B^{n+1}_R(O)$},\\
&O(|\xi|),\quad \text{for $x\in B^{n+1}_R(O)$}.
\end{aligned}
\right.\nonumber
\end{equation}
 and set 
$$
f:=r^2:=(f^1)^2+\cdots +(f^{n+1})^2.
$$
We assume $f>0$ on $M^{n+1}\setminus \xi$ and fix some $R_0>r_0$ such that $|\nabla_Mr|>0$ 
outsider $B^{n+1}_{R_0}(O)$.
Let $\{\phi_k\}$ be a sequence of smooth and decreasing functions given by

\begin{equation}
\phi_k(t)=\left\{
\begin{aligned}
&1, \quad \text{for $t\leq \frac{k-1}{k}$, }\\
&0,\quad \text{for $t\geq 1$. }
\end{aligned}
\right.\nonumber
\end{equation}
For any $r\leq \rho$ we define
$$
\gamma(r):=\phi_k(\frac{r}{\rho}),\ \ 
I(\rho)=\int \phi_k(\frac{r}{\rho}) d\mu,
$$
and 
$$
J(\rho)=\int \phi_k(\frac{r}{\rho})|(\nabla_M r)^{\perp}|^2 d\mu,
$$
where $\mu$ is a Radon measure associated with $\Sigma_{a,t}$.
Namely, for any $\mathcal{H}^{n}$-measurable $A$, 
$$
\mu(A):=\mathcal{H}^{n}(\Sigma_{a,t}\cap A).
$$
By a direct computation, we have
$$
r\gamma'(r)=-\rho\frac{\partial}{\partial \rho}[\phi_k(\frac{r}{\rho})],\ \ \ \ 
I'(\rho)=-\rho^{-1}\int r\gamma'(r)  d\mu,
$$
and
$$
J'(\rho)=-\rho^{-1}\int r\gamma'(r)|(\nabla_Mr)^{\perp}|^2 d\mu.
$$
Set 
$$
Y:=\frac{1}{2}\nabla_M f
$$
and let $\{e_i\}(1\leq i\leq n+1)$ be the orthonormal frame of  $M^{n+1}$ where $\{e_i\}$,$1\leq i\leq n$ be the tangential vectors of $\Sigma_{a,t}$,  then
\begin{align}\label{eq: divY}
div_{\Sigma_{a,t}} Y&=\frac{1}{2}\sum^{n}_{i=1}\nabla^2_M f(e_i,e_i)=n+u,
\end{align}
where
\begin{equation}
 u(x)=\left\{
\begin{aligned}
&O(|x|^{-1-\tau})|\xi|+O(|x|^{-\tau}), \quad \text{for $x\in M\setminus B^{n+1}_{R_0}(O)$ },\\
&O(1)|\xi|,\quad\ \qquad \qquad \qquad  \ \ \  \ \text{for $x\in B^{n+1}_{R_0}(O)$ }.
\end{aligned}
\right.\nonumber
\end{equation}
Let $X:=\gamma(r)Y$, then
\begin{equation}
	\begin{split}
div_{\Sigma_{a,t}} X&=n\gamma(r)+u\gamma(r)+\gamma'(r) g(\nabla_{\Sigma_{a,t}} r, Y)	\\
&=n\gamma(r)+\gamma'(r)g(\nabla_{\Sigma}r,r\nabla_Mr)+u\gamma(r)\\
&=n\gamma(r)+r\gamma'(r)(|\nabla_Mr|^2-(e_{n+1}(r))^2)+u\gamma(r)\\
&=n\gamma(r)+r\gamma'(r)-r\gamma'(r)(e_{n+1}(r))^2+r\gamma'(r)v+ u\gamma(r),
\end{split}\nonumber
\end{equation}
where $v$ satisfies
\begin{equation}
 v(x)=\left\{
\begin{aligned}
&O(|x|^{-\tau}), \quad \text{for $x\in M\setminus B^{n+1}_{R_0}(O)$, }\\
&O(1),\quad\ \quad  \ \text{for $x\in B^{n+1}_{R_0}(O)$. }
\end{aligned}
\right.\nonumber
\end{equation}
Note that $\Sigma_{a,t}$ is area-minimizing, we have
$$
\int div_{\Sigma_{a,t}} X d\mu=0.
$$
Hence,
$$
\int(n\gamma(r)+r\gamma'(r))d\mu=\int \gamma'(r)r(e_{n+1}(r))^2 d\mu -\int [r\gamma'(r)v+u\gamma(r)]d\mu.
$$
Or equivalently, 
\begin{equation}\label{eq: differential on I}
\frac{d}{d\rho}(\rho^{-n}I(\rho))=\rho^{-n}J'(\rho)+\rho^{-n}L'(\rho)+\rho^{-n-1}\int u\gamma(r)d\mu.	
\end{equation}
Here 
$$
L(\rho)=-\int \phi_k(\frac{r}{\rho})vd\mu.	
$$
Taking integral of \eqref{eq: differential on I} on $[\rho_1, \rho_2]$ and setting $k\rightarrow \infty$
yields 
\begin{equation}\label{eq: monotonicity formula 1}
	\begin{split}
	&\rho^{-n}_2\mathcal{H}^{n}(\Sigma_{a,t}\cap B^{n+1}_{\rho_2}(\xi))-	\rho^{-n}_1\mathcal{H}^{n}(\Sigma_{a,t}\cap B^{n+1}_{\rho_1}(\xi))\\
    =&\rho^{-n}_2 \int_{\Sigma_{a,t}\cap B^{n+1}_{\rho_2}(\xi)}
    \left(|(\nabla_M r)^{\perp}|^2-v\right)
	-\rho^{-n}_1 \int_{\Sigma_{a,t}\cap B^{n+1}_{\rho_1}(\xi)}
    \left(|(\nabla_M r)^{\perp}|^2-v\right)\\
	+&n\int^{\rho_2}_{\rho_1}s^{-1-n} \left(\int_{\Sigma_{a,t}\cap B^{n+1}_{s}(\xi)}|(\nabla_M r)^{\perp}|^2-v\right)ds
	+\int^{\rho_2}_{\rho_1}s^{-1-n}(\int_{\Sigma_{a,t}\cap B^{n+1}_{s}(\xi)}u)ds.\\
    \end{split}
\end{equation}
By co-area formula, for any $C^0$ function $\varphi$
we have
\[
\frac{d}{ds}\int_{\Sigma_{a,t} \cap B^{n+1}_s(\xi)}\varphi|\nabla_{\Sigma_{a,t}} r|
=\int_{\Sigma_{a,t} \cap \partial B^{n+1}_s(\xi)}\varphi,~~ \text{for a.e. $s\in [\rho_1, \rho_2]$}.
\]
Take $\varphi=(|(\nabla_M r)^{\perp}|^2-v)|\nabla_{\Sigma_{a,t}} r|^{-1}$. Then  multiply $s^{-n}$ on both sides of the equation above and take integral on $[\rho_1,\rho_2]$, we have
\begin{equation}\label{eq: first term}
\begin{split}
&\int^{\rho_2}_{\rho_1}s^{-n}
    \int_{\Sigma_{a,t}\cap \partial B^{n+1}_{s}(\xi)}(|(\nabla_M r)^{\perp}|^2-v)|\nabla_{\Sigma_{a,t}} r|^{-1}ds\\
    =&\rho^{-n}_2 \int_{\Sigma_{a,t}\cap B^{n+1}_{\rho_2}(\xi)}
    (|(\nabla_M r)^{\perp}|^2-v)
	-\rho^{-n}_1 \int_{\Sigma_{a,t}\cap B^{n+1}_{\rho_1}(\xi)}
   ( |(\nabla_M r)^{\perp}|^2-v)\\
	&+n\int^{\rho_2}_{\rho_1}s^{-1-n}\int_{\Sigma_{a,t}\cap B^{n+1}_{s}(\xi)}(|(\nabla_M r)^{\perp}|^2-v)ds.\\
    \end{split}
\end{equation}
On the other hand, set 
$$
y(s):=\int_{\Sigma_{a,t}\cap  B^{n+1}_{s}(\xi))}|v|d\mu,
$$
then  by co-area formula,  for a.e. $s\in[\rho_1,\rho_2]$, we have
$$
y'(s)=\int_{\Sigma_{a,t}\cap \partial B^{n+1}_{s}(\xi)}|v||\nabla_{\Sigma_{a,t}} r|^{-1} d\mu.
$$
It follows that
\begin{equation}\label{eq: estimate term1}
\begin{split}
    \int^{\rho_2}_{\rho_1}s^{-n}
    \int_{\Sigma_{a,t}\cap \partial B^{n+1}_{s}(\xi)}|v||\nabla_{\Sigma_{a,t}} r|^{-1}ds
=& \int^{\rho_2}_{\rho_1}s^{-n} y'(s)ds \\
\leq& \rho^{-n}_2y(\rho_2)+n\int^{\rho_2}_{\rho_1}s^{-n-1}
    \int_{\Sigma_{a,t}\cap B^{n+1}_{s}(\xi)}|v|ds.\\
\end{split}
\end{equation}
Substituting \eqref{eq: first term} and \eqref{eq: estimate term1}
into \eqref{eq: monotonicity formula 1} gives \eqref{eq: monotonicity formula}.
\end{proof}

In the following three lemmas we establish the necessary monotonicity properties to obtain density estimates in $\Sigma_{a,t}$. 

\begin{lemma}\label{monotonicity outside compact set}

(Monotonicity outside $B^{n+1}_{\rho_0}(O)$)\\
Given $\delta\in(0,1)$, there exists $\rho_0 = \rho_0(M,g,\delta)\geq R_0$, such that for any $\xi\in \Sigma_{a,t}$ with $|\xi|\geq \rho_0$ and any $\rho_1, \rho_2$ with $\delta\leq \rho_1<\rho_2$, if  $(B^{n+1}_{\rho_2}(\xi)\setminus B^{n+1}_{\rho_1}(\xi))\subset M\backslash B^{n+1}_{\rho_0}(O)$ and $B^{n+1}_{\rho_2}(\xi)\cap\partial\Sigma_{a,t}=\emptyset$,
then
\begin{equation}\label{eq: density estimate}
  \theta_{\xi}(\rho_2,a)- \theta_{\xi}(\rho_1,a)\ge -\delta.  
\end{equation}
\end{lemma}
\begin{proof} 
We introduce a function $\Phi$ given by
\begin{equation}
 \Phi(x)=\left\{
\begin{aligned}
&|x|^{-\frac{3}{2}},\ \ \quad \text{for $x\in M\setminus B^{n+1}_{R_0}(O)$ },\\
& 1,\quad\ \quad \ \ \  \ \text{for $x\in B^{n+1}_{R_0}(O)$ }.
\end{aligned}
\right.\nonumber
\end{equation}
Note that $\tau>\frac{n}{2}\geq\frac{3}{2}$. According to the definition of $u$ and $v$ in the proof of Lemma \ref{lmm: monotonicity ineq2}, we see that there exists some uniform $C$ such that 
\begin{equation}\label{eq: bound u,v}
|u|+n|v|\leq C|\xi|\Phi.    
\end{equation}
We consider the following two cases.

\textbf{Case 1:} $B^{n+1}_{\rho_1}(\xi)\subset B^{n+1}_{\rho_2}(\xi)\subset M\backslash B^{n+1}_{\rho_0}(O)$.\\
Using $v=O(x^{-\tau})$ outside $B^{n+1}_{\rho_0}(O)$ and the minimality of $\Sigma_{a,t}$,
we have
\begin{equation}\label{eq: ft for v}
  \rho^{-n}_2
  \int_{\Sigma_{a,t}\cap  B^{n+1}_{\rho_2}(\xi))}|v|d\mu\leq
 \rho_0^{-\tau}\rho^{-n}_2 \mathcal{H}^n(\Sigma_{a,t}\cap B^{n+1}_{\rho_2}(\xi))
  \leq
  C\rho_0^{-\tau}\leq\frac{\delta}{2}
\end{equation}
if we choose $\rho_0\geq C^2\delta^{-1}$. 

\textbf{Claim}  There exists some uniform $C$ depending only on $(M^{n+1}, g)$ such that 
\begin{equation}\label{eq: decay estimate1}
  \int_{\Sigma_{a,t}\cap B^{n+1}_{s}(\xi)}\Phi=\int_{\Sigma_{a,t}\cap B^{n+1}_{s}(\xi)}|x|^{-\frac{3}{2}}
 \leq C|\xi|^{-\frac{3}{2}}s^{n}.   
\end{equation}
We only need to consider the case that $|\xi|\leq 2s$,
otherwise we immediately have 
\begin{equation}\label{eq: decay estimate2}
 \int_{\Sigma_{a,t}\cap B^{n+1}_{s}(\xi)}|x|^{-\frac{3}{2}}
 \leq \mathcal{H}^n(\Sigma_{a,t}\cap B^{n+1}_{s}(\xi))(\frac{|\xi|}{2})^{-\frac{3}{2}}
 \leq C|\xi|^{-\frac{3}{2}}s^n.  
\end{equation}
Since $\Sigma_{a,t}$ is area-minimizing and $(M^{n+1},g)$ is asymptotically flat, we have
\begin{equation}
 \int_{\Sigma_{a,t}\cap B^{n+1}_{s}(\xi)}|x|^{-\frac{3}{2}}
 \leq C\int_{\Sigma_{a,t}\cap B^{n+1}_{s}(\xi)}|x|^{-\frac{3}{2}}d\bar{\mu},  
\end{equation}
where $d\bar{\mu}$ is the area element of $\Sigma_{a,t}$ in $(\mathbf{R}^{n+1}, \delta_{ij})$. 
By Lemma 65 in \cite{EK23} we have
\begin{equation}
\begin{split}
  \int_{\Sigma_{a,t}\cap B^{n+1}_{s}(\xi)}|x|^{-\frac{3}{2}}d\bar{\mu}
  \leq&\int_{\Sigma_{a,t}\cap (B^{n+1}_{|\xi|+s}(O)\setminus B^{n+1}_{|\xi|-s}(O))}|x|^{-\frac{3}{2}}d\bar{\mu}\\
    \leq&  C(|\xi|+s)^{-\frac{3}{2}}+C \Large((|\xi|+s)^{n-\frac{3}{2}}+(|\xi|-s)^{n-\frac{3}{2}}\Large)\\
    \leq&C|\xi|^{-\frac{3}{2}}s^n.
\end{split} 
\end{equation}
 Note $|\xi|\geq \rho_0+\rho_2$ and $\rho_1\geq\delta$.
By choosing 
$\rho_0\geq C\delta^{-3}$, we have
\begin{equation}\label{eq: last term}
\begin{split}
   \int^{\rho_2}_{\rho_1}s^{-1-n}(\int_{\Sigma_{a,t}\cap B^{n+1}_{s}(\xi)}\Phi)ds
    \leq& C|\xi|^{-\frac{3}{2}}\int^{\rho_2}_{\rho_1}s^{-1}ds\\
    \leq& C|\xi|^{-\frac{3}{2}}(\ln\rho_2-\ln\rho_1)\\
    \leq&C|\xi|^{-\frac{3}{2}}(\ln|\xi|-\ln\delta)\\
    <&C^{-1}|\xi|^{-1}\frac{\delta}{2}. 
\end{split}
\end{equation}
Plugging \eqref{eq: bound u,v}, \eqref{eq: ft for v} and \eqref{eq: last term}
into \eqref{eq: monotonicity formula} gives
 the desired estimate.\\
\textbf{Case 2:} $B^{n+1}_{\rho_0}(O)\subset B^{n+1}_{\rho_1}(\xi)\subset B^{n+1}_{\rho_2}(\xi)$.\\
In this case, $\rho_2>\rho_1\geq \rho_0+|\xi|$.
we have
\begin{equation}\label{eq: estimate phi 2}
\begin{split}
      \int_{\Sigma_{a,t}\cap B^{n+1}_{s}(\xi)}\Phi
      =& 
      \int_{\Sigma_{a,t}\cap B^{n+1}_{s} (\xi)\cap B^{n+1}_{R_0} (O)}\Phi
      +\int_{\Sigma_{a,t}\cap (B^{n+1}_{s} (\xi)\setminus B^{n+1}_{R_0} (O))}\Phi\\
      \leq& \mathcal{H}^{n}(\Sigma_{a,t}\cap B^{n+1}_{R_0}(O))+\int_{\Sigma_{a,t}\cap (B^{n+1}_{s} (\xi)\setminus B^{n+1}_{R_0} (O))}|x|^{-\frac{3}{2}}\\
      \leq&C(R_0^{n}+( |\xi|+s )^{n-\frac{3}{2}}),
\end{split}
\end{equation}
where we have used Lemma 65 in \cite{EK23} in the last inequality.
It follows that 
\begin{equation}\label{eq: ft for v 2}
\begin{split}
    \rho^{-n}_2
  \int_{\Sigma_{a,t}\cap  B^{n+1}_{\rho_2}(\xi))}|v|d\mu
  \leq& C\rho_2^{-n}\int_{\Sigma_{a,t}\cap B^{n+1}_{s}(\xi)}\Phi\\
  \leq&C\rho_2^{-n}(R_0^{n}+( |\xi|+\rho_2 )^{n-\frac{3}{2}})\\
  \leq&\frac{\delta}{2}  
\end{split} 
\end{equation}
by our choice of $\rho_0$.
A straightforward calculation shows
\begin{equation}\label{eq: estimate term2}
   \int_{\rho_1}^{\rho_2} s^{-1-n}(|\xi|+s)^{n-\frac{3}{2}}ds
\le C\int_{\rho_1}^{\rho_2}s^{-\frac{5}{2}}ds\le C\rho_1^{-\frac{3}{2}}<C^{-1}|\xi|^{-1}\frac{\delta}{2}. 
\end{equation}
Substituting  \eqref{eq: bound u,v} and  \eqref{eq: estimate phi 2}--\eqref{eq: estimate term2} into \eqref{eq: monotonicity formula} yields \eqref{eq: density estimate}.
\end{proof}
Next, we can obtain a density estimate by establishing an  almost monotonicity across compact set.
\begin{lemma}\label{lem: almost monotonicity across compact set}(Almost monotonicity across compact set)\\
    Let $\delta,\rho_0$ be as in Lemma \ref{monotonicity outside compact set}. Then there exists $L>2$ depending only on $\delta$ such that for any 
    $\xi\in \Sigma_{a,t}\setminus  B^{n+1}_{L\rho_0}(O) $ and  any $\rho_1, \rho_2$ with $\delta\leq \rho_1<\rho_2$, if  $B^{n+1}_{\rho_1}(\xi)\cap B^{n+1}_{\rho_0}(O)=\emptyset$, $B^{n+1}_{\rho_0}(O)\subset B^{n+1}_{\rho_2}(\xi)$ and $B^{n+1}_{\rho_2}(\xi)\cap\partial\Sigma_{a,t}=\emptyset$, then
    \begin{align*}
        \theta_{\xi}(\rho_1,a)\le (1+\delta)\theta_{\xi}(\rho_2,a)+3\delta
    \end{align*}
\end{lemma}
\begin{proof}
By our assumption, $|\xi|-\rho_0\geq\rho_1$ and $|\xi|+\rho_0\leq \rho_2$.
Divide $[\rho_1,\rho_2]$ into $[\rho_1,|\xi|-\rho_0]\cup[|\xi|-\rho_0,|\xi|+\rho_0]\cup[|\xi|+\rho_0,\rho_2]$. Then
    \begin{align*}
         B^{n+1}_{\rho_1}(\xi)\subset B^{n+1}_{|\xi|-\rho_0}(\xi)\subset M\backslash B^{n+1}_{\rho_0}(O)\ \ \text{and}\ \ 
         B^{n+1}_{\rho_0}(O)\subset B^{n+1}_{|\xi|+\rho_0}(\xi)\subset B^{n+1}_{\rho_2}(\xi).
    \end{align*}
    So by applying Lemma \ref{monotonicity outside compact set} we have
    \begin{equation}\label{eq: 3}
    \theta_{\xi}(\rho_1,a)\le \theta_{\xi}(|\xi|-\rho_0,a)+\delta\ \ \ \text{and}
\ \ \ \theta_{\xi}(|\xi|+\rho_0,a)\le \theta_{\xi}(\rho_2,a)+\delta.        
    \end{equation}
    Notice that
    \begin{equation}\label{eq: 4}
        \begin{split}
          \theta_{\xi}(|\xi|-\rho_0,a)\leq \theta_{\xi}(|\xi|+\rho_0,a) \frac{(|\xi|+\rho_0)^{n}}{(|\xi|-\rho_0)^{n}}
         \leq \theta_{\xi}(|\xi|+\rho_0,a) (\frac{L+1}{L-1})^{n}
        \le& (1+\delta)\theta_{\xi}(|\xi|+\rho_0,a) 
        \end{split}
    \end{equation}
    when $L$ is sufficiently large. Combining \eqref{eq: 3} with \eqref{eq: 4} we get the desired estimate.
\end{proof}
\begin{lemma}\label{lem: density estimate1}
  Let $\delta\in(0,\frac{1}{100}),L$ and $\rho_0$  be as above.
  Then there exists $a_0>L\rho_0>0$, such that for any $a>a_0$ and $\xi\in \Sigma_{a,t}\backslash B^{n+1}_{L\rho_0}(O)$ with $d(\xi,\partial\Sigma_{a,t})\geq \frac{a}{2}$, there holds
    \begin{align}\label{eq: density estimate2}
        \theta_{\xi}(\rho,a)\le 1+25\delta
    \end{align}
    for any $ \rho>0$ with $B^{n+1}_{\rho}(\xi)\cap\partial\Sigma_{a,t}=\emptyset$.
\end{lemma}
\begin{proof}
We first adapt the contradiction argument to  show 
 \begin{align}\label{eq: density estimate3}
        \theta_{\xi}(\rho,a)\le 1+8\delta \ \ \ \text{for}\ \ 
        \rho\geq\delta.
    \end{align}
Suppose not, then there exist a sequence of $a_i\rightarrow\infty$, and $\xi_i$ with 
    $|\xi_i|>L\rho_0$ and $\rho_i$ with $\delta\leq\rho_i$ such that
\begin{equation}\label{eq: area growth}
   \theta_{\xi_i}(\rho_i,a_i)> 1+8\delta.  
\end{equation}
We consider the following two cases:

    \textbf{Case 1:} There is a subsequence(still denoted by $i$)
    such that $L\rho_0<|\xi_i|<\epsilon a_i$ where  $\varepsilon$ is given by $1+\delta=(1-\varepsilon)^{-n}$. In this case, choose
    $\tilde{\rho}_i=(1-\epsilon)a_i$, without loss of generality, we may assume
    $\tilde{\rho}_i\geq \rho_i$.
    Then
    by Lemma \ref{monotonicity outside compact set} and  Lemma \ref{lem: almost monotonicity across compact set},  
    \begin{equation}\label{eq: little center}
       \theta_{\xi_i}(\tilde{\rho}_i,a_i)
       \geq (1+\delta)^{-1}\theta_{\xi_i}(\rho_i,a_i)-3\delta
       \geq 1+2\delta.
    \end{equation}
    Since $\Sigma_{a_i,t}$ is the area-minimizing hypersurface in asymtotically flat manifold
    with $\partial\Sigma_{a_i,t}=\partial D_{a_i,t}$, then
    \begin{equation}
    \begin{split}
  \lim_{i\rightarrow\infty}\theta_{\xi_i}(\tilde{\rho}_i,a_i)=&
   \lim_{i\rightarrow\infty}\frac{\mathcal{H}^{n}(B^{n+1}_{\tilde{\rho}_i}(\xi_i)\cap\Sigma_{a_i,t})}{\omega_{n}\tilde{\rho}_i^{n}}\\
   \leq& \lim_{i\rightarrow\infty}\frac{\mathcal{H}^{n}(B^{n+1}_{(1-\varepsilon)a_i}(\xi_i)\cap\Sigma_{a_i,t})}{\omega_{n} a_i^{n-1}}\frac{1}{(1-\varepsilon)^{n}}\\
\leq& 1+\delta.   
\end{split}
    \end{equation}
which is incompatible with \eqref{eq: little center}.  
    
    \textbf{Case 2:} $\epsilon a_i<|\xi_i|$. In this case, by Lemma \ref{monotonicity outside compact set}, we may choose some $\tilde{\rho}_i\geq \varepsilon a_i$ such that
    \begin{equation}\label{eq: rescale radius}
        \theta_{\xi_i}(\tilde{\rho}_i,a_i)\geq1+2\delta.
    \end{equation}
    We adopt the blow down argument.   
Now $\{(M^{n+1}\setminus B^{n+1}_1(O), a_i^{-2}g)\}$ converges to $(\mathbf{R}^{n+1}\setminus\{O\}, \delta_{ij})$
in the $C^{\infty}_{loc}$ sense as $i\rightarrow \infty$. 
Meanwhile, $\{(a_i^{-1}\Sigma_{a_i},a_i^{-1}\xi_i)\}$  converges to an area-minimizing current $(\bar{\Sigma}, \bar{\xi})$
with $\partial \bar{\Sigma}=\partial B^{n+1}_1(O)\cap S_0$. Thus, $\bar{\Sigma}$ must be a part of  the hyperplane. It follows that
\[
\frac{\mathcal{H}^{n}(B^{n+1}_{s}(\bar{\xi})\cap\bar{\Sigma})}{\omega_{n-1}s^{n}}= 1
\ \ \ \text{for}\ \ 0<s<1-|\bar\xi|,
\]
which contradicts with \eqref{eq: rescale radius}.

Next we use Lemma \ref{lm: monotonicity ineq2} to deal with the case $\rho\leq\delta$. By taking some larger $\rho_0$ if necessary we have for $0<\rho\leq\delta$
\begin{equation}\label{eq: area comparison}
     \mathcal{H}^{n}(B^{n+1}_\rho(\xi)\cap\Sigma_{a,t})\leq
(1+\delta)\mathcal{H}^{n}(\mathcal{B}^{n+1}_\rho(\xi)\cap\Sigma_{a,t})
\leq  (1+\delta)^2\mathcal{H}^{n}(B^{n+1}_\rho(\xi)\cap\Sigma_{a,t}),
\end{equation}
where $\mathcal{B}^{n+1}_\rho(\xi)$ denotes the geodesic ball in $(M^{n+1},g)$.
Then, using Lemma \ref{lm: monotonicity ineq2} we have
 \begin{equation}\label{eq: area comparison2}
    \begin{split}
    \theta_{\xi}(\rho,a)=&\frac{\mathcal{H}^{n}(B^{n+1}_{\rho}(\xi)\cap\Sigma_{a,t})}{\omega_{n}\rho^{n}}\\
    \leq & (1+\delta)\frac{\mathcal{H}^{n}(\mathcal{B}^{n+1}_{\rho}(\xi)\cap\Sigma_{a,t})}{\omega_{n}\rho^{n}}\\
    \leq& (1+\delta)^2\frac{\mathcal{H}^{n}(\mathcal{B}^{n+1}_{\delta}(\xi)\cap\Sigma_{a,t})}{\omega_{n}\delta^{n}} \\
    \leq&(1+\delta)^3\frac{\mathcal{H}^{n}(B^{n+1}_{\delta}(\xi)\cap\Sigma_{a,t}) }{\omega_{n}\delta^{n}}.
    \end{split} 
 \end{equation}
 Together \eqref{eq: density estimate3} with \eqref{eq: area comparison2}  gives the desired estimate.
Therefore, we complete the proof.
\end{proof}

Given $\delta\in(0,\frac{1}{10})$ and let $L,\rho_0$ be as in Lemma \ref{lem: almost monotonicity across compact set}. 
Consider a coordinate ball $B^{n+1}_{r_0}(p)\subset M^{n+1}\setminus B^{n+1}_{L\rho_0}(O)$. Then there is an isometric embedding
 $$
 i: (B^{n+1}_{r_0}(p),g)\hookrightarrow \mathbf{R}^{n+k}, ~\text{for some $k>1$}. 
 $$
In conjunction with Theorem 17.6 in \cite{Simon83}, we have the following local monotonicity formula for volume density.
\begin{lemma}\label{lmm: monotonicity ineq1}
 For any coordinate ball $B^{n+1}_{r_0}(p)\subset M^{n+1}\setminus B_{L\rho_0}(O)$, there are positive  constants $\Lambda$ and $\alpha>0$ depending only on geometry of $(B^{n+1}_{r_0}(p),g)$ and the isometric embedding $i$  so that for any area-minimizing hypersurface $\Sigma^{n}$ (may not be smooth) in  $(M^{n+1}, g)$, any $0<\sigma\leq \rho\leq r_0 $, and any $\xi\in \Sigma$ with $B^{n+k}_\sigma (\xi)\subset B^{n+k}_\rho (\xi)\subset B^{n+k}_{r_0}(p)$ we have
 \begin{equation}\label{eq: local monotonicity}
    \sigma^{-n}\mathcal{H}^{n}(\Sigma\cap B^{n+k}_\sigma (\xi))\leq e^{\Lambda r^{\alpha}_0}\rho^{-n}\mathcal{H}^{n}(\Sigma\cap B^{n+k}_\rho (\xi)). 
 \end{equation}
\end{lemma}
In our case, for $\xi\in \Sigma_{a,t}\backslash B^{n+1}_{L\rho_0}(O)$ with $d(\xi,\partial\Sigma_{a,t})\geq \frac{a}{2}$, by taking $\rho$ small enough which may depend on $\xi$, we have
\begin{equation}\label{eq: area ratio}
    \frac{\mathcal{H}^{n}(B^{n+k}_{\rho}(\xi)\cap\Sigma_{a,t})}{\omega_{n}\rho^{n}}
    \leq(1+\delta)\frac{\mathcal{H}^{n}(B^{n+1}_{\rho}(\xi)\cap\Sigma_{a,t})}{\omega_{n}\rho^{n}}\leq1+30\delta.
\end{equation}

\subsection{Curvature estimate for area-minimizing  hypersurfaces  in asymptotically flat manifolds}

Let $B^{n+k}_\rho(x)$ be the ball in $\mathbf{R}^{n+k}$ with centre $x$ and radius $\rho$, $\Omega$ be any  compact domain  with smooth boundary in $(M^{n+1}, g)$, then $(\Omega,g)$ can be isometric embedded in $\mathbf{R}^{n+k}$,  let $\mathbf{B}$  be its second fundamental forms  in $\mathbf{R}^{n+k}$. 

\begin{proposition}\label{prop: generalized mean curvature}
Let $\mathbf{H}$ be the generalized mean curvature of $\Sigma \cap \Omega$ in $\mathbf{R}^{n+k}$, then for $\mu_T$ a.e. $x \in \Sigma$,	 we have
$$
\sup_{\Omega\cap \Sigma}|\mathbf{H}|\leq n\sup_{\Omega}|\mathbf{B}|.
$$
\end{proposition}
\begin{proof}
	Let  $X:\Sigma \to \mathbf{R}^{n+k}$ be any $C^1$- vector field on $\Sigma$ with compact support set in $\Omega$,  then
	$$
	X=X^{\perp}+X^{\top},
	$$
where $X^{\perp}$ and $X^{\top}$	 denote the normal and tangential components of $X$  to $(M^{n+1},g)$ respectively. Since $\Sigma$ is area-minimizing in $(M^{n+1},g)$, we have
\[
\int div_{\Sigma} X^{\top} d\mu_T=0.
\]
Let $\{\tau_i\}_{1\leq i\leq n}$ be any orthonormal basis for the approximate tangent space $T_x \Sigma$. Then 
$$
 div_{\Sigma}X^{\perp}=\sum_{i=1}^{n}\langle\tau_i, \nabla_{\tau_i} X^{\perp}\rangle
 =\sum_{i=1}^{n}\langle\mathbf{B}(\tau_i, \tau_i), X^{\perp}\rangle
 $$
and
$$
\int div_\Sigma X d\mu_T=\int \langle\sum_{i=1}^{n-1}\mathbf{B}(\tau_i, \tau_i), X \rangle d\mu_T.
$$
Hence,
$$
\mathbf{H}=\sum_{i=1}^{n}\mathbf{B}(\tau_i, \tau_i),
$$
which implies 
$$
\sup_{B^{n+k}_\rho(x)\cap \Sigma}|\mathbf{H}|\leq n\sup_{B^{n+k}_\rho(x)\cap M^{n+1}}|\mathbf{B}|.
$$
\end{proof}
For given $\delta$,  by Proposition \ref{prop: generalized mean curvature}, we can take 
 $\rho$ small enough such that 
 \[
 \left(\rho^{p-n}\int_{B^{n+k}_\rho(\xi)\cap \Sigma_{a,t}}|\mathbf{H}|^p\right)^{\frac{1}{p}}\leq \delta
 \]
 for $p> n$. Applying the Allard's regularity Theorem to $\Sigma_{a,t}\cap B^{n+k}_{\rho}(\xi)\subset M\subset\mathbf{R}^{n+k}$, we conclude: 
\begin{corollary}\label{cor: smooth outside compact set}
    Let the notation be as in the Lemma \ref{lem: density estimate1}. Then for all $a>a_0$, $\Sigma_{a,t}$ is smooth at any point $\xi\in \Sigma_{a,t}\backslash B^{n+1}_{L\rho_0}(O)$ with $d(\xi,\partial\Sigma_{a,t})\geq \frac{a}{2}$.
\end{corollary}

Now we are ready to use the standard point-picking argument and density estimate Lemma \ref{lem: density estimate1} to derive the following curvature estimate.

\begin{lemma}\label{lem: uniform estimate for A}
For given $t\in\mathbf{R}$, there exist  uniform constants $R_1$ and $R_2$, such that for any $a>R_1$ and
    any point $x\in\Sigma_{a,t}\setminus B^{n+1}_{R_2}(O)$ with $d(x,\partial \Sigma_{a,t})\geq\frac{3a}{4}$,
    it holds $|A|(x)\leq \frac{C}{|x|}$ for some constant $C$ depending only on $(M^n, g)$. 
\end{lemma}
\begin{proof}
    Suppose the lemma is not true, then there exist a sequence of $a_k\rightarrow\infty$ and $x_k\in\Sigma_{a_k,t}$
    with $|x_k|\rightarrow\infty$ and $d(x_k, \partial\Sigma_{a_k,t})\geq \frac{3a_k}{4}$, such that 
    \[
\max_{x\in\Sigma_{a_k,t}\cap B^{n+1}_{\frac{|x_k|}{2}}(x_k)}(\frac{|x_k|}{2}-|x-x_k|)|A|(x):=c_k
\to\infty.
    \]
    Assume the maximum is attained at $y_k$ and set $r_k=\frac{|x_k|}{2}-(|y_k-x_k|)$.
    Rescale the asymptotically flat metric $g_{ij}^k$ in $B^{n+1}_{r_k}(y_k)$ to the metric $\tilde{g}_{ij}^k$
    using the transformation $\Phi$ given by $x=y_k+\frac{r_k}{c_k}\tilde{x}$, i.e.,
    \[\tilde{g}_{ij}^k=\Phi^*g_{ij}^k.
    \]
    Then $\tilde{g}_{ij}^k$ is defined in $B^{n+1}_{c_k}(O)$ and converge to the Euclidean metric $\delta_{ij}$ on $\mathbf{R}^{n+1}$ locally uniformly in $C^2$ sense by asymptotic flatness, since $|x_k|\to\infty$. 
    The second fundamental form of the rescaled hypersurface $\Phi^*(\Sigma_t\cap B^{n+1}_{r_k}(y_k))$ satisfies $|\tilde{A}|\leq 1$
    with $|\tilde{A}|(O)=1$.
Thus, by passing to a subsequence, we have $\Phi^*(\Sigma_t\cap B^{n+1}_{r_k}(y_k))$ converge to a
smooth
area-minimizing hypersurface $\tilde{\Sigma}$. 

On the other hand, by Lemma \ref{lem: density estimate1}, there also exists a sequence of $\delta_k\rightarrow 0$ such that
\begin{equation}\label{eq: asy density esimate}
        \frac{\mathcal{H}^{n}(B^{n+1}_{\rho}(y_k)\cap\Sigma_{a_k,t})}{\omega_{n}\rho^{n}}\le 1+25\delta_k,\mbox{ for }\rho>0 \ \text{with}
        \ \ B^{n+1}_{\rho}(y_k)\cap\partial\Sigma_{a_k,t}=\emptyset.
    \end{equation}
It follows that
\[
\frac{\mathcal{H}^{n}(B^{n+1}_{R}(O)\cap\tilde{\Sigma})}{\omega_{n}R^{n}}\equiv 1,
\mbox{ for any}\ \ R> 0.
\]
By the monotonicity formula for minimal hypersurface in Euclidean space,  $\tilde{\Sigma}$ must be the hyperplane, which contradicts with $|A|_{\tilde{\Sigma}}(O)=1$.   
\end{proof}
Let $\Sigma_{r, t}=spt(T_{r,t})$ be the area-minimizing current.
Now we can use catenoidal-type
barriers devised by Schoen and Yau\cite{SY81} to give a nice control of $\Sigma_{r,t}$.

For $\Lambda>1$ and $1<\beta<\tau-1$, consider the rotational symmetric function $f(r)$ in $\mathbf{R}^{n}\setminus B^{n}_1(O)$ defined by
\begin{equation}\label{eq: symmetric function}
f(r)=\Lambda \int_{r}^{+\infty}(s^{2\beta+2}-\Lambda^2)^{-\frac{1}{2}}ds,\ \ 
\text{where}\ \  r>\Lambda.
\end{equation}
Denote the graph of $f+L$ by $graph_{f+L}$ for $L\in \mathbf{R}$. Using the similar argument as in \cite{HSY24}, we can show
for given $t$, $\Sigma_{r, t}$ lies between $graph_{\pm f+t}$ for any $r\gg 1$.
Since $(M^{n+1},g)$ has bounded geometry and $\Sigma_{r, t}\cap \partial C_{\Lambda}$ lies between two hyperplanes $S_{t\pm C\Lambda}$, by applying the monotonicity formula for minimal hypersurfaces we conclude that $\Sigma_{r, t}\cap C_{\Lambda}$ is bounded between two hyperplanes $S_{t\pm C_1\Lambda}$ for some uniform $C_1$.
With this at hand we're able to construct the area-minimizing hypersurfaces in $(M,g)$. Choose a sequence of $r_i$ with $r_i\rightarrow\infty$, then for some fixed $t$, $\Sigma_{r_i,t}$ intersects with a fixed compact set that depends only on $t$.
Hence,  as area-minimizing $n$-rectifiable current, $\Sigma_{r_i,t}$ converges to an area-minimizing $n$-rectifiable current $\Sigma_t$   as $r_i\to +\infty$ by taking a subsequence if necessary.  As a consequence, we have
 \begin{corollary}\label{cor: bound ms}
 Let $\Sigma_t$ be the area-minimizing $n$-rectifiable current as above. Then $\Sigma_t$ lies between $S_{t\pm C_1\Lambda}$ for some uniform $C_1$ and can be bounded by $graph_{\pm f+t}$ outside the cylinder $C_{\Lambda}$.
 Moreover, $\Sigma_t$ is smooth outside a fixed compact set.
 \end{corollary}
Now we are ready to show that 
\begin{proposition}\label{thm: uniqueness of tangent cone}
    Let $\Sigma_t$ be an area-minimizing hypersurface as above. Then the tangent cone at the infinity is regular and unique.
\end{proposition}
	\begin{proof}
	  Since $\Sigma_t$  has polynomial volume growth and $|A|(x)\leq\frac{C}{|x|}$,
      we have $r^{-1}_k\Sigma_t$ converge to some stable minimal cone $\Gamma$ in $\mathbf{R}^{n+1}$ for any sequence
      $\{r_k\}$ with $r_k\to\infty$.  Due to the fact that $|A|(x)\leq\frac{C}{|x|}$, we know $\Gamma$ is smooth.  As $\Sigma_t$ lies between two hyperplanes, then 
      by Theorem 36.5 \cite{Simon83}, $\Gamma$ must be the hyperpalne $\{x_{n+1}=0\}$.  Hence, the tangential cone of $\Sigma_t$ at the infinity is unique. Moreover $\Gamma$ has multiplicity one.  
	\end{proof}
Apply
Theorem 5.7 in \cite{Simon84}(see also the discussion given at pp. 269-270),  $\Sigma_t$ can be represented by some graph 
$\{(y,u_t(y):y\in \mathbf{R}^{n}\}$ outside some compact set with
$|u_t(y)-t|\rightarrow 0$ as $|y|\rightarrow \infty$.
We end this subsection by giving the proof of Theorem \ref{thm: foliation}.
	
	\begin{proof}[Proof Theorem \ref{thm: foliation}]
First we deal with the case that $M$ has only one end. By the argument before, for given $t\in \mathbf{R}$, there exists an area-minimizing hypersurface $\Sigma_t$ asymptotic to the hyperplane $S_t$. The asymptotic estimate for the graph function $u_t$ follows from Proposition 9 in \cite{EK23}. We can use the same argument in \cite{HSY24} to show $\Sigma_t$ is the unique area-minimizing hypersurface asymptotic to the hyperplane $S_t$ for $|t|\gg1$ and these $\Sigma_t$ form $C^1$ foliation (see Proposition 2.13 and
Proposition 2.14 in \cite{HSY24} for details).

 In the scenario where the manifold $M$ possesses arbitrary ends, it is important to note that the Plateau problem may not always admit a solution for values of $t$ where $|t|$ is small. Nonetheless, by employing the catenoidal-type barrier argument delineated prior to Corollary \ref{cor: bound ms}, we can still conclude that $\Sigma_{r_i,t}$ lies between $S_{t\pm C\Lambda}$ for $|t|>2C\Lambda$. Then we can follow a similar argument to get the desired result.
\end{proof}

\section{Desingularization of area-minimizing hypersurface in 8-manifolds}

In this section, we want to prove Theorem \ref{prop: rigidity for minimal surface} and Theorem \ref{thm: georch free of singularity}. We begin by recalling the notion of \textit{strongly stable}.

\begin{definition}\label{defn: strongly stable hypersurface}
    Let $(M^{n+1},g)$ be complete AF manifold with arbitrary ends and let $\Sigma$ be an area-minimizing boundary in $M^{n+1}$, which is asymptotic to a hyperplane in the AF end $E$. Then $\Sigma\cap E$ is an AF end for $\Sigma$. We say $\Sigma$ is stable under the asymptotically constant variation (or strongly stable), if for any function $\phi\in \Lip_{loc}(\Sigma\backslash\mathcal{S})$ which is either compactly supported or asymptotically constant in the sense of Definition \ref{defn: strong test function}, there holds
	\begin{equation}\label{eq: strongly stability}
		\int_{\Sigma\backslash\mathcal{S}}|\nabla\phi|^2-(Ric(\nu,\nu)+|A|^2)\phi^2\geq 0. 
	\end{equation}
    where $\nu$ is unit normal vector of $\Sigma^{n}$ in $(M^{n+1},g)$.
\end{definition}

\begin{remark}
    Definition \ref{defn: strongly stable hypersurface} exhibits a slight modification from the conventional definition, where $\phi$ is typically required to be a smooth function defined on the entire ambient space. Nevertheless, this adapted definition suffices for our purposes, as it is enough for the application of Theorem \ref{thm:pmt with singularity4}.
\end{remark}

Next, we prove two lemmas for codimension $1$ minimizing current with non-empty singular set.	
	
        \begin{lemma}\label{lem: |A|>0}
            Let $(M^{n+1},g)$ be a Riemannian manifold and $\Sigma$ be an area-minimizing hypersurface with singular set $\mathcal{S}$ such that $\Sigma$ is smoothly embedded outside $\mathcal{S}$. Suppose $\mathcal{S}\ne\emptyset$, then for any $r>0$, there exists a point $p\in \mathcal{B}_r^{n+1}(\mathcal{S})\cap reg\Sigma$ that satisfies $|A|^2(p)>0$.
        \end{lemma}
        \begin{proof}
            The proof follows from analyzing nontrivial tangent cone at singular set, which has also been carried out in \cite{CLZ24}. Here we just include a proof for completeness.
            
            First, due to \cite[Theorem 27.8]{Simon83}, we deduce that for each $x\in\Sigma$, there exists an open set $W$ containing $x$, such that $\Sigma\cap W$ is an area-minimizing boundary in $W$. Thus, we can reduce the problem to the case of area-minimizing boundaries. We isometrically embed a compact domain of $(M,g)$  which has non-empty intersection with $\mathcal{S}$ into an Euclidean space $\mathbf{R}^{n+k}$, and assume the conclusion is not true, then $|A|=0$ in a neighborhood of $\mathcal{S}$. It follows from standard dimension reduction \cite[Appendix A]{Simon83} that at some point $q\in\mathcal{S}$, the tangent cone $C_q\subset T_qM$ splits as $\mathcal{C}^d\times \mathbf{R}^{n-d+k}$, where $\mathcal{C}^d$ is a nontrivial minimizing cone in $\mathbf{R}^{d+1}$ with isolated singularity. 

            Denote $\Sigma_i = \eta_{q,\lambda_i}\Sigma$ $(\lambda_i\to 0)$ to be a blow up sequence at $q$ (Here $\eta_{q,\lambda_i}(x) = \lambda_i^{-1}(x-q)$). Then by \cite[Theorem 34.5]{Simon83} $\Sigma_i$ converges to $C_q$ in multiplicity one in varifold sense, as  $\Sigma$ is an area-minimizing boundary. At the same time $M_i = \eta_{q,\lambda_i,}M$ converges to $T_qM$ in $C^1$ sense. The monotonicity formula implies the convergence also holds in Hausdorff sense. Combined with Allard's regularity theorem and standard results in PDE theory, we know that the regular part of $\Sigma_i$ converges to the regular part of $C_q$ in $C^2$. This implies the second fundamental form of $C_q\subset T_qM$ is identically zero on its regular part, which is a contradiction.
        \end{proof}

        \begin{lemma}\label{lem: connectedness}
            Let $\Sigma$ be in the previous lemma. Assume $\mathcal{S}$ is an isolated singular set, and assume further that the tangent cone of $\Sigma$ at each point has isolated singularity. Then for each $q\in \mathcal{S}$, there is a $\delta_0>0$, such that for any $r<\delta_0$, $\partial \mathcal{B}_{r}^{n+1}(q)\cap\Sigma$ is connected.
        \end{lemma} 
        \begin{proof}
            For simplicity, we assume \( M = \mathbf{R}^{n+1} \), as the general case can be handled using a similar approach. Without loss of generality, we may assume that \( \Sigma \) is an area-minimizing boundary, analogous to Lemma \ref{lem: |A|>0}. Suppose, for contradiction, that there exists a sequence \( r_i \to 0 \) for which the statement of the lemma fails. By the reasoning in Lemma \ref{lem: |A|>0}, the rescaled surfaces \( \Sigma_i = \eta_{q,r_i}(\Sigma) \) converge locally smoothly to a tangent cone \( C_q \subset T_qM \) in the smooth part. Consequently, the sets \( \eta_{q,r_i}\left(\Sigma \cap \left(B_{10r_i}^{n+1}(q) \setminus B_{\frac{1}{10}r_i}^{n+1}(q)\right)\right) = \Sigma_i \cap \left(B_{10}^{n+1} \setminus B_{\frac{1}{10}}^{n+1}\right) \) converge locally smoothly to \( C_q \cap \left(B_{10}^{n+1} \setminus B_{\frac{1}{10}}^{n+1}\right) \). Since \( C_q \cap \left(B_{10}^{n+1} \setminus B_{\frac{1}{10}}^{n+1}\right) \) has a compact closure, the convergence is smooth over this region. For sufficiently large \( i \), \( \Sigma_i \) can be expressed as a graph over \( C_q \cap \left(B_{10}^{n+1} \setminus B_{\frac{1}{10}}^{n+1}\right) \). Given that \( C_q \) has an isolated singularity and any two minimal hypersurfaces in \( S^n \) must intersect, it follows that \( C_q \cap \left(B_{10}^{n+1} \setminus B_{\frac{1}{10}}^{n+1}\right) \) is connected. Therefore, \( \Sigma_i \cap \partial B_1^{n+1} \) must also be connected for sufficiently large \( i \), leading to a contradiction.
        \end{proof}

We now verify that \(\Sigma\) in Theorem \ref{prop: rigidity for minimal surface} satisfies most of the assumptions of Theorem \ref{thm:pmt with singularity4}, and it satisfies all the assumptions provided $\mathcal{S}\ne\emptyset$.

For the almost manifold condition, conditions (1) and (2) are straightforward. For condition (3), we take \(B(p,r)\) to be the extrinsic ball \(\mathcal{B}_r^{n+1}(p)\). Conditions (4) and (5) follow directly from the results in Appendix A. 

Regarding the conditions specified in the statement of Theorem \ref{thm:pmt with singularity4}, conditions (1), (2), (3), and (4) are naturally satisfied since \(\mathcal{S}\) is isolated. Condition (5) is guaranteed by Lemma \ref{lem: connectedness}, condition (6) follows from the monotonicity formular for minimal hypersurface. 

   \begin{proof}[Proof of Theorem \ref{prop: rigidity for minimal surface}]
   
        (1) Assume $\mathcal{S}\ne\emptyset$. By the paragraph above we see $\Sigma$ satisfies all the conditions (1)-(6) of Theorem \ref{thm:pmt with singularity4}.
        Now $\Sigma^n$ has $\frac{1}{2}$-scalar curvature not less than $\frac{1}{2}(R_g+|A|^2)$ in the strong spectral sense due to its strong stability, which implies $m_{ADM}(\Sigma,g_{\Sigma},E\cap\Sigma)>0$. However, the condition $\tau>n-2$, in conjunction with \cite[Lemma 22]{EK23} implies that $m_{ADM}(\Sigma,g_{\Sigma},E\cap\Sigma)=0$, which is a contradiction. Therefore, we have $\mathcal{S}=\emptyset$. The rigidity of $\Sigma$ as stated in the proposition is an immediate consequence of \cite[p.14-15]{Carlotto16},\cite[Proposition 30]{EK23}.

        (2) We adopt a contradiction argument. Suppose $\Sigma\not\subset U_1$, then by the strong stability of $\Sigma^n$ we know that $(\Sigma^n,g_{\Sigma})$ has a $\frac{1}{2}$-scalar curvature not less than $\frac{1}{2}(R_g+|A|^2)$ in the strong spectral sense.

    First, lets assume $n+1\le 7$. Since $R_g>0$ on $\Sigma\cap (U_2\backslash U_1)$, it follows from Theorem \ref{thm: PMT arbitrary end spectral} (1) that $m_{ADM}(\Sigma,g_{\Sigma},E\cap\Sigma)>0$. However, the condition $\tau>n-2$, in conjunction with \cite[Lemma 22]{EK23} implies that $m_{ADM}(\Sigma,g_{\Sigma},E\cap\Sigma)=0$, a contradiction. For the case $n+1 = 8$, we use the same argument and apply Theorem \ref{thm:pmt with singularity4} to obtain the desired result. The remaining part is a consequence of Corollary \ref{prop: rigidity for minimal surface} (1).

    When $\Sigma\backslash\mathcal{S}$ is spin, we can apply \cite[Theorem 1.4]{Zei20} to establish that the positive mass theorem for ALF manifolds with arbitrary ends, as demonstrated in Proposition \ref{prop: PMT for S^1 symmetric ALF}, holds in all dimensions, provided that $\bar{M}$ in the statement of Proposition \ref{prop: PMT for S^1 symmetric ALF} is spin. Consequently, we can proceed with the same line of reasoning to obtain the desired result.
    \end{proof}

\begin{proof}[Proof of Theorem \ref{thm: georch free of singularity}]
    Assume $\mathcal{S}\ne \emptyset$.  If $n\le 7$, then by the same argument as in the proof of Theorem \ref{prop: rigidity for minimal surface}, we are able to verify that $\Sigma$ satisfies the condition of Theorem \ref{thm:non-existence psc on singular space1}. This is a contradiction. For the case that $\Sigma\backslash \mathcal{S}$ is spin, we can apply \cite[Theorem 1.1]{LSWZ24} to obtain the result.
\end{proof}

    \begin{proof}[Proof of Corollary \ref{cor:8dim georch}]
    Since the case for $n\le 7$ has been established as classical work by Schoen and Yau \cite{SY79b}, our focus is directed towards the specific case where $n=8$. Decompose $\mathbf{T}^8 = \mathbf{T}^7\times \mathbf{S}^1$. Without loss of generality, we can assume that $f$ is transversal to $\mathbf{T}^7\times \{1\}$, so $\Sigma_0 = f^{-1}(\mathbf{T}^7)$ is an embedded regular submanifold in $M$. By classical results in geometric measure theory we can find a homological minimizing integer multiplicity current $\tau$ with $\partial R = \tau-\|\Sigma_0\|$, where $R$ is an integer multiplicity $8$-current. From the regularity theory we know that the support of $\tau$, denoted by $\Sigma$, has isolated singular set $\mathcal{S}$. By constancy theorem, the restriction of $\tau$ on $M\backslash\mathcal{S}$ is equal to $m\| \Sigma\backslash \mathcal{S}\|$ for some $m\in\mathbf{Z}$.

    Consider the composition of the following sequence of maps
    \begin{align*}
        F:\Sigma\hookrightarrow M^8\longrightarrow \mathbf{T}^8\longrightarrow \mathbf{T}^7,
    \end{align*}
    where the last map is the projection map from $\mathbf{T}^8$ to its $\mathbf{T}^7$-factor. Let $\omega\in\Omega_c^7(\mathbf{T}^7\backslash F(\mathcal{S}))$ be a differential form with integration $1$ on $\mathbf{T}^7\backslash F(\mathcal{S})$. We have
    \begin{align*}
        m\deg(F|_{\Sigma\backslash \mathcal{S}})\int_{\Sigma\backslash\mathcal{S}}F^*\omega =& m\int_{\Sigma\backslash\mathcal{S}}F^*\omega = \tau(F^*\omega) = \|\Sigma_0\|(F^*\omega)\\
        = &\int_{\Sigma_0}F^*\omega = \deg (F|_{\Sigma_0}) = \deg f\ne 0.
    \end{align*}
    Using the definition in Appendix B, we see that $M$ admits a nonzero degree map to $\mathbf{T}^7$ in the sense of Definition \ref{defn: degree 2}. This contradicts Theorem \ref{thm: georch free of singularity}.
\end{proof}

\section{Behavior of free boundary minimizing hypersurfaces in cylinders in AF manifolds}
\subsection{An effective version of PMT in dimension 8}

$\quad$

 Let $(M, g)$ be as in Theorem \ref{thm: 8dim Schoen conj}.
 We introduce a sequence of \textit{free boundary problems with inner obstacle in cylinders} to effectively compensate for the lack of compactness in minimal surfaces caused by arbitrary ends.

Recall $C_{R}$ is the coordinate cylinder  with radius $R$ (see the beginning of Section 5 for its definition).
Let
\begin{align*}
    \partial E\subset V_1\subset V_2\subset\dots
\end{align*}
be any  compact exhaustion of  $M\setminus E$ and let $j:\mathbb{R}\longrightarrow \mathbb{Z}_+$ be any non-decreasing index function with 
\begin{align*}
    \lim_{R\to\infty} j(R) = \infty
\end{align*}
 We define(see Figure $6$)
\begin{equation}\label{eq: 79}
\begin{split}
\mathcal{F}_{R}&:=\{\Sigma=\partial \Omega \backslash \partial C_{R}: \Omega \subset C_{R}\cup V_{j(R)} \text{  is a Caccioppoli set that satisfies}\\
	&\text{ $C_{R} \cap \{z\leq a\}\subset \Omega $ and $\Omega \cap \{z\geq b\}= \emptyset  $ for some $-\infty< a\leq b<\infty$} \}.
	\end{split}
\end{equation}
\begin{figure}
    \centering
    \includegraphics[width=12cm]{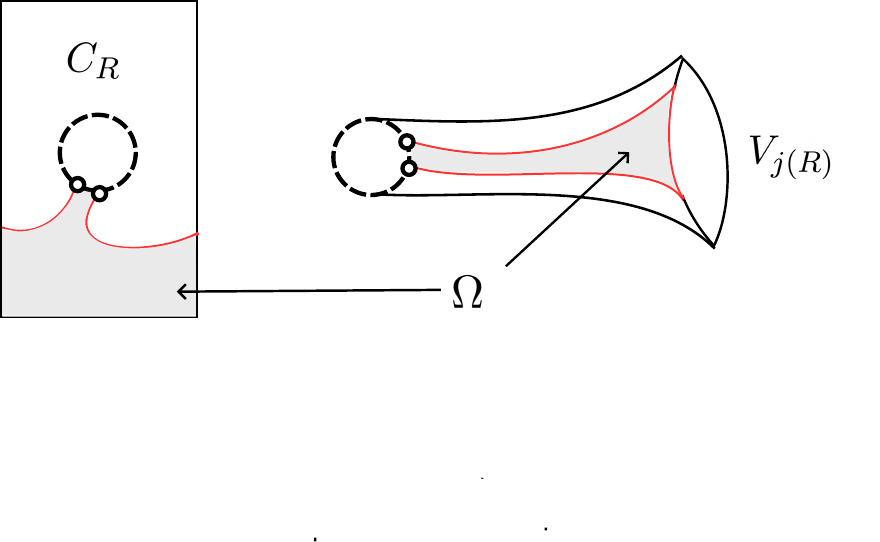}
    \caption{This figure demonstrates the free boundary problem with inner obstacle defined by \eqref{eq: 79}. The left rectangle denotes part of the cylinder $C_R$; the gray area denotes $\Omega$; the red line denotes $\Sigma$,  and $\partial V_{R_i}$ is the inner obstacle lying in arbitrary ends.}
    \label{f5}
\end{figure}
 We now proceed to present the proof of Theorem \ref{thm: 8dim Schoen conj}, adopting the strategy in \cite{HSY24}. The new point lies in the application of Theorem \ref{prop: rigidity for minimal surface}, which enables us to establish the existence of the strongly stable hypersurface $\Sigma_p$ passing through $p$ for any $p\in E$ with $|z(p)|\gg1$.

        \begin{proof}[Proof of Theorem \ref{thm: 8dim Schoen conj}]
         Suppose there is a sequence of compact exhaustion  $ \partial E\subset V_1\subset V_2\subset\dots$ of  $M\setminus E$  and a sequence $\{R_i\}$ with $R_i\rightarrow\infty$ as $i\rightarrow\infty$ such that the corresponding free boundary minimizing hypersurfaces with inner obstacle $ \Sigma_i:=\Sigma_{R_i}=\partial \Omega_i\setminus \partial C_{R_i}$ exist and satisfy $\Sigma_i\cap K\neq\emptyset$ for some compact $K$. By \cite[Corollary 12.27]{Mag12}, $\Sigma_i$ converges subsequentially to a limit area-minimizing boundary $\Sigma$. We divide our proof into 3 steps.
        \textbf{Step 1:} \textit{$\Sigma\cap E$  is asymptotic to a hyperplane  which is parallel to $S_0$ at infinity in Euclidean sense, and $\Sigma$ is strongly stable in the sense of Definition \ref{defn: strongly stable hypersurface}.} 

        For $A\subset M$, we define the height $Z(A)$ of $A$ by
\begin{equation}
    Z(A)=\left\{
    \begin{aligned}
     &\sup_{p\in A\cap E}z(p)-\inf_{p\in A\cap E}z(p),\ \ A\cap E\neq\emptyset,\\
    &0,\ \ A\cap E=\emptyset.   
    \end{aligned}
    \right.
\end{equation}
Then, by the same argument as in \cite[Lemma 3.2]{HSY24}, $Z(\Sigma_i)=o(R_i)$.
We can use the same argument as in  Lemma \ref{lem: uniform estimate for A} to show the existence of a constant $C$ such that  $|A|(x)\leq \frac{C}{|x|}$ for $x\in\Sigma\cap E$ with $|x|\gg1$. By Proposition \ref{thm: uniqueness of tangent cone}, the tangent cone of $\Sigma$ at infinity is regular and unique. Thus, we can use  the argument in the proof of Proposition 9 in \cite{EK23} to show that there exists some rotation $S\in SO(n+1)$ and constant $a\in \mathbf{R}$ such that  outside some compact set, $S(\Sigma\setminus B^{n+1}_{r_o}(O))=\{(y, u(y)):y\in\mathbf{R}^{n}\} $ with $u(y)$ satisfying $|u(y)-a|=O(|y|^{1-\tau+\varepsilon})$ for any $\varepsilon>0$.  Note that $Z(\Sigma_i)=o(R_i)$. Then  the rotation $S$ must be  the identity. Hence, $\Sigma\cap E$ is
asymptotic to a hyperplane  which is parallel to $S_0$ at infinity in Euclidean sense.
The strong stability of $\Sigma$ follows from the calculation in \cite{Schoen1989}, see also \cite[Lemma 3.6]{HSY24}.

$\quad$
       
        \textbf{Step 2:} We can construct  a family of  area minimizing boundaries in a  new class of free boundary under perturbed metrics.
        
         Using Theorem \ref{prop: rigidity for minimal surface}, we can show that $\Sigma\subset U_1$ and $\Sigma$ is isometric to $\mathbf{R}^{n}$. Since $\Sigma$ and $\Sigma_{i}$ have multiplicity one and $\Sigma_i$ converges to $\Sigma$ in varifold sense, using the Allard's regularity theorem, we know that for any region $L$ with compact closure, $\Sigma_i\cap L$ must be smooth for sufficiently large $i$. Note that $\Sigma$ separates the AF end $E$ into two parts: the upper part $E_+$ and the lower part $E_-$. Let $p$ be a point in $E_+$ with $|z(p)|\geq T_0$, where $T_0$ is given by Theorem \ref{thm: foliation}. Then as in \cite[Lemma 3.3]{HSY24}, for any $r_0>0$, there exists $0<r<r_0$, an open set $W\subset M$ with compact closure that satisfies $W\cap \Sigma\ne\emptyset$, and a family of Riemannian metrics $\lbrace g(s)\rbrace_{s\in[0,1]}$, such that the following holds

            (1) $g(s)\to g$ smoothly as $s\to 0$,

            (2) $g(s)=g$ in $M\backslash W$,

            (3) $g(s)<g$ in $W$,

            (4) $R(g(s))>0$ in $\lbrace x\in W: \dist(x,p)>r\rbrace$,

            (5) For $i$ sufficiently large, $\Sigma_i$ is weakly mean convex and strictly mean convex at one interior point under $g(s)$ with respect to the normal vector pointing into $E_-$.

        Let $C_{R_i}^+$ denote the closure of the region in $C_{R_i}$ that lies  above $\Sigma_i\cap E$ and
         \begin{equation}\label{eq:Gr}
        	\begin{split}
        \mathcal{G}_{R_i} = &\lbrace \Sigma = \partial\Omega\setminus  \partial C_{R_i}^+  : ~\Omega \mbox{ is a Caccioppoli set in } C_{R_i}^+\cup V_i,\\
      &\Sigma_i\subset\Omega\quad\text{and}\quad\Omega\cap\{z\ge b\} = \emptyset \mbox{ for some }b\rbrace
      \end{split}
    \end{equation}
    Similar to \cite[Lemma 3.4]{HSY24}, we are able to prove that there exists a free boundary minimal hypersurface $\Sigma_i(s)$ which minimizes the volume in $\mathcal{G}_{R_i}$ under the metric $g(s)$.  Here, a slight difference is that $\Sigma_i$ may have singularity. However, thanks to \cite[Theorem B.1]{Wang24}, $\Sigma_i$ remains as an effective barrier. Moreover, by the argument in \cite[Lemma 3.5]{HSY24}, for a fixed $s$, there exists a compact set $W_0\subset W$, such that for all sufficiently large $i$, there holds $\Sigma_i(s)\cap W_0\ne\emptyset$.
    Here the free boundary minimal surfaces $\Sigma_i(s)$
		may also have singularities. Now by the standard result in geometric measure theory, $\Sigma_i(s)$ will converge to a limit area-minimizing hypersurface $\Sigma(s)$.

$\quad$
        
    \textbf{Step 3:} By \cite[Lemma 3.6]{HSY24}, $\Sigma(s)\cap E$ is asymptotic to a hyperplane  which is parallel to $S_0$ at infinity in Euclidean sense and $\Sigma(s)$ is strongly stable. By the desingularizing Theorem \ref{prop: rigidity for minimal surface}, we conclude that $\Sigma(s)\cap B^{n+1}_r(p)\ne\emptyset$. By letting $s\rightarrow 0$ and $r\rightarrow 0$, we obtain a strongly stable area-minimizing hypersurface $\Sigma_p$ in $(M^{n+1}, g)$ that passes through $p$. The uniqueness result in Theorem \ref{thm: foliation} shows that $\Sigma_p$ coincides with $\Sigma_t$. Again due to Theorem \ref{prop: rigidity for minimal surface}, $\Sigma_p$ is isometric to $\mathbf{R}^{n}$ with $R_g = |A|^2 = Ric(\nu,\nu) = 0$ along $\Sigma$. Combined with \cite[Proposition A.3]{HSY24}, this implies that the region  in the AF end $E$ that lies above $\Sigma_{T_0}$ is isometric to $\mathbf{R}^{n+1}_+$. Similarly, the region  in the AF end $E$ that lies below $\Sigma_{-T_0}$ is also isometric to $\mathbf{R}^{n+1}_-$. By the definition of the ADM mass, we see that 
    $m_{ADM}(M,g,E)=0$,  which contradicts our assumption that $m_{ADM}(M,g,E)\neq0$.
        \end{proof}

\subsection{Proof of the positive mass theorem for smooth AF $8$-manifolds with arbitrary ends}

\begin{proof}[Proof of Theorem \ref{thm: pmt8dim}]
    For the inequality part, we argue by contradiction. Assume that $m_{ADM}(M,g,E)<0$. By \cite[Proposition 3.2]{Zhu23}, we may assume that $(M^{n+1},g)$ is asymptotically Schwarzschild and $E\backslash B^{n+1}_{r_1}$ is harmonically flat for some $r_1>1$, i.e. $g_{ij} = u^{\frac{4}{n-1}}\delta_{ij}$ in $\mathbf{R}^{n+1}\backslash B^{n+1}_{r_1}(O)$, where $u$ is harmonic in $\mathbf{R}^{n+1}\backslash B^{n+1}_{r_1}(O)$. We have the expansion
    \begin{align*}
        u = 1+\frac{m}{2|x|^{n-1}} + O(|x|^{-n})
    \end{align*}
    in the AF end $E$. Using \cite[Lemma 5.1]{LUY21}, there exists $\varphi\in C^{\infty}(M)$ and $r_2>r_1$, such that

    \begin{equation}\label{eq: 80}
\left\{
\begin{aligned}
&\varphi = 1-\frac{m}{4|x|^{n-1}}\quad \text{for} \quad|x|>r_2,\\
&\varphi = \mbox{const}>0\quad \text{for} \quad x\in M\backslash (E\backslash B_{r_1}(O)),\\
&\Delta_g\varphi\le 0 \quad \text{in} \quad M,\\
&\Delta_g\varphi< 0 \quad \text{in} \quad E\backslash B_{r_2}(O).
\end{aligned}
\right.
\end{equation}
Let $\tilde{g} = \varphi^{\frac{4}{n-1}}g$, then we have $R_{\tilde{g}}>0$ in $E\backslash B_{r_2}(O)$ and
    \begin{align}\label{eq: 50}
        \tilde{g}_{ij} = \varphi^{\frac{4}{n-1}}\delta_{ij} \quad\text{ in }\quad E\backslash B^{n+1}_{r_2}(O),
    \end{align}
    where 
    \begin{align*}
        \varphi(x) = (1+\frac{m}{2|x|^{n-1}})(1-\frac{m}{4|x|^{n-1}})+O(|x|^{-n})
    \end{align*}
    for some $m<0$. Note that the second item in \eqref{eq: 80} guarantees the completeness of $\tilde{g}$. 

    Consider the cylinders $C_{R_i}$ with $R_i\to\infty$. For $R_i,t_0>r_2$, denote $C_{R_i,t_0}$ to be the intersection of $C_{R_i}$ with the region bounded between two hyperplanes $S_{\pm t_0}$. It follows from \cite{SY79} that the mean curvatures of $S_{\pm t}$ respect to $\pm\frac{\partial}{\partial t}$ directions are positive for all $t\ge t_0$. Furthermore, \eqref{eq: 50} ensures that the dihedral angle between $\partial C_{R_i}$ and $S_{\pm t}$ is $\frac{\pi}{2}$. Denote $\partial E\subset V_1\subset V_2\subset\dots$ to be an compact exhaustion of $M\setminus E$,  we can seek for an  area-minimizing boundary $\Sigma_{R_i}$ in $(C_{R_i,t_0}\cup V_i,\tilde{g})$ with inner obstacle by the compactness theorem of Caccioppoli sets \cite[Theorem 12.26]{Mag12} (See Figure \ref{f5}). In fact, $\Sigma_{R_i}$ 
    minimizes the area in $\mathcal{F}_{R_i}$. Since $\Sigma_{R_i}\cap K\neq \emptyset$ for some fixed compact $K$, we get the contradiction by Theorem \ref{thm: 8dim Schoen conj}.

    For the rigidity part, it follows directly from the inequality part and the proof of \cite[Theorem 1.2]{Zhu23}
\end{proof}

We end this section with two remarks.
\begin{remark}\label{remark: arbitrary ends}
    When $n\le 7$, we may modify the proof of Theorem \ref{thm: pmt8dim} to obtain an alternative proof for the positive mass theorem for AF manifold $(M^n,g)$ with arbitrary ends. Assume $R_g>0$ everywhere without loss of generality. First, we can find an area-minimizing boundary $\Sigma_1\subset M$ which is also strongly stable following the proof of Theorem \ref{thm: pmt8dim}. By Proposition \ref{prop: eq1-arbitrary end}, we can solve $u_1\in C^{\infty}(\Sigma_1)$ with $-\Delta_{\Sigma_1}u_1+\frac{1}{4}R_{\Sigma_1} = 0$ and $\lim_{x\in {\Sigma_1\cap E},|x|\to\infty}u_1 = 1$. Let
    \begin{align*}
        (\hat{\Sigma}_1, \hat{g}_1) = (\Sigma_1\times \mathbf{S}^1, u_1^2ds_1^2+g_\Sigma).
    \end{align*}
    Then $R_{\hat{g}_1}>0$ everywhere. In the $k$-th step, we begin with the $n$-dimensional ALF manifold with arbitrary ends
    \begin{align*}
        (\hat{\Sigma}_{k-1}, \hat{g}_{k-1}) = (\Sigma_{k-1}\times \mathbf{T}^{k-1}, u_1^2ds_1^2+\dots +u_{k-1}^2ds^2_{k-1}+g_{\Sigma_{k-1}})
    \end{align*}
    where $u_i\in C^{\infty}(\Sigma_{k-1})$ and $\Sigma_{k-1}\subset\Sigma_{k-2}\subset\dots\Sigma_1\subset M$ is a sequence of submanifolds. Next, using the method in the proof of Theorem \ref{thm: pmt8dim}, we can find an area-minimizing boundary $\bar{\Sigma}_k\subset \hat{\Sigma}_{k-1}$ that is also strongly stable. Following the argument of \cite[p.188]{GL83}, $\bar{\Sigma}_k$ is $\mathbf{T}^{k-1}$-symmetric, so we just regard it as $\Sigma_k\times \mathbf{T}^{k-1}$ with $\Sigma_k\subset\Sigma_{k-1}$. Then by applying Proposition \ref{prop: eq1-arbitrary end} and conducting the warped product process once again we obtain the $k$-th ALF manifold with arbitrary ends
    \begin{align*}
        (\hat{\Sigma}_{k}, \hat{g}_{k}) = (\Sigma_{k}\times \mathbf{T}^{k}, u_1^2ds_1^2+\dots +u_{k}^2ds^2_{k}+g_{\Sigma_{k}}).
    \end{align*}
    This reduction process continues until $\dim\Sigma_k\le 3$, at which point the positive mass theorem follows from the spin-theoretic results in \cite{Zei20}. A contradiction follows from the argument of the mass decay Lemma \ref{lem: mass decay ALF}. An interesting aspect of this proof is that it is based on the torical-symmetrization method for minimal hypersurfaces (see \cite{FcS1980},\cite{GL83}, and \cite{Gro18}), providing an alternative to the $\mu$-bubble approach.
    
\end{remark}

\begin{remark}
    The singular positive mass theorem developed in this paper is well applied in dimension 8, the case of isolated singularities. The current dimensional restriction arises from technical challenges in the analysis of singular harmonic functions (Section 2.4), particularly when dealing with non-isolated singular sets. If these analytical obstacles can be overcome, it would become possible  to extend Theorem \ref{thm: 8dim Schoen conj}, Theorem \ref{prop: rigidity for minimal surface} and Theorem \ref{thm: pmt8dim} to the case $n+1\ge 9$. Moreover, building upon recent advances in generic regularity theory \cite{CMS23},  \cite{CMS24} and \cite{CMSW2025}, we anticipate that Theorem \ref{thm: pmt8dim} may be further extensible to higher dimensions.
\end{remark}

\appendix
\section{Poincare-Sobolev inequality on area-minimizing boundaries}

In this appendix we want to verify a local Poincare-Sobolev inequality. Let $\mathcal{B}_r$ be a geodesic ball with radius $r$ in $(M^n,g)$ contained in a normal coordinates chart, $\beta$ and $\gamma$ be two positive constants depending only on $n$. Then we have
\begin{proposition}\label{prop: Pioncare-Sobolev inequality}
	Let $S$ be an oriented boundary of least area in  $\mathcal{B}_r$ with $\partial S \cap \mathcal{B}_r=\phi$. Then there exists $\beta,\gamma$ depending only on $(\mathcal{B}_r,g)$, such that for every $f\in C^1(\mathcal{B}_ r)$ we have
	$$
	\inf_{k\in \mathbf{R}}\{\int_{\mathcal{B}_{\beta r}}|f-k|^{\frac{n}{n-1}} d\|S\|\}^{\frac{n-1}{n}}
	\leq 2\gamma\int_{\mathcal{B}_r}|\nabla_S f|d\|S\|.
	$$
\end{proposition}

To achieve this, we need  the following isoperimetric inequality.

\begin{lemma}\label{lmm: isoperimetric ineq}
	Let $S$ be a minimal $n$-integral current with compact support in $(M^n,g)$	then one has the isoperimetric inequality
	$$
	M(S)^{\frac{n-1}{n}}\leq \gamma M(\partial S).
	$$
	Here $M(S)$ denotes the mass of $S$. 
\end{lemma} 
It suffices to verify  Lemma \ref{lmm: isoperimetric ineq}  for the case that $S$ is in $\mathcal{B}_r$. Once this is done,  Theorem 2 in \cite{BG1972}, and hence Proposition \ref{prop: Pioncare-Sobolev inequality},  can be established on $(M^n,g)$ by the exactly same arguments. 

\begin{proof}[Proof of Lemma \ref{lmm: isoperimetric ineq}]
	As $\mathcal{B}_r$ is contained in a normal coordinates of $(M^{n},g)$, we may assume there is Euclidean metric $g_0$ on $\mathcal{B}_r$, then there is a constant $\Lambda$ depends only on $(M^n,g)$ with 
	$$
	\Lambda^{-1} g_0\leq g\leq \Lambda g_0.
	$$
	For any $n$-integral current $G$ in $\mathcal{B}_r$, let $M_0(G)$ denote the mass of $G$	 w.r.t $g_0$, then by the definition of mass, we have
	\begin{equation}\label{eq: eq2}
		\Lambda^{-\frac{n}{2}} M_0(G)\leq  M(G)\leq \Lambda^{\frac{n}{2}} M_0(G).	
	\end{equation}
	Let $S_0$ be the least area $n$-integral current in $(\mathcal{B}_r,g_0)$ with $\partial S_0=\partial S$ (we may assume $\mathcal{B}_r$ is a convex ball with $g_0$ and $g$). Then by the isoperimetric inequality for minimal $n$-integral currents with compact support in $(\mathcal{B}_r,g_0)\hookrightarrow\mathbf{R}^n$, we have
	$$
	M_0(S_0)^{\frac{n-1}{n}}\leq \gamma_0 M_0(\partial S), ~\text{for some constant depends only on $n$}
	$$
	In conjunction with \eqref{eq: eq2}, we have
	$$
	M(S)\leq M(S_0)\leq \Lambda^{\frac{n}{2}} M_0(S_0)\leq \Lambda^{\frac{n}{2}} \gamma_0M_0(\partial S)\leq \Lambda^{n}\gamma_0 M(\partial S).
	$$
	Setting $\gamma=\Lambda^{n}\gamma_0$ gives the conclusion.
\end{proof}

Next, we establish a Sobolev inequality on $S$ by using Proposition \ref{prop: Pioncare-Sobolev inequality}.

\begin{proposition}\label{prop: Sobolev inequality on S}
    Let $S$ be an area-minimizing boundary in $(M,g)$. Let $\mathcal{U}$ be region with compact closure in $(M,g)$ such that $(\supp S)\backslash\mathcal{U}\ne\emptyset$.Then there exists $C = C(\mathcal{U},S,g)$, such that for any $f\in C^1_0(\mathcal{U})$, there holds
    \begin{align}\label{eq: 40}
        \int_{\mathcal{U}}|f|^{\frac{n}{n-1}} d\|S\|\le C\int_{\mathcal{U}}|\nabla_S f| d\|S\|.
    \end{align}
\end{proposition}
\begin{proof}
    We adopt the argument in \cite{SY79} by Scheon-Yau. Let $\mathcal{U}_1$ be a neighborhood of $\mathcal{U}$ with compact closure such that $(\supp S)\backslash\mathcal{U}_1\ne\emptyset$. For each $p\in\mathcal{U}_1$, there exist $r_p,\beta_p,\gamma_p>0$ that satisfy the property of Proposition \ref{prop: Pioncare-Sobolev inequality}. By finite covering theorem we can select $p_1,p_2,\dots,p_k\in \supp S$ such that
    $$\supp S\cap\mathcal{U}_1\subset \bigcup_{i=1}^k \mathcal{B}_{\beta_ir_i}(p_i).
    $$
    Assume \eqref{eq: 40} is not true, then we can select a sequence of $f_i\in C^1_0(\mathcal{U})$ with
    \begin{align}\label{eq: 41}
        \int_{\mathcal{U}}|\nabla_S f_j| d\|S\|<j^{-1} \mbox{ and }\int_{\mathcal{U}}|f_j|^{\frac{n}{n-1}} d\|S\| = 1.
    \end{align}
    Let $\tilde{f}_j\in Lip(\mathcal{U}_1)$ be the function that extends $f_j$ to $\mathcal{U}_1\setminus \mathcal{U}$ by zero. It follows from Proposition \ref{prop: Pioncare-Sobolev inequality} that
    \begin{align*}
        (\int_{\mathcal{B}_{\beta_i r_i}(p_i)}|\tilde{f}_j-k_{j,i}|^{\frac{n}{n-1}} d\|S\|)^{\frac{n-1}{n}}
	\leq 2\gamma_i\int_{\mathcal{B}_{r_i}(p_i)}|\nabla_S \tilde{f}_j|d\|S\|<Cj^{-1}.
    \end{align*}
    for some $f_{j,i}$. By \eqref{eq: 41} we see $k_{j,i}$ is uniformly bounded, so $k_{j,i}$ converges subsequentially to a constant $k_i$ as $j\to\infty$, and
    \begin{align*}
        \int_{\mathcal{B}_{\beta_i r_i}(p_i)}|\tilde{f}_j-k_i|^{\frac{n}{n-1}} d\|S\|\to 0\mbox{ as }j\to\infty.
    \end{align*}
    Since $\tilde{f}_j\equiv 0$ on $\mathcal{U}_1\backslash\mathcal{U}$, we conclude that $k_i = 0$ for all $i$. Consequently, we deduce that
    \begin{align*}
        \int_{\mathcal{U}}|f_j|^{\frac{n}{n-1}} d\|S\| \to 0
\mbox{ as }j\to\infty.
\end{align*}
This is not compatible with \eqref{eq: 41}.
\end{proof}

By a similar argument, we can also establish the following.

\begin{proposition}\label{prop: Sobolev inequality on S 2}
    Let $(M,g)$ be a complete manifold with an AF end $E$, and let $S$ be an area-minimizing boundary in $(M,g)$. Suppose $\mathcal{U}$ is a neighborhood of $E$ such that $\mathcal{U}\backslash E$ has compact closure in $M$. Then there exists $C = C(\mathcal{U},S,g)$, such that for any $f\in C^1_0(\mathcal{U})$, there holds
    \begin{align*}
        \int_{\mathcal{U}}|f|^{\frac{n}{n-1}} d\|S\|\le C\int_{\mathcal{U}}|\nabla_S f| d\|S\|.
    \end{align*}
\end{proposition}

\section{Degree of proper maps on noncompact manifolds and singular spaces}
In this appendix, we record some definitions of the degree of proper map between possibly noncompact manifolds and singular spaces. We begin with the following proposition.

\begin{proposition}\label{prop: equivalence degree}
    Let $X,Y$ be oriented $n$-manifolds and $f:Y\longrightarrow X$ be a smooth proper map. Then the degree of $f$ can be defined in one of the following three ways, and all these definitions are equivalent:
    
    (1) For a regular value $x\in X$ of $f$, count the number of elements of $f^{-1}(x)$ with respect to the orientation.

    (2) Consider the induced map of the locally finite homology
    \begin{align*}
        f_*: \mathbf{Z} = H_{n}^{lf}(Y)\longrightarrow H_{n}^{lf}(X) = \mathbf{Z}
    \end{align*}
    and set $f_*([Y]) = (\deg f)[X]$. Here $[X]$ denotes the fundamental class of $X$ in the locally finite homology.

    (3) Consider the induced map of the compactly supported cohomology
    \begin{align*}
        f^*: \mathbf{Z} = H_{c}^{n}(X)\longrightarrow H_{c}^{n}(Y) = \mathbf{Z}
    \end{align*}
    and set $f^*([Y]^*) = (\deg f)[X]^*$. Here $[X]^*$ denotes the fundamental class of $X$ in the compactly supported cohomology.

    (4) Consider the induced map of the compactly supported de Rham cohomology
    \begin{align*}
        f^*: \mathbf{R} = H_{dR,c}^{n}(X)\longrightarrow H_{dR,c}^{n}(Y) = \mathbf{R}
    \end{align*}
    In this case we have
    \begin{align*}
        \int_{Y}f^*\omega = (\deg f)\int_X\omega
    \end{align*}
    for any compactly supported differential $n$-form $\omega$ on $X$.
\end{proposition}

\begin{proof}
    (1)$\Longleftrightarrow$(2) follows from the excision lemma in combination with a local homology group argument. (2)$\Longleftrightarrow$(3) follows from the non-degeneracy of the bilinear form
    \begin{align*}
        H^n_c(X)\otimes H_{n}^{lf}(X)\longrightarrow\mathbf{Z}
    \end{align*}
    (3)$\Longleftrightarrow$(4) is a consequence of the universal coefficient theorem and the de Rham Theorem.
\end{proof}

As an application of Proposition \ref{prop: equivalence degree}, we can give the following definition for degrees of quasi-proper maps, which slightly extends \cite[Definition 1.6]{CCZ23}.

\begin{definition}\label{defn: non-zero degree}
            Let $X^n$ and $Y^n$ be oriented smooth $n$-manifolds, where $Y^n$ is not necessarily compact. Let $f: Y^n\longrightarrow X^n$ be quasi-proper (see \cite[Theorem 1.3]{CCZ23}) and satisfies
            \begin{itemize}
                \item $S_{\infty}$ consists of discrete points, where 
                \begin{align*}
                    S_{\infty} = \bigcap_{K\subset Y \text{ compact }} \overline{f(Y-K)}
                \end{align*}
            \end{itemize}
            Then it is direct to check that the restriction map $f|_{Y\backslash f^{-1}(S_{\infty})}: Y\backslash f^{-1}(S_{\infty})\longrightarrow X\backslash S_{\infty}$ is proper, and we say $f$ has degree $k$ if $f|_{Y\backslash f^{-1}(S_{\infty})}: Y\backslash f^{-1}(S_{\infty})\longrightarrow X\backslash S_{\infty}$ has degree $k$ as a proper map.
        \end{definition}

Next, we focus on the definition of non-zero degree map defined on singular spaces. We will introduce two different but compatible definitions. The first definition is directly linked to Definition \ref{defn: non-zero degree}, and the second one is more intrinsic.

\begin{definition}\label{defn: degree 2}
    Let $M$ be a compact topological space, such that $M\backslash \mathcal{S}$ is a smooth oriented $n$-manifold, where $\mathcal{S} = \{p_1,p_2,\dots,p_k\}$. Let $N^n$ be a compact oriented manifold. Let $\psi:M\longrightarrow N$ be a continuous map, then the restriction map $\psi|_{M\backslash\mathcal{S}}$ is quasi-proper. The degree of $\psi$ is defined as the degree of the quasi-proper map $\psi|_{M\backslash\mathcal{S}}:M\backslash\mathcal{S}\longrightarrow N$ as in Definition \ref{defn: non-zero degree}.
\end{definition}

Since $M$ may not have manifold structure, we aim to establish the concept of non-zero degree through purely topological considerations. This leads us to the following lemma:

\begin{lemma}
    Let $M,N,\psi$ be as in Definition \ref{defn: degree 2}. Assume the induced map
    \begin{align*}
        \psi_*:H_n(M)\longrightarrow H_n(N)\cong\mathbf{Z}
    \end{align*}
    is not identically zero, then $\psi$ has non-zero degree in the sense of Definition \ref{defn: degree 2}.
\end{lemma}
\begin{proof}
    For a locally compact space $X$ and an open set $U\subset X$, there is a restriction map $\iota: H_{*}^{lf}(X)\longrightarrow H^{lf}_n(U)$ (see for instance \cite[Section 6.1.3]{Morgan} for the definition). In the case that $X$ is an oriented $n$-manifold, $\iota$ is an isomorphism on $H_n^{lf}$. The naturality of $\iota$ implies that the following diagram is commutative
    \begin{equation*}
\xymatrix{&H_n(M)\ar[r]\ar[d]^{\psi_*}&H_{n}^{lf}(M\backslash \psi^{-1}(\psi(\mathcal{S})))\ar[d]^{\psi_*}\\
&H_n(N)\ar[r]^{\cong}&H_n^{lf}(N\backslash\mathcal{S}) \quad \quad }
\end{equation*}
where the horizontal arrows are given by $\iota$. From our assumption and the fact that the lower horizontal map is an isomorphism, we see that the vertical map on the right side is not identically zero. This shows that $\psi$ has non-zero degree in the sense of Definition \ref{defn: degree 2}.
\end{proof}

\begin{remark}
    If $n\ge 2$ and $M$ has almost manifold structure with isolated singulaities as considered in \cite{DaSW24b}, $i.e.$ there exists a compact oriented manifold $N$ with boundary, such that $M$ is the quotient space of $N$ obtained by pinching each component of $\partial N$ into a point, then $H_n(M) = \mathbf{Z}$.
\end{remark}

As a consequence, the following definition is consistent with Definition \ref{defn: degree 2}.
\begin{definition}\label{defn: degree 4}
    Let $M,N,\psi$ be as in Definition \ref{defn: degree 2}. Then $\psi$ is said to have non-zero degree, if the induced map
    \begin{align*}
        \psi_*:H_n(M)\longrightarrow H_n(N)\cong\mathbf{Z}
    \end{align*}
    is not identically zero.
\end{definition}

\section{Decay estimates for solutions to elliptic equations on AF ends}
In this appendix, we recall some results about the decay estimates for solutions to elliptic equations on AF ends.

\begin{lemma}\label{lem: asymptotic behavior1}
		Let $(E,g)$ be a $n$-dimensional AF end with asymptotic order $\tau\le n-2$. Let $f\in C^{2,\alpha}(E)$ for some $\alpha>0$. Suppose $f$ satisfies
		$$
		f(x)=O(r^{-\tau-2}) \ \quad \text{ for $x\rightarrow \infty$ on $E$},
		$$
		where $r:=|x|$, $u\in C^{4,\alpha}(E)$ satisfies
		$$
		-\Delta u=f \quad \text{in }\ \ E 
		$$
		with $u(x)\rightarrow 0$ for $x\rightarrow \infty$ on $E$. Then for each $\epsilon>0$, there holds
		
		$$
		|u|+r|\nabla u|+r^2|\nabla^2 u| = O(r^{-\tau+\epsilon}).
		$$
	\end{lemma}
	
	\begin{proof}
		Without loss of generality, we can assume $E$ is diffeomorphic to $\mathbf{R}^n$, then $u$ satisfies the equation
		\begin{align*}
			g^{ij}(D_{ij}u-\Gamma_{ij}^k D_ku) = 0 \ \text{ on}\quad \mathbf{R}^{n}.
		\end{align*}
		Since $\Gamma_{ij}^k = O(r^{-\tau-1})$, by Schauder estimate we know $u = O(r^{-1})$. Therefore, we may assume 
		$$
		-\bar{\Delta} u=\bar{f} \ \text{ on} \quad \mathbf{R}^{n}
		$$
		for some   $\bar{f}\in C^{2,\alpha}(\mathbf{R}^{n})$ with
        $f(x)=O(r^{-\tau-2})$ for $x\rightarrow \infty$ on $\mathbf{R}^{n}$,
        where $\bar{\Delta}$ denotes the standard Laplace operator on $\mathbf{R}^{n}$ .
        The conclusion then follows from the Claim in \cite[p.327]{Lee19}.
	\end{proof}

\begin{lemma}\label{lem: asymptotic behavior2}
    Let $(E,g)$ be a $n$-dimensional AF end with asymptotic order $\frac{n-2}{2}<\tau\le n-2$. $u\in C^{4,\alpha}(E)$ satisfies
		\begin{equation}\label{eq: harmonic function}
		-\Delta u=0 \quad \text{in}\ \ \  E
		\end{equation}
    with $u(x)\to 0$ as $|x|\to\infty$. Then there is a constant $a$, such that for any $\varepsilon'>0$ there holds
    \begin{align*}
        u(x) = ar^{2-n}+O(r^{1-n})+O(r^{-\tau+2-n+\varepsilon'}),
    \end{align*}
    where $r=|x|$.
\end{lemma}

\begin{proof}
    By Lemma \ref{lem: asymptotic behavior1}, 
    for any $\epsilon>0$,
    we have $|u(x)|+r|\nabla u|+r^2
    |\nabla^2 u|=O(r^{-\tau+\varepsilon})$. Rewrite \eqref{eq: harmonic function} as
    \begin{equation}\label{eq: possion eq}
        \bar{\Delta}u=f\ \ \text{on}
        \ \ \mathbf{R}^n\setminus B^n_R(O),
    \end{equation}
    where $\bar{\Delta}$ is the Laplacian in 
    $\mathbf{R}^n$. Then we have $f=O(r^{-2-2\tau+\varepsilon})$.
By standard PDE theory,  given any $\varepsilon'>0$, we can find a solution $v$ of the problem
\begin{equation}\label{eq: laplacian of v}
        \bar{\Delta}v=f\ \ \text{on}
        \ \ \mathbf{R}^n\setminus B^n_R(O),
    \end{equation}
satisfying
\begin{equation}\label{eq: asypmtotic expansion1}
 |v|+r|\partial v|+r^2|\partial ^2 v|=O(r^{-2\tau+\varepsilon+\varepsilon'}).   
\end{equation}
Now $w =u-v$ is a harmonic function defined in $\mathbf{R}^n\setminus B^n_R(O)$ with $w(x)\rightarrow 0$ as $x\rightarrow\infty$.
Thus
\begin{equation}\label{eq: asypmtotic expansion2}
    w(x)=ar^{2-n}+O(r^{1-n}).
\end{equation} 
Note that $\tau>\frac{n-2}{2}$. By \eqref{eq: asypmtotic expansion1}
and \eqref{eq: asypmtotic expansion2} we can improve the term $f$ in \eqref{eq: possion eq}
by $f=O(r^{-\tau-n+\varepsilon+\varepsilon'})$. Repeating the argument before, we find a solution $v$ of \eqref{eq: laplacian of v} with
\begin{equation}\label{eq: asypmtotic expansion3}
  |v|+r|\partial v|+r^2|\partial ^2 v|=O(r^{-\tau+2-n+\varepsilon'}).  
\end{equation}
for any $\varepsilon'>0$.
Combining \eqref{eq: asypmtotic expansion2}
and \eqref{eq: asypmtotic expansion3}
gives the desired estimate.

\end{proof}
\bibliographystyle{alpha}

\bibliography{Positive}

\begin{thebibliography}{DWWW24}

\bibitem[AGS14]{AGS14}
Luigi Ambrosio, Nicola Gigli, and Giuseppe Savar\'{e}.
\newblock Metric measure spaces with riemannian ricci curvature bounded from
  below.
\newblock {\em Duke Math. J.}, 163:1405--1490, 2014.

\bibitem[AX24]{AX24}
Gioacchino Antonelli and Kai Xu.
\newblock New spectral bishop-gromov and bonnet-myers theorems and applications
  to isoperimetry, 2024.

\bibitem[Bar86]{Bar86}
Robert Bartnik.
\newblock The mass of an asymptotically flat manifold,.
\newblock {\em Comm. Pure Appl. Math}, 39(5):661--693, 1986.

\bibitem[BBN10]{BBN10}
Hubert Bray, Simon Brendle, , and Andr{\'e} Neves.
\newblock Rigidity of area-minimizing two-spheres in three-manifolds.
\newblock {\em Comm. Anal. Geom.}, 18:821--830, 2010.

\bibitem[BG72]{BG1972}
E.~Bombieri and E.~Giusti.
\newblock Harnack's inequality for elliptic differential equations on minimal
  surfaces.
\newblock {\em Invent. Math.}, 15:24--46, 1972.

\bibitem[Car16]{Carlotto16}
Alessandro Carlotto.
\newblock Rigidity of stable minimal hypersurfaces in asymptotically flat
  spaces.
\newblock {\em Calc. Var. Partial Differential Equations}, 55(3):Art. 54, 20,
  2016.

\bibitem[CCE16]{CCE16}
Alessandro Carlotto, Otis Chodosh, and Michael Eichmair.
\newblock Effective versions of the positive mass theorem.
\newblock {\em Invent. Math.}, 206(3):975--1016, 2016.

\bibitem[CCZ23]{CCZ23}
Shuli Chen, Jianchun Chu, and Jintian Zhu.
\newblock Positive scalar curvature metrics and aspherical summands, 2023.

\bibitem[CFZ24]{CFZ24}
Simone Cecchini, Georg Frenck, and Rudolf Zeidler.
\newblock Positive scalar curvature with point singularities, 2024.

\bibitem[CG71]{CG71}
Jeff Cheeger and Detlef Gromoll.
\newblock The splitting theorem for manifolds of nonnegative ricci curvature.
\newblock {\em J. Diﬀerential Geometry}, 6:119--128, 1971.

\bibitem[CG00]{CG00}
Mingliang Cai and Gregory Galloway.
\newblock Rigidity of area minimizing tori in 3-manifolds of nonnegative scalar
  curvature.
\newblock {\em Comm. Anal. Geom.}, 8:565--573, 2000.

\bibitem[CL23]{CL23}
Otis Chodosh and Chao Li.
\newblock Stable anisotropic minimal hypersurfaces in $\mathbf {R}^{4}$.
\newblock {\em Forum of Mathematics, Pi.}, 11:e3, 2023.

\bibitem[CL24a]{CL2024}
Otis Chodosh and Chao Li.
\newblock Generalized soap bubbles and the topology of manifolds with positive
  scalar curvature.
\newblock {\em Ann. of Math. (2)}, 199(2):707--740, 2024.

\bibitem[CL24b]{CL24b}
Otis Chodosh and Chao Li.
\newblock Stable minimal hypersurfaces in $\mathbf{R}^4$.
\newblock {\em Acta Math.}, 233(1):1--31, 2024.

\bibitem[CLSZ21]{CLSZ2021}
Jie Chen, Peng Liu, Yuguang Shi, and Jintian Zhu.
\newblock Incompressible hypersurface, positive scalar curvature and positive
  mass theorem, 2021.

\bibitem[CLZ22]{CLZ22}
Jianchun Chu, Man~Chun Lee, and Jintian Zhu.
\newblock Singular positive mass theorem with arbitrary ends, 2022.

\bibitem[CLZ24]{CLZ24}
Jianchun Chu, Man~Chun Lee, and Jintian Zhu.
\newblock Homological $n$-systole in $(n+1)$-manifolds and bi-ricci curvature,
  2024.

\bibitem[CM11]{CM11}
Tobias~Holck Colding and William~P. Minicozzi, II.
\newblock {\em A course in minimal surfaces}, volume 121 of {\em Graduate
  Studies in Mathematics}.
\newblock American Mathematical Society, Providence, RI, 2011.

\bibitem[CMS23]{CMS23}
Otis Chodosh, Christos Mantoulidis, and Felix Schulze.
\newblock Generic regularity for minimizing hypersurfaces in dimensions 9 and
  10, 2023.

\bibitem[CMS24]{CMS24}
Otis Chodosh, Christos Mantoulidis, and Felix Schulze.
\newblock Improved generic regularity of codimension-1 minimizing integral
  currents.
\newblock {\em Ars Inveniendi Analytica}, may 10 2024.

\bibitem[CMSW25]{CMSW2025}
Otis Chodosh, Christos Mantoulidis, Felix Schulze, and Zhihan Wang.
\newblock Generic regularity for minimizing hypersurfaces in dimension 11,
  2025.

\bibitem[Dai04]{Dai04}
Xianzhe Dai.
\newblock A positive mass theorem for spaces with asymptotic susy
  compactification.
\newblock {\em Commun. Math. Phys.}, 244(2):335--345, 2004.

\bibitem[DSW24a]{DaSW2024}
Xianzhe Dai, Yukai Sun, and Changliang Wang.
\newblock The positive mass theorem for asymptotically flat manifolds with
  isolated conical singularities.
\newblock {\em Science China Mathematics}, 2024.

\bibitem[DSW24b]{DaSW24b}
Xianzhe Dai, Yukai Sun, and Changliang Wang.
\newblock Positive scalar curvature and isolated conical singularity, 2024.

\bibitem[DWWW24]{DWWW2024}
Xianzhe Dai, Changliang Wang, Lihe Wang, and Guofang Wei.
\newblock Singular metrics with nonnegative scalar curvature and rcd, 2024.

\bibitem[EK23]{EK23}
Michael Eichmair and Thomas Koerber.
\newblock Schoen's conjecture for limits of isoperimetric surfaces, 2023.

\bibitem[EK24]{EK24}
Michael Eichmair and Thomas Koerber.
\newblock Foliations of asymptotically flat manifolds by stable constant mean
  curvature spheres.
\newblock {\em J. Differential Geom.}, 128(3):1037--1083, 2024.

\bibitem[EM13]{EM13}
Michael Eichmair and Jan Metzger.
\newblock Unique isoperimetric foliations of asymptotically flat manifolds in
  all dimensions.
\newblock {\em Invent. Math.}, 194:591--630, 2013.

\bibitem[FCS80]{FcS1980}
Doris Fischer-Colbrie and Richard Schoen.
\newblock The structure of complete stable minimal surfaces in {$3$}-manifolds
  of nonnegative scalar curvature.
\newblock {\em Comm. Pure Appl. Math.}, 33(2):199--211, 1980.

\bibitem[Giu84]{Giu1984}
Enrico Giusti.
\newblock {\em Minimal surfaces and functions of bounded variation}, volume~80
  of {\em Monographs in Mathematics}.
\newblock Birkh\"auser Verlag, Basel, 1984.

\bibitem[GL83]{GL83}
Mikhael Gromov and H.~Blaine Lawson.
\newblock Positive scalar curvature and the dirac operator on complete
  riemannian manifolds.
\newblock {\em Inst. Hautes \'{E}tudes Sci. Publ. Math.}, 58:83--196, 1983.

\bibitem[Gro18]{Gro18}
Misha Gromov.
\newblock Metric inequalities with scalar curvature.
\newblock {\em Geom. Funct. Anal}, 28(3):645--726, 2018.

\bibitem[Gro20]{Gro20}
Misha Gromov.
\newblock No metrics with positive scalar curvatures on aspherical 5-manifolds,
  2020.

\bibitem[Gro23]{Gro23}
Misha Gromov.
\newblock {\em Four lectures on scalar curvature}.
\newblock World Sci. Publ., Hackensack, NJ, 2023.

\bibitem[Gro24]{Gr2024}
Misha Gromov.
\newblock Product inequalities for {$\Bbb T^\rtimes$}-stabilized scalar
  curvature.
\newblock {\em SIGMA Symmetry Integrability Geom. Methods Appl.}, 20:Paper No.
  038, 25, 2024.

\bibitem[GT83]{GT}
David Gilbarg and Neil~S. Trudinger.
\newblock {\em Elliptic partial differential equations of second order}, volume
  224 of {\em Grundlehren der mathematischen Wissenschaften [Fundamental
  Principles of Mathematical Sciences]}.
\newblock Springer-Verlag, Berlin, second edition, 1983.

\bibitem[Guo24]{guo2024}
Yifan Guo.
\newblock Green's functions on minimal submanifolds.
\newblock 2024.

\bibitem[HK00]{HK00}
Piotr Haj{\l}asz and Pekka Koskela.
\newblock Sobolev met poincare.
\newblock {\em Mem. Amer. Math. Soc.}, 145(688), 2000.

\bibitem[Hon18]{Hon18}
Shouhei Honda.
\newblock Bakry-\'{e}mery conditions on almost smooth metric measure spaces.
\newblock {\em Anal. Geom. Metr. Spaces}, 6:129--145, 2018.

\bibitem[HSY25]{HSY24}
Shihang He, Yuguang Shi, and Haobin Yu.
\newblock Foliation of area minimizing hypersurfaces in asymptotically flat
  manifolds and {S}choen's conjecture.
\newblock {\em Calc. Var. Partial Differential Equations}, 64(2):Paper No. 48,
  2025.

\bibitem[JSZ22]{JSZ22}
Wenshuai Jiang, Weimin Sheng, and Huaiyu Zhang.
\newblock Removable singularity of positive mass theorem with continuous
  metrics.
\newblock {\em Math. Z}, 302:839--874, 2022.

\bibitem[Kaz82]{Kazdan82}
Jerry~L. Kazdan.
\newblock Deformation to positive scalar curvature on complete manifolds.
\newblock {\em Math. Ann.}, 261(2):227--234, 1982.

\bibitem[Kaz24]{Kaz24}
Demetre Kazaras.
\newblock Desingularizing positive scalar curvature 4-manifolds.
\newblock {\em Math. Ann.}, 390:4951--4972, 2024.

\bibitem[Lee19]{Lee19}
Dan~A. Lee.
\newblock {\em Geometric Relativity}.
\newblock Graduate Studies in Mathematics. American Mathematical Society,
  Providence, RI, 2019.

\bibitem[Li04]{Li04}
Peter Li.
\newblock {\em Lectures on harmonic functions}.
\newblock 2004.

\bibitem[LM89]{LM89}
H.~Blaine Lawson and Marie-Louise Michelsohn.
\newblock {\em Spin geometry}, volume~38 of {\em Princeton Mathematical
  Series}.
\newblock Princeton University Press, Princeton, 1989.

\bibitem[LM19]{LM2019}
Chao Li and Christos Mantoulidis.
\newblock Positive scalar curvature with skeleton singularities.
\newblock {\em Math. Ann.}, 374(1-2):99--131, 2019.

\bibitem[LM23]{LM2023}
Chao Li and Christos Mantoulidis.
\newblock Metrics with {$\lambda _1(-\Delta + k R) \ge 0$} and flexibility in
  the {R}iemannian {P}enrose inequality.
\newblock {\em Comm. Math. Phys.}, 401(2):1831--1877, 2023.

\bibitem[LSWZ24]{LSWZ24}
Yihan Li, Guangxiang Su, Xiangsheng Wang, and Weiping Zhang.
\newblock Llarull's theorem on odd dimensional manifolds: the noncompact case,
  2024.

\bibitem[LUY24]{LUY21}
Martin Lesourd, Ryan Unger, and Shing-Tung Yau.
\newblock The positive mass theorem with arbitrary ends.
\newblock {\em J. Differential Geom.}, 2024.

\bibitem[Mag12]{Mag12}
Francesco Maggi.
\newblock {\em Sets of finite perimeter and geometric variational problems, An
  introduction to geometric measure theory}.
\newblock Cambridge Studies in Advanced Mathematics, 2012.

\bibitem[Min09]{Min09}
Vincent Minerbe.
\newblock A mass for alf manifolds.
\newblock {\em Commun. Math. Phys.}, 289(3):925--955, 2009.

\bibitem[Mor]{Morgan}
John Morgan.
\newblock Homotopy theory lecture notes
  (https://scgp.stonybrook.edu/archives/27538).

\bibitem[Sch89]{Schoen1989}
Richard~M. Schoen.
\newblock Variational theory for the total scalar curvature functional for
  {R}iemannian metrics and related topics.
\newblock 1365:120--154, 1989.

\bibitem[Sim83]{Simon83}
Leon Simon.
\newblock {\em Lectures on geometric measure theory}, volume~3 of {\em
  Proceedings of the Centre for Mathematical Analysis, Australian National
  University}.
\newblock Australian National University, Centre for Mathematical Analysis,
  Canberra, 1983.

\bibitem[Sim85]{Simon84}
Leon Simon.
\newblock Isolated singularities of extrema of geometric variational problems.
\newblock In {\em Harmonic mappings and minimal immersions ({M}ontecatini,
  1984)}, volume 1161 of {\em Lecture Notes in Math.}, pages 206--277.
  Springer, Berlin, 1985.

\bibitem[Sma93]{Smale93}
Nathan Smale.
\newblock Generic regularity of homologically area minimizing hypersurfaces in
  eight dimensional manifolds.
\newblock {\em Comm. Anal. Geom.}, 1(2):217–228, 1993.

\bibitem[ST18]{ST2018}
Yuguang Shi and Luen-Fai Tam.
\newblock Scalar curvature and singular metrics.
\newblock {\em Pacific J. Math.}, 293(2):427--470, 2018.

\bibitem[SY79a]{SY79}
Richard Schoen and Shing~Tung Yau.
\newblock On the proof of the positive mass conjecture in general relativity.
\newblock {\em Comm. Math. Phys.}, 65(1):45--76, 1979.

\bibitem[SY79b]{SY79b}
Richard Schoen and Shing-Tung Yau.
\newblock On the structure of manifolds with positive scalar curvature.
\newblock {\em Manuscripta Math.}, 28(1-3):159--183, 1979.

\bibitem[SY81]{SY81}
Richard Schoen and Shing~Tung Yau.
\newblock Proof of the positive mass theorem. {II}.
\newblock {\em Comm. Math. Phys.}, 79(2):231--260, 1981.

\bibitem[SY22]{ScY2019}
Richard Schoen and Shing-Tung Yau.
\newblock Positive scalar curvature and minimal hypersurface singularities.
\newblock In {\em Surveys in differential geometry 2019. {D}ifferential
  geometry, {C}alabi-{Y}au theory, and general relativity. {P}art 2}, volume~24
  of {\em Surv. Differ. Geom.}, pages 441--480. Int. Press, Boston, MA, [2022]
  \copyright 2022.

\bibitem[Wan24]{Wang24}
Zhihan Wang.
\newblock Mean convex smoothing of mean convex cones.
\newblock {\em Geom. Funct. Anal.}, 34:263--301, 2024.

\bibitem[WX24]{WX2024}
Jinmin Wang and Zhizhang Xie.
\newblock Scalar curvature rigidity of spheres with subsets removed and
  $l^\infty$ metrics, 2024.

\bibitem[WY23]{WY2023}
Tongrui Wang and Xuan Yao.
\newblock Generalized $s^1$-stability theorem, 2023.

\bibitem[Zei20]{Zei20}
Rudolf Zeidler.
\newblock Width, largeness and index theory.
\newblock {\em SIGMA 16}, 2020.

\bibitem[Zhu21]{Zhu21}
Jintian Zhu.
\newblock Width estimate and doubly warped product.
\newblock {\em Trans. Amer. Math. Soc}, 374(2):1497--1511, 2021.

\bibitem[Zhu23]{Zhu23}
Jintian Zhu.
\newblock Positive mass theorem with arbitrary ends and its application.
\newblock {\em Int. Math. Res. Not. IMRN}, 2023(11):9880--9900, 2023.

\bibitem[ZZ00]{ZZ00}
Liqun Zhang and Xiao Zhang.
\newblock Remarks on positive mass theorem.
\newblock {\em Commun. Math. Phys.}, 208:663--669, 2000.

\end{thebibliography}

\end{document}